\documentclass[11pt,a4paper]{article}
\usepackage[utf8]{inputenc}
\usepackage[T1]{fontenc}
\usepackage{lmodern}
\usepackage{amsmath,amssymb,amsthm}
\usepackage[margin=2.5cm]{geometry}
\usepackage{microtype}
\usepackage{enumitem}
\usepackage{booktabs}
\usepackage{array}
\usepackage{xcolor}
\usepackage{url}
\usepackage{cite}
\usepackage[colorlinks=true,linkcolor=blue,citecolor=blue,urlcolor=blue]{hyperref}
\hypersetup{pdftitle={Resolvent Components of Completely Normal Elements and the Morgan--Mullen Conjecture},pdfauthor={Florian Hellwig}}

\newcommand{\F}{\mathbb{F}}
\newcommand{\Z}{\mathbb{Z}}
\newcommand{\ord}{\operatorname{ord}}
\newcommand{\lcm}{\operatorname{lcm}}
\newcommand{\rad}{\operatorname{rad}}
\newcommand{\CN}{\mathrm{CN}}
\newcommand{\PCN}{\mathrm{PCN}}
\newcommand{\Supp}{\operatorname{Supp}}
\newcommand{\cC}{\mathcal{C}}
\newcommand{\cE}{\mathcal{E}}
\newcommand{\Wcl}{W_{\mathrm{cl}}}
\newcommand{\Wcr}{W_{\mathrm{cr}}}

\newtheorem{theorem}{Theorem}[section]
\newtheorem{proposition}[theorem]{Proposition}
\newtheorem{lemma}[theorem]{Lemma}
\newtheorem{corollary}[theorem]{Corollary}
\newtheorem{conjecture}[theorem]{Conjecture}
\theoremstyle{definition}
\newtheorem{definition}[theorem]{Definition}
\theoremstyle{remark}
\newtheorem{remark}[theorem]{Remark}

\title{Resolvent Components of Completely Normal Elements and the Morgan--Mullen Conjecture}
\author{Florian Hellwig\\[2pt]
\small Institute of Mathematics, University of Zurich, Winterthurerstrasse 190, 8057 Zurich, Switzerland\\
\small \texttt{florian.hellwig@fhellwig.com}}
\date{October 2026}

\begin{document}
\maketitle

\begin{abstract}
\noindent
We study completely normal elements of $\F_{q^n}$ over $\F_q$, that is, elements that are normal over every
intermediate field, through partial Lagrange resolvents. In these coordinates, complete normality splits into
independent components, in a way closely related to Hachenberger's decomposition theory. Our main tool is a structure
theorem for the components; together with a finite computation, it shows that every component has weight (the sum
of the probabilities that its conditions fail) at most $3/4$, uniformly in $n$. Combined with the classical character-sum
method, this gives an explicit sufficient criterion for the existence of a primitive completely normal element. We
verify the criterion for all $n>q$ except $70$ pairs $(q,n)$, which we settle with explicit elements found by
computer. Since the case $n\le q$ is due to Garefalakis and Kapetanakis, this confirms the conjecture of Morgan and
Mullen (1996) that every $\F_{q^n}$ contains a primitive element that is completely normal over $\F_q$. The same
estimates give two further results. First, Hachenberger's conjectured lower bound $(q-1)^{n'}q^{(p^a-1)n'}$ for the
number of completely normal elements, where $p$ is the characteristic and $n=p^an'$ with $p\nmid n'$, holds for all
pairs $(q,n)$; it was known for regular pairs. Second, primitivity and complete normality are asymptotically
independent: as $q^n\to\infty$, the number of primitive completely normal elements is
$(1+o(1))\,\varphi(q^n-1)/(q^n-1)$ times the number of completely normal elements (for fixed $n$ this was known). Our
arguments for $n>q$ and for both further results, including every computer-assisted step, are formalised in Lean~4
with Mathlib; the formalisation takes the case $n\le q$ and an explicit bound of Rosser and Schoenfeld as axioms and
relies on Lean's compiler and interpreter for the large computations.
\end{abstract}



\section{Introduction}\label{sec:intro}

A basis of an extension $\F_{q^n}/\F_q$ of finite fields is normal if it consists of the conjugates of a single
element, which is then called normal; normal bases of finite fields have been studied since the nineteenth century. An
element is primitive if it generates the multiplicative group of $\F_{q^n}$. The \emph{primitive normal basis
theorem}, proved in general by Lenstra and Schoof \cite{lenstra1987primitive}, states that every extension
$\F_{q^n}/\F_q$ contains an element that is both primitive and normal. An element is completely normal if it is normal
not only over $\F_q$ but over every intermediate field, and by the \emph{complete normal basis theorem} of Blessenohl
and Johnsen \cite{blessenohl1986}, completely normal elements exist in every extension of finite fields. In 1996
Morgan and Mullen conjectured that primitivity and complete normality can always be combined: every extension
$\F_{q^n}/\F_q$ should contain a primitive completely normal element. Garefalakis and Kapetanakis
\cite{garefalakis2019jpaa} proved the conjecture for $n\le q$, and it has also been proved for several further classes
of extensions, in particular by Hachenberger. In this paper we prove it for $n>q$.

To state the conjecture precisely, let $q$ be a power of a prime $p$ and let $n\ge1$. Put $E=\F_{q^n}$ and $\sigma(x)=x^q$, and write $n=p^an'$ with
$p\nmid n'$.
An element $w\in E$ is \emph{normal} over $\F_{q^d}$ ($d\mid n$) if its conjugates $\sigma^{dt}(w)$, $0\le t<n/d$,
form a basis of $E$ over $\F_{q^d}$; it is \emph{completely normal} (CN) over $\F_q$ if it is normal over $\F_{q^d}$ for
every $d\mid n$, and it is \emph{primitive} if it generates the cyclic group $E^\times$. We say that $(q,n)$ has the
\emph{Morgan--Mullen property} if $E$ contains a primitive element that is completely normal over $\F_q$.
Further notation is collected in Section~\ref{subsec:notation}.

\begin{conjecture}[Morgan--Mullen \cite{morgan1996}]\label{conj:mm}
Every pair $(q,n)$ has the Morgan--Mullen property.
\end{conjecture}
Our approach to the case $n>q$ is to study completely normal elements through partial
Lagrange resolvents, which we use to bound the density of these elements from below; this bound is then
combined with the classical character-sum method and a finite computation. Together with the case $n\le q$, this
confirms Conjecture~\ref{conj:mm} (Theorem~\ref{thm:intro} below). The idea of the proof is outlined in
Section~\ref{subsec:idea}.

\subsection{Background}
That every extension of finite fields has a normal basis goes back to Eisenstein~\cite{eisenstein1850lehrsaetze},
\mbox{Sch\"onemann}~\cite{schoenemann1850} and Hensel~\cite{hensel1888}; for arbitrary finite Galois extensions the normal basis
theorem is due to Deuring \cite{deuring1932galois} (see also \cite{lang2002algebra}), and Noether
\cite{noether1932normalbasis} had given an argument for certain infinite ground fields (see \cite{blessenohl2007normal}).
The number of normal elements of $\F_{q^n}$ over $\F_q$ was determined by
Hensel \cite{hensel1888} and, in terms of linearised polynomials, by Ore \cite{ore1934contributions} (see
\cite[Theorem~3.73]{lidl1997finite}). Normal bases also play a role in
arithmetic in finite fields \cite{mullin1989optimal,menezes1993applications,gao1993normal,gao2013normal,gao2013complexity}.
For general background on finite fields we refer to
\cite{lidl1997finite,jungnickel1993finite,mullen2013handbook,hachenberger2020topics}.

Before the general proof of Lenstra and Schoof \cite{lenstra1987primitive}, the primitive normal basis theorem had
been proved by Carlitz for large $q^n$ \cite{carlitz1952primitive,carlitz1952problems} and by Davenport for prime
fields \cite{davenport1968bases}; later, Cohen and Huczynska gave a computer-free proof \cite{cohen2003pnbt}. Many refinements have been studied, for instance
with prescribed traces, in towers of extensions, in a strong form, for $k$-normal elements, or for generators of
cyclic $\F_q[x]$-modules
\cite{cohen1999prescribed,hachenberger1999towers,cohen2010strong,hachenberger2016cyclic,huczynska2013knormal,reis2018onenormal,kapetanakis2015thesis,kapetanakis2019variations,mazumder2026traces};
see \cite{cohen2013prescribed} for a survey of results on prescribed coefficients, and \cite{cohen1990primitive}
for primitive elements with prescribed trace.

Completely normal elements arose in Faith's study of completely basic fields \cite{faith1957}, where he also proved
their existence over infinite ground fields; an extension
is \emph{completely basic} if every normal element is completely normal. Blessenohl and Johnsen
proved the complete normal basis theorem for every finite Galois extension \cite{blessenohl1986} (part of their proof
was later simplified by Blessenohl \cite{blessenohl1990supplement});
Hachenberger \cite{hachenberger1994} gave a simpler proof for finite fields. Blessenohl and Johnsen characterised the
completely basic abelian extensions \cite{blessenohl1991}; for finite fields it follows that $\F_{q^n}/\F_q$ is
completely basic if and only if $r\nmid\ord_{(n/r)'}(q)$ for every prime $r\mid n$ (for $n$ a power of a prime $r\ne p$
this is one of the main results of Blessenohl \cite{blessenohl1990abelsche}). Hachenberger gave an alternative proof
of this criterion \cite[Theorem~15.7]{hachenberger1997book} (see also \cite[Theorem~5.4.18]{hachenberger2013handbook}).
Hachenberger developed a detailed theory of completely normal (``completely
free'') elements, based on the module structure of $E$
\cite{hachenberger1994,hachenberger1996explicit,hachenberger1996primepower,hachenberger1996survey,hachenberger1997book,hachenberger2001decomposition,hachenberger2013handbook,hachenberger2020topics};
we compare it with our approach in Section~\ref{sec:remarks}. Explicit constructions of completely normal elements
and polynomials are given in
\cite{hachenberger1996explicit,scheerhorn1996ima,scheerhorn1996dcc,chapman1997,aravena2020}, and completely normal
elements in extensions of number fields are studied, for example, in \cite{koo2013,das2026pv}.

Morgan and Mullen, after tabulating primitive normal polynomials \cite{morgan1994}, computed completely normal primitive
polynomials of minimal weight over $\F_p$ for all primes $p\le97$ and all $n$ with $p^n<10^{50}$ and stated
Conjecture~\ref{conj:mm} \cite{morgan1996}; see \cite[Section~13.11]{hachenberger2020topics} for an overview of results on the
conjecture. The Morgan--Mullen property is known to hold in the following cases.
\begin{itemize}[itemsep=1pt]
\item \emph{Completely basic extensions}: there every primitive normal element is completely normal, so the primitive
normal basis theorem \cite{lenstra1987primitive,cohen2003pnbt} applies.
\item \emph{Regular extensions} in Hachenberger's sense \cite[Definition~20.1]{hachenberger1997book} (see also
\cite[Definition~1.3]{hachenberger2001regular}), that is, the extensions with
$\gcd(n,\ord_r(q))=1$ for every prime $r\mid n'$ (equivalently, $\gcd(n,\ord_{\rad(n')}(q))=1$): by Hachenberger
\cite[Theorem~1.4]{hachenberger2001regular} under the additional condition $q\equiv1\pmod4$ when $q$ is odd and $n$ is even, and
in the remaining case in the preprint \cite[Theorem~1.1]{hachenberger2019exceptional} (both cases are stated in
\cite[Result~13.11.4]{hachenberger2020topics}). See also
\cite[Theorem~2]{hachenberger2010twopower} for $n=2^m$ with $q\equiv3\pmod4$ and $m\ge e+3$, where $2^e$ is the largest power
of $2$ dividing $q^2-1$, and Blessenohl \cite{blessenohl2005mm} for $n=2^m$ with
$q\equiv3\pmod4$ and $2^m\mid q^2-1$.
\item \emph{Degree or its $p$-free part bounded in terms of $q$}: by Hachenberger \cite[Theorem~3]{hachenberger2016asymptotic} if
$q\ge n^{7/2}$ and $n\ge7$, or $q\ge n^3$ and $n\ge37$; by Garefalakis and Kapetanakis if $n\le q$
\cite[Theorem~1.3]{garefalakis2019jpaa} and, more generally, if $n'<q$ \cite[Theorem~1.4]{garefalakis2019jpaa}; and by
the same authors if $n<q^{4/3}$ for odd $n$ and $n<q^{5/4}$ for even $n$ with $(q-1)\nmid n$
\cite[Theorem~1.4]{garefalakis2019dcc}, and, for sufficiently large $q$, if $q\le n\le q^2$ for odd $n$ and
$q\le n\le0.43\,q^2$ for even $n$ with $(q-1)\nmid n$ \cite[Theorem~1.3]{garefalakis2019dcc}.
\item \emph{Partially completely basic extensions} ($n=n_1n_2$ with coprime $n_1,n_2>1$ and $\F_{q^{n_2}}/\F_q$
completely basic) under further conditions, including those with $n_1\in\{2,3\}$: by Garefalakis and Kapetanakis
\cite[Theorem~7.2]{garefalakis2026cmb} (for $q=2$ their computation \cite[Section~7]{garefalakis2026cmb} covers
$n_1=3$ only with $n_2$ even, and not $n_1=2$).
\item \emph{Computations}: by Hachenberger and Hackenberg \cite[Section~1, Computational Results~1
and~2]{hachenberger2019computational}, for $n\le202$ and every $q$ (their computation combined with
\cite[Theorem~1.4]{garefalakis2019jpaa} and the regular case), and for every $n$ with $p^n<10^{80}$ over every prime
field $\F_p$ with $p<10^4$.
\end{itemize}
In the most recent of these works the conjecture is still described as open in general \cite{garefalakis2026cmb}.

\subsection{Results}\label{subsec:main}
\begin{theorem}\label{thm:intro}
For every prime power $q$ and every $n\ge1$, the field $\F_{q^n}$ contains a primitive element that is completely normal over
$\F_q$.
\end{theorem}
This is the statement of Conjecture~\ref{conj:mm}; the proof is completed in Section~\ref{sec:main}
(Theorem~\ref{thm:main}). For $n\le q$ it rests on the theorem of
Garefalakis and Kapetanakis \cite[Theorem~1.3]{garefalakis2019jpaa}, whose proof also uses a computer calculation, the
primitive normal basis theorem for completely basic extensions (in particular for $24$ of the $37$ pairs left by that
calculation) and the examples of Morgan and Mullen \cite{morgan1996} for the remaining $13$ pairs; apart from what
enters through this theorem, none of the partial results listed above is used.

We now describe the method. An element $w\in E$ is completely normal if and only if it
satisfies finitely many \emph{conditions} of the form
$w\notin V$, one for each intermediate field $\F_{q^d}$ and each distinct irreducible factor of $x^{n/d}-1$ over
$\F_{q^d}$, where $V$ is a subspace of $E$ determined by the field and the factor (Section~\ref{subsec:idea}). For a random element, a condition fails
with probability $q^{-\operatorname{codim}V}$, where $\operatorname{codim}V$ is the codimension of $V$ over $\F_q$; we also
call it the \emph{codimension} of the condition.
Only some of the conditions, called \emph{essential}, are needed: an element is completely normal as soon as it
satisfies all essential conditions (Theorem~\ref{thm:essential}). These fall into groups, called \emph{resolvent
components}, that are independent of one another for a random element; the \emph{weight} of a component is the sum
of the failure probabilities of its conditions. By this independence, the density of completely normal elements is
at least the product, over the components, of one minus the weight, as long as every weight is less than~$1$
(Theorem~\ref{thm:product}); so what has to be bounded is the weight of each component, not the total weight of all
conditions. Every condition of codimension~$1$ is essential and forms a
component by itself, which is passed with probability exactly $1-\frac1q$ (Theorem~\ref{thm:criterion} and
Lemma~\ref{lem:qorbits}(c)). The main new ingredient is a bound
$W(q)$, uniform in $n$, for the weight of every other component (Theorem~\ref{thm:introW} below), which rests on a
structure theorem for the components (Theorem~\ref{thm:struct}). With this bound, the classical character-sum method
works uniformly in $n$, for regular and non-regular pairs alike. Earlier applications of the method to general pairs
estimate the number of completely normal elements from below by a sum over the intermediate fields, which amounts to
a union bound. For fixed $q$, however, the total weight of all conditions is unbounded in $n$, and the general
results obtained in this way require $n$, or its $p$-free part $n'$, to be at most a small power of $q$
(Section~\ref{subsec:idea}). The relation to Hachenberger's decomposition
theory of completely normal elements is discussed in Section~\ref{sec:remarks}.

The same estimates give information on the number $\CN_q(n)$ of completely normal elements of $E$ over $\F_q$ and the
number $\PCN_q(n)$ of primitive ones among them. First, the weight bounds of Section~\ref{sec:struct} prove
Hachenberger's conjectured lower bound for $\CN_q(n)$ \cite[Conjecture~5.4.55]{hachenberger2013handbook}, together with
the equality case (Corollary~\ref{cor:hachenberger}); both were known for regular pairs
\cite[\mbox{Theorem}~6.1]{hachenberger1997book}.
\begin{theorem}\label{thm:introhach}
For every prime power $q$ and every $n\ge1$, the number $\CN_q(n)$ of completely normal elements of $\F_{q^n}$ over
$\F_q$ satisfies
\[
\CN_q(n)\ \ge\ (q-1)^{n'}q^{(p^a-1)n'},
\]
where $p$ is the characteristic and $n=p^an'$ with $p\nmid n'$; equality holds if and only if $n'\mid q-1$.
\end{theorem}
Second, the weight bounds, combined with the counting criterion of Section~\ref{sec:counting} and the bounds for the
number of essential conditions of Section~\ref{sec:M}, show that
primitivity and complete normality are asymptotically independent. Let $\theta=\varphi(q^n-1)/(q^n-1)$ denote the
proportion of primitive elements in $E^\times$.
\begin{theorem}\label{thm:intropcn}
The ratio $\PCN_q(n)/(\theta\,\CN_q(n))$ tends to $1$ as $q^n\to\infty$; that is, for every $\varepsilon>0$ there is an $X$ such that
$\bigl|\PCN_q(n)/(\theta\,\CN_q(n))-1\bigr|<\varepsilon$ for all pairs $(q,n)$ with $q^n\ge X$.
\end{theorem}
This is Corollary~\ref{cor:pcn}(b); part~(a) of that corollary gives an explicit bound for the relative error.
For fixed $n$ and $q\to\infty$ the convergence was known. In this case
$\CN_q(n)/q^n\to1$, as $\CN_q(n)/q^n\ge1-n/q$ by the union bound and Corollary~\ref{cor:Mn} (Remark~\ref{rem:pcnphi}), so the statement is
equivalent to $\PCN_q(n)/\varphi(q^n-1)\to1$, which Hachenberger proved \cite[Theorem~2]{hachenberger2016asymptotic}
(see also \cite[Section~2.7]{hachenberger2019computational}). Theorem~\ref{thm:intropcn} adds the uniformity in $n$.
Remark~\ref{rem:pcnphi} also gives $\PCN_q(n)/\varphi(q^n-1)\to1$ without computation whenever $n/q\to0$. For fixed
$q$, by contrast,
$\PCN_q(n)/\varphi(q^n-1)$ does not tend to $1$ as $n\to\infty$: by Corollary~\ref{cor:pcn} it is asymptotic to
$\CN_q(n)/q^n$, which is at most $1-\frac1q$, since an element of trace $0$ is not normal.

\subsection{Computer-assisted steps and formal verification}
The proofs of Proposition~\ref{prop:Msmall}, Theorem~\ref{thm:W} and
Propositions~\ref{prop:range}--\ref{prop:certificates} involve finite computations by computer. Theorems~\ref{thm:intro}--\ref{thm:introW},
Theorem~\ref{thm:main} and Corollaries~\ref{cor:hachenberger} and~\ref{cor:pcn} rest on these statements; no other
theorem, proposition, lemma or corollary of the paper depends on these computations.
The computations are listed in Section~\ref{sec:comp}, and the programs and data are publicly available (see the end
of Section~\ref{sec:comp}).
In addition, our proof of the case $n>q$ and of Corollaries~\ref{cor:hachenberger} and~\ref{cor:pcn}, including every
computer-assisted step, has been formalised in the proof assistant
Lean~4 \cite{demoura2021lean4} on top of the library Mathlib \cite{mathlib2020}. Besides Lean's foundations, the
formal proof assumes only two cited results, namely the case $n\le q$ \cite[Theorem~1.3]{garefalakis2019jpaa} and the
bound $\log\lcm(1,\dots,\lfloor x\rfloor)<1.03883\,x$ ($x>0$) of Rosser and Schoenfeld \cite[Theorem~12]{rosser1962approximate},
which are stated as axioms. The large finite computations are evaluated outside Lean's kernel, so for them one also
trusts Lean's compiler and interpreter, together with the natively compiled library code that the interpreter calls.
The formalisation is described in Appendix~\ref{app:lean}. For an overview of formal verification in mathematics, see \cite{hales2008notices,avigad2024formalturn}.

\subsection{Notation}\label{subsec:notation}
Throughout, $q$ is a power of the prime $p$, $E=\F_{q^n}$, $\sigma(x)=x^q$ is the Frobenius of $E/\F_q$, and
$n=p^an'$ with $p\nmid n'$. For $j\ge1$ we write $j'$ or $j_{p'}$ for the $p$-free part of $j$ (with parentheses
around compound expressions, as in $(\ord_K(q))_{p'}$). We call $j$ \emph{$p$-free} if $p\nmid j$, i.e.\ if $j=j'$.
We write $v_\ell$ for the $\ell$-adic valuation and $\rad(j)$ for the product of the primes dividing $j$. We fix
an algebraic closure $\overline{\F}_q$; for $p\nmid m$ we write $\mu_m\subseteq\overline{\F}_q^{\,\times}$ for the group of $m$-th
roots of unity, and for $e$ coprime to $p$ we write $\ord_e(q)$ for the multiplicative order of $q$ modulo $e$
($\ord_1(q)=1$); thus $\ord_e(q)$ is the degree over $\F_q$ of any element of order $e$, while $\ord\eta$ denotes the
order of an element $\eta$ of a group. The \emph{$q$-orbit} (or \emph{Frobenius orbit}) of $\eta\in\mu_{n'}$ is
$\{\eta,\eta^q,\eta^{q^2},\dots\}$; it has $\ord_{\ord\eta}(q)$ elements. Further, $\sigma_1(m)$, $\tau(m)$,
$\varphi(m)$, $\mu(m)$ and $\omega(m)$ denote the sum of divisors, the number of divisors, Euler's function, the M\"obius
function and the number of distinct prime factors of $m$, respectively. Moreover, $\Phi_k$ is the $k$-th cyclotomic
polynomial, $1_A$ is the indicator function of a set $A$, and $\log$ is the natural logarithm (also written $\ln$).
The numbers of completely normal and of primitive completely normal elements of $E$ over $\F_q$ are denoted by
$\CN_q(n)$ and $\PCN_q(n)$, and $\delta_n=\CN_q(n)/q^n$. Conditions (Definition~\ref{def:cond}) are denoted by $C$, $C'$, $C_0$,
components (Section~\ref{sec:product}) by $\cC$, and the weight of a component by $w_{\cC}$. We write $M$ for the number of essential conditions, $M_1$ for the number of
those of codimension $1$ and $S$ for the total weight of the others.
Like $\delta_n$, $\omega(q^n-1)$ and the number $L$ of Theorem~\ref{thm:counting}, these numbers depend on the pair $(q,n)$. We
keep the pair fixed and do not show this dependence, except in statements in which $n$ or $q$ varies; there we write
$M(q,n)$ and $S(q,n)$.

\subsection{Idea of the proof}\label{subsec:idea}
The proof follows
the classical character-sum method; the new work lies mainly in the lower bound for the density of completely normal
elements that this method requires.

\paragraph{The difficulty.} By this method
\cite{carlitz1952primitive,carlitz1952sums,davenport1968bases,lenstra1987primitive,cohen2003pnbt,hachenberger2001regular,garefalakis2019jpaa}, a
set of density $\delta$ in $E$ that is the complement of a union of $M$ subspaces contains a primitive element as soon
as, roughly, $\delta\,q^{n/2}\ge2^{\omega(q^n-1)+M}$ (cf.\ Theorem~\ref{thm:counting}(b), whose proof applies to any such
set). The completely normal elements form such a set: an element $w$ is normal over $\F_{q^d}$ if and only if it lies in none of
the subspaces $g(\sigma^d)E$, where $g$ runs through the irreducible factors of $x^{n/d}-1$ over $\F_{q^d}$
(Lemma~\ref{lem:classical}). We call the pairs $(d,g)$ \emph{conditions} (of \emph{level} $d$); a condition $C$ fails for
$w$ if $w$ lies in its subspace $V_C$, which for random $w$ happens with probability $q^{-\operatorname{codim}V_C}$, the
codimension being taken over $\F_q$. Two quantities therefore have to be controlled: the number $M$ of conditions (after
the reduction below, of essential conditions), which must stay below about $\frac n2\log_2q$ (see the paragraph \emph{The criterion}
below), and the density $\delta_n$ of completely normal elements, which must be bounded from below. The second is the
harder of the two. The natural lower bound is the union bound
$\delta_n\ge1-\sum_Cq^{-\operatorname{codim}V_C}$, and this is essentially how earlier approaches proceed: they
bound the number of completely normal elements from below by means of a sum over all intermediate fields, which works
when $n$ or $n'$ is at most a small power of $q$ \cite{hachenberger2016asymptotic,garefalakis2019jpaa,garefalakis2019dcc}, or
by means of a sum over an essential set of them, as in the computation of Hachenberger and Hackenberg for $n\le202$
\cite[Proposition~3.2.5]{hachenberger2019computational}. (Other approaches exploit
the module structure of $E$ for regular pairs \cite{hachenberger2001regular,hachenberger2010twopower,hachenberger2019exceptional},
or use translates, by elements of $\F_q$, of products of completely normal elements of subextensions of coprime
degrees \cite{garefalakis2026cmb}.) For
fixed $q$, however, the total weight $\sum_Cq^{-\operatorname{codim}V_C}$ is unbounded as $n$ varies, even after the
reduction to essential conditions described below. Already the conditions of level $1$ (all of which are
essential), one for each distinct irreducible
factor $g$ of $x^n-1$ over $\F_q$, have total weight $\sum_gq^{-\deg g}$. Since $1-x\ge4^{-x}$ for $0\le x\le\frac12$, the
density $\prod_g(1-q^{-\deg g})$ of normal elements is at least $4^{-\sum_gq^{-\deg g}}$, and this density tends to $0$
along suitable $n$ \cite{frandsen2000density}, so $\sum_gq^{-\deg g}\to\infty$ along these $n$. For $(q,n)=(2,4599)$,
for instance, the essential conditions have total
weight about $3.41$ (Section~\ref{sec:4599}), and the union bound gives no information.

\paragraph{Independent components.} The first idea is to choose coordinates on $E$ in which the conditions are as
independent as possible. If $p\nmid n$, these coordinates are the Lagrange resolvents
\[
F_\eta(w)=\sum_{t\in\Z/n}\eta^{-t}\sigma^t(w),\qquad\eta\in\mu_n
\]
(for $p\mid n$ one uses partial resolvents, Section~\ref{sec:coord}; in this paragraph we assume
$p\nmid n$). Since $F_{\eta^q}(w)=\eta^{-q}F_\eta(w)^q$, the
coordinates along a Frobenius orbit $\{\eta,\eta^q,\eta^{q^2},\dots\}$ of $\mu_n$ are determined by any one of them, and
if $w$ is uniformly distributed in $E$, the coordinates of different orbits are independent and uniformly distributed,
with $q^s$ possible coordinate vectors for an orbit of size $s$ (Lemma~\ref{lem:coord}). In these coordinates, every subspace $g(\sigma^d)E$ is the zero set of a single \emph{fibre sum}
$\sum_{\eta^d=\xi}F_\eta(w)$, where $\xi$ is a root of $g$ (Theorem~\ref{thm:criterion}); we call $\xi$, which is
defined up to $\xi\mapsto\xi^{q^d}$, the \emph{target} and $\{\eta:\eta^d=\xi\}$ the \emph{fibre} of the condition. The
codimension of the condition is $\lcm(d,\ord_e(q))$, where $e=\ord\xi$ and $\ord_e(q)$ is the degree of $\xi$ over
$\F_q$. Only the \emph{essential} conditions, that is, those whose level $d$ divides $\ord_e(q)$, are needed
(Theorem~\ref{thm:essential}); their codimension is $\ord_e(q)$. Two conditions whose fibres meet no common Frobenius
orbit read disjoint sets of coordinates and are therefore independent. We call the classes of the equivalence relation on
the essential conditions generated by ``the fibres meet a common Frobenius orbit'' the \emph{resolvent components} $\cC$
(Section~\ref{sec:product} and Lemma~\ref{lem:qorbits}(a); for $p\mid n$ the same holds with the partial resolvents and
fibres of Section~\ref{sec:coord}). Then
\[
\delta_n=\prod_{\cC}\Pr\bigl[w\text{ passes all conditions of }\cC\bigr]\ \ge\ \prod_{\cC}\bigl(1-w_{\cC}\bigr),\qquad
w_{\cC}=\sum_{C\in\cC}q^{-\operatorname{codim}V_C},
\]
the inequality holding when all $w_{\cC}<1$ (Theorem~\ref{thm:product}). Hence a union bound is needed only inside each
component: what has to be small is not the total weight but the \emph{weight} $w_{\cC}$ of every single component. For
example, for $(q,n)=(2,21)$ there are four components; they read Frobenius orbits of total size $1$, $2$, $9$ and $9$
and have weights $\frac12$, $\frac14$, $\frac{33}{64}$ and $\frac{33}{64}$ (so the total weight $\frac{57}{32}$ exceeds
$1$). Exactly $1$, $3$, $294$ and $294$ of their $2$, $4$, $2^9$ and $2^9$ possible coordinate vectors satisfy
all their conditions. Thus
$\CN_2(21)=1\cdot3\cdot294^2=259308$, i.e.\ $\delta_{21}\approx0.124$, while the product bound gives
$\delta_{21}\ge\frac12\cdot\frac34\cdot(\frac{31}{64})^2\approx0.088$
(Section~\ref{sec:remarks}). Theorems~\ref{thm:criterion} and~\ref{thm:product} are
closely related to Hachenberger's decomposition theory of completely normal elements, which already factorises
$\CN_q(n)$ into independent local factors \cite{hachenberger1997book,hachenberger2001decomposition} (see
\cite[Section~5.3]{hachenberger2019computational}); what the resolvent coordinates add is a description,
by fibre sums, of independent factors, which can be finer than his but can also be coarser (Section~\ref{sec:remarks}), and,
above all, the uniform bound for the weights of the components (Theorem~\ref{thm:introW} below). In the example
$(2,21)$ ours are finer, but Hachenberger's own computation \cite[Example~14.6]{hachenberger1997book} uses the same
splitting. Hachenberger and Hackenberg showed that complete normality needs to be tested only over an essential set
of intermediate fields \cite[Definitions~3.2.1 and~3.2.3, Proposition~3.2.4]{hachenberger2019computational};
Theorem~\ref{thm:essential} refines this reduction to individual targets (see Section~\ref{sec:remarks}).

\paragraph{Components are small.} It remains to bound the weight of a single component, uniformly in $n$; this
is the main new step (Section~\ref{sec:struct}). Starting from one fibre
element, we explore a component by two moves: $\eta\mapsto\eta^q$ (along a Frobenius orbit) and
$\eta\mapsto\eta\varepsilon$ with $\varepsilon\in\mu_{d'}$ (inside the fibre of a condition at level $d$,
which is a coset of $\mu_{d'}$). If $\eta$ lies
in the fibre of an essential condition with target $\xi$, then $e=\ord\xi$ divides $\ord\eta$, so the level $d$, which
divides $\ord_e(q)$, also divides $\ord_{\ord\eta}(q)$: an $\varepsilon$-move from $\eta$ only uses roots of unity whose
order divides $(\ord_{\ord\eta}(q))_{p'}$. Hence the moves cannot produce roots of unity of arbitrary order: both
preserve $\mu_J$ whenever $J$ is $p$-free and $(\ord_J(q))_{p'}\mid J$. The fibre of an essential condition of
codimension $o_0$ lies in $\mu_{K_0}$, where $K_0=o_0'(q^{o_0}-1)$, since its level divides $o_0$ and its target has order
dividing $q^{o_0}-1$. So if $o_0$ is the smallest codimension in the component, all fibre elements of its conditions
lie in $\mu_K$, where the \emph{closure} $K$ is the smallest $p$-free multiple of $K_0$ with $(\ord_K(q))_{p'}\mid K$,
obtained by iterating $J\mapsto\lcm(J,(\ord_J(q))_{p'})$ from $J=K_0$; it depends only on $q$ and $o_0$ (it is the
choice~(B) of $K$ in Section~\ref{sec:struct}). Consequently every level in the component divides $\ord_K(q)$, so both moves preserve the Frobenius orbit of
$\eta^D$, where $D=(\ord_K(q))_{p'}$, and the component has at most $\sigma_1(o)\gcd(q^o-1,D)$ conditions of
codimension $o$ (Theorem~\ref{thm:struct}). Although $K\ge q^{o_0}-1$, the number $D$ is small, as $q$ has order only $o_0$ modulo
the large factor $q^{o_0}-1$ of $K$; in fact $\log_2D=O((\log o_0)^2)$ (Lemma~\ref{lem:closuresize}). So the weight
of the component is at most $\sum_{o\ge o_0}\sigma_1(o)\,D\,q^{-o}$, independently of $n$, and this is small when
$o_0$ is large. (The cruder choice~(A) of Section~\ref{sec:struct}, $K=\lcm(K_0,\lcm(1,\dots,o_0)_{p'})$, gives
$D<o_0\,2.83^{o_0}$ by the
Rosser--Schoenfeld bound for $\lcm(1,\dots,o_0)$ \cite[Theorem~12]{rosser1962approximate}; this suffices for large $o_0$ if
$q\ge3$ (as $2.83<3$), but not if $q=2$, and is used for $q\ge4$ and for large $o_0$ when $q=3$.) For small $o_0$ the bounds
of Corollary~\ref{cor:weight} are evaluated by computer. For the four pairs $(q,o_0)=(2,3),(2,6),(3,2),(3,4)$, where
both bounds (from the closure~(B) and from the cruder choice~(A)) exceed the values
$W(q)$ below,
every component with minimal codimension~$o_0$ already occurs for a pair $(q,n_0)$ with $n_0$ dividing an explicit number depending only on $q$ and $o_0$
(Proposition~\ref{prop:finite}), so that the weights are bounded by a finite computation. The result is the
following theorem.
\begin{theorem}\label{thm:introW}
For every prime power $q$ and every $n\ge1$, every resolvent component for $(q,n)$ without a condition of
codimension~$1$ has weight at most $W(q)$, where
\[
W(2)=\tfrac34,\qquad W(3)=\tfrac49,\qquad W(q)\le0.34\quad(q\ge4);
\]
for $q\ge4$ the values $W(q)$ are given in Theorem~\ref{thm:W}.
\end{theorem}
For $q=2,3$ these bounds are attained (Section~\ref{sec:remarks}).

\paragraph{The criterion.} These bounds are now inserted into the character-sum method; put $W=W(q)$. Recall that
the essential conditions of codimension~$1$ are components by themselves, each passed with probability $1-\frac1q$
(Lemma~\ref{lem:qorbits}(c)), and that every other component has weight at most $W$. Since $1-x\ge(1-W)^{x/W}$ for
$0\le x\le W<1$ (by concavity of $\log(1-x)$), the product formula gives $\delta_n\ge(1-\frac1q)^{M_1}(1-W)^{S/W}$
(Theorem~\ref{thm:C}(a)), where, as in Section~\ref{subsec:notation}, $M_1$ is the number of essential conditions of
codimension~$1$ and $S$ is the total weight of all other essential conditions. In the rough inequality $\delta\,q^{n/2}\ge2^{\omega(q^n-1)+M}$ of the paragraph \emph{The difficulty}, $2^M$
is an upper bound for the product $L$ of the $M$ factors $2(1-q^{-\operatorname{codim}V_C})\le2$, one for each essential
condition $C$; the precise statement, Theorem~\ref{thm:counting}, contains $L$ itself. For the $M_1$ conditions of
codimension~$1$ this factor equals $2(1-\frac1q)$, so $(1-\frac1q)^{M_1}$ occurs on both sides (in the lower bound for $\delta_n$ and in $L$) and cancels, and we
obtain the following criterion, which is part~(b) of Theorem~\ref{thm:C}.
\begin{theorem}\label{thm:introC}
Let $0<W<1$ be such that every resolvent component without a condition of codimension~$1$ has weight at most $W$, and
put $c(W)=\log_2\!\frac1{1-W}\big/W$. Let $M$ be the number of essential conditions, and let $S$ be the total weight of
those of codimension~$\ge2$. If
\[
\log_2(q^n-1)-\tfrac n2\log_2q\ \ge\ \omega(q^n-1)+M+c(W)\,S, \tag{$*$}
\]
then $(q,n)$ has the Morgan--Mullen property. Upper bounds may be substituted for $\omega(q^n-1)$, $M$ and $S$.
\end{theorem}
By Theorem~\ref{thm:introW} one may take $W=W(q)$. The left-hand side of $(*)$ is about $\frac n2\log_2q$. On the right,
$\omega(q^n-1)\le\frac nk\log_2q+O_k(1)$ for every $k$ and $S(q,n)=O(\log n)$ (Lemma~\ref{lem:analytic}), so the decisive
term is $M$. Writing $M=\sum_{\eta\in\mu_{n'}}\rho(\eta)$ as a sum of local
contributions (Lemma~\ref{lem:Mrho}), each at most $p^a$ times the divisor sum $f(o)=\sum_{d\mid o}\min(\frac1d,\frac do)
\le1$ with $o=\ord_{\ord\eta}(q)$ (Proposition~\ref{prop:Mlocal}, Lemma~\ref{lem:divsum}), we obtain $M(q,n)\le n$ for all
pairs (Corollary~\ref{cor:Mn}); this suffices for $q\ge5$, where $\frac12\log_2q>1$. For $q=2,3,4$, where
$\frac12\log_2q$ is $0.5$, about $0.79$ and $1$, sharper bounds for the local contributions
(computed exactly for small $o$ when $q=2,3$) give $M(q,n)\le0.41\,n$, $0.62\,n$ and $0.75\,n$, respectively, for all
but finitely many $n$
(Proposition~\ref{prop:Msmall}).
Consequently $(*)$ holds for all $n$ beyond explicit thresholds, which are at most $171$ for $q\ge8$ and $14$ for
$q\ge32$ (Proposition~\ref{prop:range}). For the remaining pairs with $n>q$, we evaluate $(*)$ with the exact values of
$M$ and $S$ and the bound of Lemma~\ref{lem:omega} for $\omega(q^n-1)$ (Proposition~\ref{prop:finiteranges}). It holds
with at least one bit to spare except for $70$ pairs, all with $q\le19$ and $n\le378$; these are settled by explicit
primitive completely normal elements (Proposition~\ref{prop:certificates}). The case $n\le q$ is the theorem of
Garefalakis and Kapetanakis \cite[Theorem~1.3]{garefalakis2019jpaa}.

\paragraph{The corollaries.} Both corollaries rest on the lower bound
$\delta_n\ge(1-\frac1q)^{M_1}(1-W)^{S/W}=(1-\frac1q)^{M_1}2^{-c(W)S}$ of Theorem~\ref{thm:C}(a). For
Corollary~\ref{cor:hachenberger} it is compared, one root of unity at a time, with $(1-\frac1q)^{n'}$, which is
Hachenberger's bound $(q-1)^{n'}q^{(p^a-1)n'}$ divided by $q^n$. For this, the number $S$ splits, like
$M$, into local contributions, $S=\sum_{\eta\in\mu_{n'}}s(\eta)$ (Lemma~\ref{lem:Mrho}); the $M_1=\gcd(n',q-1)$
elements $\eta\in\mu_{n'}$ that lie in $\F_q$ have $s(\eta)=0$ and contribute $(1-\frac1q)^{M_1}$ on both sides, and
every other $\eta$ contributes $2^{-c(W)s(\eta)}\ge1-\frac1q$. (To show that equality holds only if $n'\mid q-1$, a short extra argument is needed for $q=2$, $n'=3$, where the
two bounds can coincide.) For Corollary~\ref{cor:pcn} the character-sum
argument bounds $|\PCN_q(n)-\theta\,\CN_q(n)|$ by $\theta\,2^{\omega(q^n-1)}q^{n/2}L$ (cf.\
Theorem~\ref{thm:counting}(a)); dividing by $\theta\,\CN_q(n)=\theta\,q^n\delta_n$ and inserting
$L\le2^M(1-\frac1q)^{M_1}$ and the lower bound for $\delta_n$ shows that the relative error is less than
$2^{-\Delta}$, where $\Delta$ is the margin of $(*)$, that is, its left-hand side minus its right-hand side. The estimates for
large $n$ above make $\Delta$ grow at least linearly in $n$, uniformly in $q$; for bounded $n$ and large $q$ a cruder
estimate suffices.

\paragraph{Organisation.} Section~\ref{sec:coord} introduces the coordinates, the fibre-sum criterion and the essential
conditions, and Section~\ref{sec:product} the components and the product formula. The results of
Sections~\ref{sec:counting}--\ref{sec:M} do not involve the components: Section~\ref{sec:counting} gives the counting
criterion, Section~\ref{sec:bounds} bounds $\omega(q^n-1)$ and gives $M$ and $S$ exactly, as sums over the
\emph{types} (level, order of the target) of the essential conditions, and Section~\ref{sec:M} bounds $M$. Section~\ref{sec:struct} proves the structure theorem and the weight
bounds. Section~\ref{sec:main} combines everything
and proves the main theorem, Hachenberger's lower bound for $\CN_q(n)$ (Corollary~\ref{cor:hachenberger}) and the
asymptotic independence of primitivity and complete normality (Corollary~\ref{cor:pcn}).
Section~\ref{sec:comp} lists the computations,
Section~\ref{sec:4599} works through the example $(q,n)=(2,4599)$ (it is not part of the proof and can
be read directly after Section~\ref{sec:main}), Section~\ref{sec:checks} lists further consistency
checks, and Section~\ref{sec:remarks} discusses the relation to earlier work, in particular to Hachenberger's
decomposition theory, and further questions. Appendix~\ref{app:lean} describes the Lean formalisation.

\section{Coordinates by partial resolvents}\label{sec:coord}

This section introduces coordinates on $E$ in which an element is normal over an intermediate field $\F_{q^d}$ if and
only if certain sums of its coordinates, one for each distinct irreducible factor of $x^{n/d}-1$ over $\F_{q^d}$, are
non-zero (Theorem~\ref{thm:criterion}). For random $w$, the coordinates of different pair orbits (defined below) are
independent (Lemma~\ref{lem:coord}(c)), and by Theorem~\ref{thm:essential} only the essential conditions are needed.

On a first reading one may assume $p\nmid n$, as in the paragraph \emph{Independent components} of
Section~\ref{subsec:idea}. Then $a=0$ and $n'=n$, and in Definition~\ref{def:cond}, where $d=p^cd'$ with $p\nmid d'$ and
$N=n'/d'$, every divisor $d$ of $n$ has $c=0$, $d'=d$ and $N=n/d$. The coordinates are the Lagrange resolvents $F_\eta(w)(0)=\sum_{t\in\Z/n}\eta^{-t}\sigma^t(w)$,
$\eta\in\mu_n$, and the pair orbits are just the Frobenius orbits ($q$-orbits) $\{\eta,\eta^q,\eta^{q^2},\dots\}$ of
$\mu_n$. A condition at level $d$ asks that the sum of the $F_\eta(w)(0)$ over all $\eta$ with $\eta^d=\xi$ be non-zero,
for a fixed $\xi\in\mu_{n/d}$, its target (defined up to $\xi\mapsto\xi^{q^d}$).

If $p\mid n$, the $n'$ full resolvents $\sum_{t\in\Z/n}\eta^{-t}\sigma^t(w)$, $\eta\in\mu_{n'}$, no longer determine $w$,
since they all vanish on $(\sigma^{n'}-1)E\ne0$. Splitting each of them according to $t\bmod p^a$ gives the partial
resolvents $F_\eta(w)(u)$, $u\in\Z/p^a$, defined below, which do (Lemma~\ref{lem:coord}(a)).

We use the notation of Section~\ref{subsec:notation}, and normality and CN as defined in
Section~\ref{sec:intro} before Conjecture~\ref{conj:mm}. Fix a finite field $\Omega\subseteq\overline{\F}_q$ containing
$E$ and $\mu_{n'}$ (e.g.\ the subfield of degree $\lcm(n,\ord_{n'}(q))$ over $\F_q$) and a generator $\zeta$ of $\mu_{n'}$.

\begin{lemma}[{cf.\ \cite[Theorem~2.39]{lidl1997finite}}]\label{lem:classical}
Let $Q$ be a power of $p$ and let $m\ge1$. Let $\sigma_Q(x)=x^Q$ on $\F_{Q^m}$, let $\Omega'\supseteq\F_{Q^m}$ be a field
containing $\mu_{m'}$, and for $w\in\F_{Q^m}$ and $\xi\in\mu_{m'}$ let
\[
R_\xi(w)=\sum_{s\in\Z/m}\xi^{-s}\sigma_Q^s(w)\in\Omega'
\]
be the \emph{(Lagrange) resolvent} of $w$ at $\xi$. Then $w$ is
normal over $\F_Q$ if and only if $R_\xi(w)\ne0$ for every $\xi\in\mu_{m'}$. Moreover $\{w:R_\xi(w)=0\}=g_\xi(\sigma_Q)\F_{Q^m}$,
where $g_\xi$ is the minimal polynomial of $\xi$ over $\F_Q$.
\end{lemma}
\begin{proof}
By the normal basis theorem, $\F_{Q^m}$ is a cyclic $\F_Q[x]$-module ($x$ acting as $\sigma_Q$) with annihilator
$x^m-1$, and the normal elements are its generators. Fix a generator $w_0$; every $w\in\F_{Q^m}$ is $f(\sigma_Q)w_0$
with $f\in\F_Q[x]$, and $w$ is a generator iff $\gcd(f,x^m-1)=1$. Since $x^m-1=(x^{m'}-1)^{m/m'}$ and
$x^{m'}-1$ is separable with root set $\mu_{m'}$, this happens iff $f(\xi)\ne0$ for all
$\xi\in\mu_{m'}$. From $R_\xi(\sigma_Qu)=\xi R_\xi(u)$ (the sum runs over a full period) and
$R_\xi(cu)=cR_\xi(u)$ for $c\in\F_Q$ we get $R_\xi(f(\sigma_Q)w_0)=f(\xi)R_\xi(w_0)$. Finally
$R_\xi(w_0)\ne0$: otherwise $R_\xi$ would vanish on $\F_{Q^m}$, but $R_\xi$ is a non-trivial $\Omega'$-linear
combination of the $m$ distinct field embeddings $\sigma_Q^s\colon\F_{Q^m}\to\Omega'$, which are linearly
independent by Dedekind's lemma. Hence $R_\xi(w)=0\iff f(\xi)=0\iff g_\xi\mid f\iff w\in g_\xi(\sigma_Q)\F_{Q^m}$.
\end{proof}

Lemma~\ref{lem:classical} treats one intermediate field at a time; the following coordinates serve all levels
$d\mid n$ at once.

\begin{definition}
For $w\in E$, $\eta\in\mu_{n'}$ and $u\in\Z/p^a$ let
\[
F_\eta(w)(u)\;=\;\sum_{t\in\Z/n,\ t\equiv u\ (\mathrm{mod}\ p^a)}\eta^{-t}\sigma^t(w)\ \in\Omega .
\]
We call the $F_\eta(w)(u)$ the \emph{partial (Lagrange) resolvents} of $w$; for $p\nmid n$ (so $u=0$) they are the
Lagrange resolvents of $E/\F_q$, and we also write $F_\eta(w)=F_\eta(w)(0)$.
The \emph{pair orbits} are the orbits of the permutation $(\eta,u)\mapsto(\eta^q,u+1)$ of
$\mu_{n'}\times\Z/p^a$.
\end{definition}

\begin{lemma}[cf.\ \cite{mattson1961new} for $p\nmid n$]\label{lem:coord}
\leavevmode
\begin{enumerate}[label=(\alph*),beginpenalty=10000]
\item We have $\sigma^t(w)=\frac1{n'}\sum_{\eta\in\mu_{n'}}\eta^tF_\eta(w)(t\bmod p^a)$ for all $t\in\Z/n$; in
particular the map $w\mapsto(F_\eta(w)(u))_{\eta,u}$ is injective.
\item We have $F_{\eta^{q^j}}(w)(u+j)=\eta^{-jq^j}F_\eta(w)(u)^{q^j}$ for all $j\ge0$.
\item Choose a representative $(\eta_\pi,u_\pi)$ of every pair orbit $\pi$. The map
$w\mapsto\bigl(F_{\eta_\pi}(w)(u_\pi)\bigr)_\pi$ is an $\F_q$-linear isomorphism of $E$ onto
$\prod_\pi X_\pi$, where $X_\pi\subseteq\Omega$ is an $\F_q$-subspace of dimension $|\pi|$. Consequently,
if $w$ is uniformly distributed in $E$, the coordinates of distinct pair orbits are independent and
uniformly distributed in the $X_\pi$.
\end{enumerate}
\end{lemma}
\begin{proof}
(a) We have $\sum_\eta\eta^tF_\eta(w)(t\bmod p^a)=\sum_{t'\equiv t\,\allowbreak (p^a)}\sigma^{t'}(w)\sum_\eta\eta^{t-t'}$, and
the inner sum is $n'$ if $t'\equiv t\pmod{n'}$ and $0$ otherwise; with $t'\equiv t\pmod{p^a}$ this forces
$t'=t$.

(b) Write $F_\eta(u)$ for $F_\eta(w)(u)$. For $j=1$,
\[
F_\eta(u)^q=\sum_{t\equiv u}\eta^{-tq}\sigma^{t+1}(w)
=\eta^q\sum_{t'\equiv u+1}(\eta^q)^{-t'}\sigma^{t'}(w)=\eta^qF_{\eta^q}(u+1);
\]
induct.

(c) By (b), every coordinate on $\pi$ is determined by the representative's, so the map is injective
by (a). With $s=|\pi|$ we have $(\eta_\pi^{q^s},u_\pi+s)=(\eta_\pi,u_\pi)$, and (b) with $j=s$ gives
$F_{\eta_\pi}(u_\pi)=\eta_\pi^{-s}F_{\eta_\pi}(u_\pi)^{q^s}$, so the
representative's coordinate lies in $X_\pi=\{x\in\Omega:x=\eta_\pi^{-s}x^{q^s}\}$, an $\F_q$-subspace
with at most $q^s$ elements. Thus $n=\dim E\le\sum_\pi\dim X_\pi\le\sum_\pi|\pi|=n'p^a=n$, and all
inequalities are equalities.
\end{proof}

Fix a divisor $d$ of $n$, write $d=p^cd'$ with $p\nmid d'$, and put $N=n'/d'$. Then $[E:\F_{q^d}]=n/d=p^{a-c}N$, the roots of
$x^{n/d}-1=(x^N-1)^{p^{a-c}}$ form the group $\mu_N$, which is generated by $\zeta^d$, and the exponents $ds$ of the
conjugates $\sigma^{ds}(w)$ of $w$ over $\F_{q^d}$ run modulo $p^a$ exactly through $p^c\Z/p^a$. Accordingly, the
resolvent of Lemma~\ref{lem:classical} for $E/\F_{q^d}$ at $\zeta^{dk}$ ($k\in\Z/N$) is, up to the factor $1/d'$, the sum of the
coordinates $F_\eta(w)(u)$ with $\eta^d=\zeta^{dk}$ and $p^c\mid u$ (see the proof of Theorem~\ref{thm:criterion});
this sum is the fibre sum $C_{d,k}(w)$ of the next definition.

\begin{definition}\label{def:cond}
A \emph{condition} is a pair $(d,k)$ with $d\mid n$ and $k\in\Z/N$, where $d=p^cd'$ and $N=n'/d'$; we call $d$ the
\emph{level} of the condition. Put
\[
C_{d,k}(w)=\sum_{t\in\Z/d'}\ \sum_{u\in p^c\Z/p^a}F_{\zeta^{k+tN}}(w)(u),\qquad
e=e(d,k)=\frac{N}{\gcd(N,k)},\qquad o(d,k)=\ord_e(q) .
\]
Conditions $(d,k)$ and $(d,q^dk)$ are identified, so a condition at level $d$ is an orbit of $k\mapsto q^dk$ on $\Z/N$. The
\emph{support} $\Supp(d,k)$ is the set of pair orbits meeting $\{(\zeta^{k+tN},u):t\in\Z/d',\ u\in p^c\Z/p^a\}$. The
condition is \emph{essential} if $d\mid o(d,k)$ (the name follows the $(q,n)$-essential sets of
intermediate fields of Hachenberger and Hackenberg \cite[Definition~3.2.1]{hachenberger2019computational}; see
Section~\ref{sec:remarks}).
\end{definition}
The identification is consistent: since $C_{d,q^dk}=\zeta^{-dq^dk}C_{d,k}^{q^d}$ (Lemma~\ref{lem:coord}(b) with $j=d$,
using $\zeta^{dN}=1$) and $\gcd(N,q^dk)=\gcd(N,k)$, the set $\{w:C_{d,k}(w)=0\}$ and the numbers $e$ and $o(d,k)$
depend only on the orbit, and hence so does essentiality.
Note that $e$ is the order of $\zeta^{d'k}$ and that the $\eta$'s occurring in $C_{d,k}$ are exactly the
solutions of $\eta^{d'}=\zeta^{d'k}$; we call this set the \emph{fibre} of the condition and $\zeta^{d'k}$ its
\emph{target}.
As $k$ runs through $\Z/N$, the target runs through $\mu_N$.
Replacing the representative $k$ by $q^dk$ raises the target and every fibre element to the power $q^d$. So the fibre
of a condition is defined only up to $\eta\mapsto\eta^{q^d}$, while its support and the $q$-orbits of its fibre
elements do not depend on the representative (the pair orbit of $(\eta,u)$ contains $(\eta^{q^d},u+d)$, and
$p^c\mid u+d$). Statements below about fibres hold for the fibre of every representative, and conclusions from
them use only these invariant data; we say that $\eta$ lies in the fibre of $C$ if it lies in the fibre
of some representative of $C$. For a condition $C=(d,k)$ we also write $C(w)=C_{d,k}(w)$
(whether $C(w)=0$ does not depend on the representative), and we use $e(C)$, $o(C)$ and $\Supp(C)$ accordingly. We put
$V_C=V_{d,k}=\{w:C(w)=0\}$ (Theorem~\ref{thm:criterion}).

\begin{theorem}[fibre-sum criterion]\label{thm:criterion}
Let $d\mid n$. Then $w\in E$ is normal over $\F_{q^d}$ if and only if $C_{d,k}(w)\ne0$ for every
$k\in\Z/N$. The kernel $V_{d,k}=\{w:C_{d,k}(w)=0\}$ is an $\F_q$-subspace of codimension
$\lcm(d,o(d,k))$; for an essential condition this is $o(d,k)$.
\end{theorem}
\begin{proof}
Apply Lemma~\ref{lem:classical} with $Q=q^d$, $m=n/d$ and $\sigma_Q=\sigma^d$; then $m'=N$. Put
$\hat\xi=\zeta^{p^cd'k}=(\zeta^{d'k})^{p^c}$, the $p^c$-th power of the target $\zeta^{d'k}$; as $k$ runs through $\Z/N$, the element $\hat\xi$ runs through $\mu_N$, and
$\ord\hat\xi=e$ because $p^c$-th powers preserve orders in $\mu_{n'}$. Inserting Lemma~\ref{lem:coord}(a) into
$R_{\hat\xi}(w)=\sum_{s\in\Z/m}\hat\xi^{-s}\sigma^{ds}(w)$ gives
\[
R_{\hat\xi}(w)=\frac1{n'}\sum_{\eta}\ \sum_{s\in\Z/m}(\hat\xi^{-1}\eta^d)^s\,F_\eta(w)(ds\bmod p^a).
\]
By the Chinese remainder theorem, $\Z/m\cong\Z/p^{a-c}\times\Z/N$; the factor $(\hat\xi^{-1}\eta^d)^s$
depends only on $s\bmod N$ (as $\hat\xi^{-1}\eta^d\in\mu_N$) and $ds\bmod p^a$ only on $s\bmod p^{a-c}$.
The sum over $s\bmod N$ is $N$ if $\eta^d=\hat\xi$ and $0$ otherwise, and $s\bmod p^{a-c}\mapsto ds\bmod p^a$ is a
bijection from $\Z/p^{a-c}$ onto $p^c\Z/p^a$. Since $\eta\mapsto\eta^{p^c}$ is a bijection of $\mu_{n'}$, we have
$\eta^d=\hat\xi$ iff $\eta^{d'}=\zeta^{d'k}$ iff $\eta=\zeta^{k+tN}$ for some $t\in\Z/d'$. Hence $R_{\hat\xi}(w)=\frac{N}{n'}C_{d,\allowbreak k}(w)=\frac1{d'}C_{d,\allowbreak k}(w)$,
and the first claim follows. By Lemma~\ref{lem:classical}, $V_{d,k}=g_{\hat\xi}(\sigma^d)E$, which has
$\F_{q^d}$-codimension $\deg g_{\hat\xi}=\ord_e(q^d)=\ord_e(q)/\gcd(\ord_e(q),d)$, i.e.\ $\F_q$-codimension
$d\,\ord_e(q)/\gcd(d,\ord_e(q))=\lcm(d,\ord_e(q))$.
\end{proof}

The proof shows that the conditions at level $d$ correspond bijectively to the distinct irreducible factors of $x^{n/d}-1$
over $\F_{q^d}$: as $k$ runs through $\Z/N$, the element $\hat\xi=\zeta^{p^cd'k}$ runs through $\mu_N$, the
identification of $k$ with
$q^dk$ corresponds to the identification of $\hat\xi$ with $\hat\xi^{q^d}$, and the condition $(d,k)$ corresponds to the minimal
polynomial $g_{\hat\xi}$ of $\hat\xi$ over $\F_{q^d}$, an irreducible factor of $x^N-1$ and hence of $x^{n/d}-1=(x^N-1)^{p^{a-c}}$.
Here $\hat\xi$ is the $p^c$-th power of the target $\zeta^{d'k}$ of $(d,k)$ (the target itself if $p\nmid d$); it has
order $e$, and $\{\eta:\eta^d=\hat\xi\}$ is the fibre. For $p\mid d$, however, the target itself need not be a
root of $g_{\hat\xi}$ (see the example at the end of this section).
Thus Theorem~\ref{thm:criterion} is the classical description of normality over $\F_{q^d}$ (an element is normal if and
only if it lies in none of the maximal proper $\sigma^d$-invariant $\F_{q^d}$-subspaces $g(\sigma^d)E$, $g$ an
irreducible factor of $x^{n/d}-1$), written in the coordinates of Lemma~\ref{lem:coord}. What the resolvent form adds is
that each $V_{d,k}$ is the zero set of a single fibre sum, which involves only the coordinates of the pair orbits in
$\Supp(d,k)$; this is what makes the independence arguments of Section~\ref{sec:product} possible.

\begin{theorem}[essential reduction]\label{thm:essential}
An element $w\in E$ is CN if and only if $C_{d,k}(w)\ne0$ for every \emph{essential} condition $(d,k)$.
\end{theorem}
\begin{proof}
By Theorem~\ref{thm:criterion}, $w$ is CN if and only if it avoids all kernels $V_{d,k}$; so it suffices to show that
for every condition $(d,k)$ that is not essential there is an essential condition $C$ with $V_{d,k}\subseteq V_C$.

Let $(d,k)$ be a condition, write $d=p^cd'$ and $m=n/d$, and let $\hat\xi=\zeta^{p^cd'k}$ and $e=\ord\hat\xi$ be as in the proof of
Theorem~\ref{thm:criterion}, so that $V_{d,k}=\{w:R_{\hat\xi}(w)=0\}$ with $R_{\hat\xi}(w)=\sum_{s\in\Z/m}\hat\xi^{-s}\sigma^{ds}(w)$.
Raising to the power $q^j$ termwise gives $R_{\hat\xi}(w)^{q^j}=R_{\hat\xi^{q^j}}(\sigma^jw)$, and $R_{\hat\xi}(\sigma^du)=\hat\xi R_{\hat\xi}(u)$.
Suppose that $(d,k)$ is not essential, and put $c_0=\gcd(d,\ord_e(q))<d$.
We show in two steps that $V_{d,k}\subseteq V_{C'}$ for an essential condition $C'$ at level $c_0$: first, $V_{d,k}$ is
invariant under $\sigma^{c_0}$, since $\sigma^{c_0}$ maps it onto the kernel of the resolvent at $\hat\xi^{q^{c_0}}$ and
$\hat\xi^{q^{c_0}}$ lies in the $q^d$-orbit of $\hat\xi$; second, the resolvent of $C'$ is a linear combination of the values
$R_{\hat\xi}(\sigma^{c_0r}w)$, $0\le r<d/c_0$.

First, $V_{d,k}$ is invariant under $\sigma^{c_0}$. Given $r\ge0$, choose $b\ge0$ with $\ord_e(q)\mid db+c_0r$ (possible since
$\gcd(d,\ord_e(q))=c_0$) and put $\hat\xi'=\hat\xi^{q^{c_0r}}$, so that $(\hat\xi')^{q^{db}}=\hat\xi$. If $R_{\hat\xi}(w)=0$, then
$R_{\hat\xi'}(\sigma^{c_0r}w)=R_{\hat\xi}(w)^{q^{c_0r}}=0$, hence
$0=R_{\hat\xi'}(\sigma^{c_0r}w)^{q^{db}}=R_{\hat\xi}(\sigma^{db}\sigma^{c_0r}w)=\hat\xi^bR_{\hat\xi}(\sigma^{c_0r}w)$, and $R_{\hat\xi}(\sigma^{c_0r}w)=0$
as $\hat\xi\ne0$.

Second, fix an integer representative $\tilde k$ of $k$, write $c_0=p^{\gamma}c_0'$ with $p\nmid c_0'$, and let $C'$ be the
condition $(c_0,\tilde k\bmod n'/c_0')$ at level $c_0$. The element attached to $C'$ as in the proof of
Theorem~\ref{thm:criterion} is $\hat\xi_1=\zeta^{p^{\gamma}c_0'\tilde k}$, the $p^\gamma$-th power of its target $\zeta^{c_0'\tilde k}$, so that $\hat\xi_1^{d/c_0}=\zeta^{d\tilde k}=\hat\xi$ and $V_{C'}=\{w:R^{(c_0)}_{\hat\xi_1}(w)=0\}$ with
\[
R^{(c_0)}_{\hat\xi_1}(w)=\sum_{t\in\Z/(n/c_0)}\hat\xi_1^{-t}\sigma^{c_0t}(w).
\] Writing
$t\in\Z/(n/c_0)$ as $t=\frac d{c_0}s+r$ with $s\in\Z/m$ and $0\le r<d/c_0$, we obtain
\[
R^{(c_0)}_{\hat\xi_1}(w)=\sum_{0\le r<d/c_0}\hat\xi_1^{-r}\,R_{\hat\xi}(\sigma^{c_0r}w),
\]
which vanishes for $w\in V_{d,k}$ by the first step. Hence $V_{d,k}\subseteq V_{C'}$. Finally, $C'$ is essential:
the order $e(C')$ of its target equals $\ord\hat\xi_1$, and $\hat\xi=\hat\xi_1^{d/c_0}$ gives $e\mid e(C')$, so
$\ord_e(q)\mid\ord_{e(C')}(q)=o(C')$; as $c_0\mid\ord_e(q)$, the level $c_0$ divides $o(C')$.
\end{proof}

For $n>1$ the levels that carry an essential condition, i.e.\ the $d\mid n$ with $d\mid\ord_{(n/d)'}(q)$, are exactly
the elements of the $(q,n)$-essential set $\mathcal D^*$ of Hachenberger and Hackenberg
\cite[Proposition~3.2.4]{hachenberger2019computational}, for which $w$ is CN if and only if it is normal over
$\F_{q^d}$ for every $d\in\mathcal D^*$. Indeed, by \cite[Definition~3.2.3 and Proposition~3.2.4]{hachenberger2019computational}, a divisor
$d<n$ of $n$ lies in $\mathcal D^*$ iff every prime $r\mid d$ divides $\ord_{(n/d)'}(q^{d/r})$, which is equivalent to
$d\mid\ord_{(n/d)'}(q)$. The construction of $\mathcal D^*$ rests on \cite[Theorem~15.5]{hachenberger1997book}, which
Hachenberger derives from his admissibility lemma for the cyclotomic parts of $E$
\cite[Lemma~15.3 and Corollary~15.4]{hachenberger1997book}.
Theorem~\ref{thm:essential} refines this reduction to individual targets: at level $d$ only the targets of order $e$
with $d\mid\ord_e(q)$ are needed. Its statement can also be derived from \cite[Lemma~15.3]{hachenberger1997book},
applied over $\F_{q^{c_0}}$ with $c_0=\gcd(d,\ord_e(q))$ (as $\gcd(d/c_0,\ord_e(q^{c_0}))=1$); the proof above gives
the inclusion of kernels directly in the resolvent coordinates.

From Section~\ref{sec:product} on, only essential conditions are used. Their defining property $d\mid o(d,k)$, that the
level divides the degree $\ord_e(q)$ of the target over $\F_q$ (which is then the codimension), is also the key to the
structure theorem of Section~\ref{sec:struct}.

\paragraph{An example with $p\mid n$.} Let $(q,n)=(2,6)$, so $p=2$, $a=1$, $n'=3$ and
$\mu_3=\{1,\zeta,\zeta^2\}\subseteq\F_4$. The pair orbits are $\{(1,0),(1,1)\}$, $\{(\zeta,0),(\zeta^2,1)\}$ and
$\{(\zeta,1),(\zeta^2,0)\}$. At level $1$ ($c=0$, $N=3$, $k\sim2k$) there are two conditions, $(1,0)$ and
$(1,1)=(1,2)$, with fibre sums $F_1(w)(0)+F_1(w)(1)$ and $F_\zeta(w)(0)+F_\zeta(w)(1)$ and codimensions $1$ and $2$,
respectively. At
level $2$ ($c=1$, $d'=1$, $N=3$, \mbox{$q^2\equiv1\pmod3$}) there are three, $C_{2,k}(w)=F_{\zeta^k}(w)(0)$ for $k=0,1,2$, each
of codimension~$2$; for example, $(2,1)$ has target $\zeta$ and kernel $(\sigma^2-\zeta^2)E$, so it corresponds to the
factor $x-\zeta^2$ of $x^3-1$ over $\F_4$, of which its target is not a root. At levels $3$ and $6$ ($N=1$) there is one
condition each, of codimension $3$ and $6$, respectively. Exactly $(1,0)$, $(1,1)$, $(2,1)$ and $(2,2)$ are
essential; for instance, $V_{2,0}\subseteq V_{1,0}$, since $F_1(w)(1)=F_1(w)(0)^2$ by Lemma~\ref{lem:coord}(b). There are $\CN_2(6)=12$
completely normal elements.

\section{Components and the product formula}\label{sec:product}

This section groups the essential conditions, which by Theorem~\ref{thm:essential} are all that is needed, into
\emph{components} such that distinct components involve disjoint sets of pair orbits; by Lemma~\ref{lem:coord}(c) the
density $\delta_n$ is then the product of the probabilities that a uniformly distributed $w$ passes the individual
components (Theorem~\ref{thm:product}). Lemma~\ref{lem:qorbits} describes the components in terms of the $q$-orbits of
$\mu_{n'}$ met by the fibres, without the residues $u$; in this form they are studied in Section~\ref{sec:struct}.

\begin{definition}
Let $\cE$ be the set of all essential conditions. Two conditions of $\cE$ are \emph{adjacent} if their supports share a
pair orbit. The classes of the generated equivalence relation are the \emph{components} (or \emph{resolvent
components}) of $\cE$. The support of a component is the union of the supports of its members. The \emph{weight} of a
set $\cC\subseteq\cE$ is $w_{\cC}=\sum_{C\in\cC}q^{-o(C)}$, the sum of the failure probabilities $\Pr[C(w)=0]$ of its
members for $w$ uniformly distributed in $E$ (Theorem~\ref{thm:criterion}).
\end{definition}

\begin{theorem}[product formula]\label{thm:product}
Let $\cC_1,\dots,\cC_s$ be the components of $\cE$ and let $w$ be uniformly distributed in $E$. The events
$G_i=\{C(w)\ne0\text{ for all }C\in\cC_i\}$ are independent, and
\[
\delta_n=\frac{\CN_q(n)}{q^n}=\prod_{i=1}^s\Pr[G_i]
\ \ge\ \prod_{i=1}^s\bigl(1-w_{\cC_i}\bigr)\quad\text{if $w_{\cC_i}<1$ for all $i$.}
\]
\end{theorem}
\begin{proof}
The fibre sum $C_{d,k}(w)$ is a sum of coordinates $F_\eta(w)(u)$ with $(\eta,u)$ in pair orbits of $\Supp(d,k)$, and by
Lemma~\ref{lem:coord}(b) each of them is a function of the representative coordinate of its pair
orbit. So $G_i$ is determined by the coordinates of the orbits in the support of $\cC_i$; the supports of
$\cC_1,\dots,\cC_s$ are pairwise disjoint (a pair orbit in the supports of $C\in\cC_i$ and $C'\in\cC_j$ makes $C$ and $C'$ adjacent, so
$i=j$), and by Lemma~\ref{lem:coord}(c) the corresponding coordinate vectors are
independent. By Theorem~\ref{thm:essential}, $w$ is CN if and only if it lies in $G_1\cap\dots\cap G_s$, which with the
independence gives the formula for $\delta_n$; the inequality is the union
bound inside each component, since $\Pr[C(w)=0]=q^{-o(C)}$ by Theorem~\ref{thm:criterion}.
\end{proof}

Together with Theorem~\ref{thm:criterion}, Theorem~\ref{thm:product} is closely related to Hachenberger's decomposition
theory; see Section~\ref{sec:remarks}.

The product formula is exact, and the only loss in the lower bound is the union bound inside each component, which is
an equality for a component consisting of a single condition. What has to be bounded is therefore not the total weight
of $\cE$, which for fixed $q$ is unbounded in $n$ (Section~\ref{subsec:idea}), but the weight of each component,
uniformly in $n$. By Theorem~\ref{thm:introW},
which is proved in Section~\ref{sec:struct}, every component without a condition of codimension~$1$ has weight
at most a constant $W(q)<1$. Section~\ref{sec:main} turns this into a lower bound for
$\delta_n$ and combines it with the counting criterion of Section~\ref{sec:counting} (Theorem~\ref{thm:C}).

\begin{lemma}\label{lem:qorbits}
\leavevmode
\begin{enumerate}[label=(\alph*),beginpenalty=10000]
\item Let $C,C'\in\cE$. If $C$ and $C'$ are adjacent, their fibres contain elements of a common $q$-orbit. If the
fibre of $C$ meets the $q$-orbit $O$ of $\eta\in\mu_{n'}$, then $C$ is adjacent to the level-one condition with fibre
$\{\eta\}$; this condition is essential, and its support consists of all pair orbits contained in $O\times\Z/p^a$.
Consequently, two conditions of $\cE$ lie in the same component if and only if they are connected by a chain of
conditions of $\cE$ in which consecutive fibres contain elements of a common $q$-orbit. The elements $\eta$ in the fibres of
the members of a component form a union of $q$-orbits $O$, and the support of the component consists of the pair
orbits contained in these sets $O\times\Z/p^a$.
\item If $\eta$ lies in the fibre of a representative of a condition at level $d$, this fibre is the coset
$\eta\,\mu_{d'}$.
\item An essential condition $C$ with $o(C)=1$, i.e.\ of codimension $1$, is necessarily of level $1$ and forms a component
by itself (a \emph{codimension-one singleton}); every other essential condition has codimension~$\ge2$, and its
component contains no condition of codimension $1$. So the components other than the codimension-one singletons
contain exactly the essential conditions of codimension~$\ge2$, and their weights add up to $S$.
\end{enumerate}
\end{lemma}
\begin{proof}
(a) For every $\eta$ the level-one condition with fibre $\{\eta\}$ is essential and its support consists of the pair
orbits through $(\eta,u)$, $u\in\Z/p^a$, which together cover $O\times\Z/p^a$ ($O$ the $q$-orbit of $\eta$; the pair
orbit of $(\eta,u-j)$ contains $(\eta^{q^j},u)$). Any condition of $\cE$ whose support contains a
pair orbit in $O\times\Z/p^a$, in particular any condition of $\cE$ whose fibre meets $O$, is adjacent to it. Hence two
conditions of $\cE$ whose fibres meet a common $q$-orbit lie in the
same component. Conversely, if two conditions are adjacent, a pair orbit $\pi$ lies in both supports, so $\pi$ contains
$(\eta_1,u_1)$ and $(\eta_2,u_2)$ with $\eta_1$, $\eta_2$ in the two fibres; as $\pi$ is an orbit of
$(\eta,u)\mapsto(\eta^q,u+1)$, we have $\eta_2=\eta_1^{q^j}$ for some $j$, so the two fibres meet a common $q$-orbit. Thus the
component relation is the equivalence relation on $\cE$ generated by ``the fibres meet a common $q$-orbit''. If $\eta$
lies in the fibre of a member of a component, the level-one condition with fibre $\{\eta\}$ belongs to the component;
the fibres of its representatives are the sets $\{\eta^{q^j}\}$, and its support consists of all pair orbits in
$O\times\Z/p^a$, $O$ the $q$-orbit of $\eta$. As every pair orbit in the support of a member passes through a point
$(\eta,u)$ with $\eta$ in the fibre of that member, this proves the description of the $\eta$'s and of the support of a
component.

(b) An element $x$ satisfies $x^{d'}=\eta^{d'}$ if and only if $(x/\eta)^{d'}=1$.

(c) If $o(C)=\ord_e(q)=1$, then $d\mid\ord_e(q)$ forces $d=1$, and the fibre is $\{\eta\}$ for some $\eta$ with
$\ord_{\ord\eta}(q)=1$, so $\eta^q=\eta$.
By (a), any essential condition adjacent to it, say of level $\tilde d=p^{\tilde c}\tilde d'$, has a fibre meeting the
$q$-orbit $\{\eta\}$, i.e.\ containing $\eta$; its target is $\eta^{\tilde d'}$, again with $o=1$, so it has level $1$ and is
the same condition. Hence the component consists of this condition alone. So a component containing a condition of
codimension~$\ge2$ contains none of codimension $1$, and as the components partition $\cE$, the weights of the
components other than the codimension-one singletons add up to $S$.
\end{proof}

If $p\mid n$, two conditions whose fibres meet a common $q$-orbit need not be adjacent, although by Lemma~\ref{lem:qorbits}(a) they
lie in the same component. In the example $(q,n)=(2,6)$ at the end of
Section~\ref{sec:coord}, the two essential conditions $(2,1)$ and $(2,2)$ at level $2$ have the fibres $\{\zeta\}$ and
$\{\zeta^2\}=\{\zeta^q\}$, which lie in a common $q$-orbit; but as $C_{2,k}(w)=F_{\zeta^k}(w)(0)$, their supports are
the different single pair orbits $\{(\zeta,0),(\zeta^2,1)\}$ and $\{(\zeta,1),(\zeta^2,0)\}$, so the two conditions are
not adjacent.

\section{The counting criterion}\label{sec:counting}

In this section we estimate $\PCN_q(n)$ from above and below in terms of the density $\delta_n$
(Theorem~\ref{thm:counting}(a), which is also the starting point of Corollary~\ref{cor:pcn}), and we show that a lower bound $\delta_*$
for $\delta_n$ yields a sufficient criterion for the Morgan--Mullen property (Theorem~\ref{thm:counting}(b)). In
Section~\ref{sec:main}, $\delta_*$ is obtained from the product formula (Theorem~\ref{thm:product}) and the weight bound
of Theorem~\ref{thm:introW}; the other quantities in this criterion, $\omega(q^n-1)$ and, through $L$, the number $M$
of essential conditions, are treated in Sections~\ref{sec:bounds} and~\ref{sec:M}.

The argument of this section is the classical character-sum method
\cite{carlitz1952primitive,carlitz1952sums,davenport1968bases,lenstra1987primitive,cohen2003pnbt}, in the form used for completely
normal elements in \cite{garefalakis2019jpaa,garefalakis2019dcc}, whose main term is the number of completely normal
elements; see also \cite{ostafe2013applications}. This form was inspired by Hachenberger's estimate of this number
in \cite{hachenberger2016asymptotic} (see \cite{garefalakis2026cmb}). For regular extensions, Hachenberger combined the method with the module
structure of $E$ \cite{hachenberger2001regular,hachenberger2010twopower,hachenberger2019exceptional}; in
\cite[Proposition~5.1 and (5.10)--(5.15)]{hachenberger2001regular} and
\cite[Propositions~5.2 and~7.1]{hachenberger2019exceptional} the main term is likewise the exact number of
completely normal elements, written as a product over the cyclotomic modules.

In the criterion of Garefalakis and Kapetanakis \cite[Theorem~3.1]{garefalakis2019jpaa}, the error term contains a
factor $2(1-q^{-\operatorname{codim}V_C})$ for every condition $C$ of every level $d<n$; Hachenberger and Hackenberg
\cite[Proposition~3.2.6]{hachenberger2019computational} restrict these factors to the conditions at the levels of an
essential set of intermediate fields, and in Hachenberger's criterion for regular pairs
\cite[Theorem~5.3]{hachenberger2001regular} (with $q\equiv1\pmod4$ if $q$ is odd and $n$ is even) they occur
only for the conditions over one intermediate field per cyclotomic module. By Theorem~\ref{thm:essential}, only the essential conditions occur in
Theorem~\ref{thm:counting}; their factors make up the number $L$.

Let $\psi$ be a non-trivial additive character of $E$, and for $b\in E$ put $\psi_b(w)=\psi(bw)$. Let $P\subseteq
E^\times$ be the set of primitive elements, and put $\hat P(b)=\sum_{w\in P}\psi_b(w)$. Finally, let
$\theta=\varphi(q^n-1)/(q^n-1)$ and $\omega=\omega(q^n-1)$.

\begin{lemma}\label{lem:Gamma}
If $b\ne0$, then $|\hat P(b)|\le\theta\bigl(1+(2^\omega-1)q^{n/2}\bigr)$.
\end{lemma}
\begin{proof}
By Vinogradov's formula (see e.g.\ \cite[(2.1)]{lenstra1987primitive}, \cite{cohen2003pnbt},
\cite[(3.1)]{cohen2010strong}, \cite[Exercise~5.14]{lidl1997finite} or \cite[Proposition~10.2.5]{hachenberger2020topics})
\[
1_P(w)=\theta\sum_{D\mid q^n-1}\frac{\mu(D)}{\varphi(D)}\sum_{\ord\chi=D}\chi(w)\qquad(w\in E^\times),
\]
where $\chi$ runs through the multiplicative characters of $E$. Hence
\[
\hat P(b)=\theta\sum_{D\mid q^n-1}\frac{\mu(D)}{\varphi(D)}\sum_{\ord\chi=D}G(\chi,\psi_b),\qquad
G(\chi,\psi_b)=\sum_{w\in E^\times}\chi(w)\psi_b(w).
\] As $b\ne0$ and $\psi$ is non-trivial, $\psi_b$ is non-trivial. Hence
$G(\chi,\psi_b)=\sum_{w\ne0}\psi_b(w)=-1$ for trivial $\chi$ \cite[Proposition~10.3.2]{hachenberger2020topics}, and
otherwise $|G(\chi,\psi_b)|=q^{n/2}$ \cite[Theorem~5.11]{lidl1997finite}, \cite[Corollary~10.3.3]{hachenberger2020topics}
(see also \cite{ireland1990classical,berndt1998gauss,evans2013gauss}).
Only squarefree $D$ contribute, there are $2^\omega$ of them, and there are $\varphi(D)$ characters of order $D$.
So $D=1$ contributes $-\theta$, and each of the $2^\omega-1$ squarefree $D>1$ contributes at most $\theta q^{n/2}$ in
absolute value.
\end{proof}

\begin{theorem}[counting criterion]\label{thm:counting}
Let $M$ be the number of essential conditions and $M_1$ the number of those of codimension~$1$. Let
$L=\prod_{C}2\bigl(1-q^{-o(C)}\bigr)$, the product over all essential conditions, and let $\beta=1+(2^\omega-1)q^{n/2}$.
Then $L\le2^M(1-\frac1q)^{M_1}$, and:
\begin{enumerate}[label=(\alph*)]
\item the number $\PCN_q(n)$ satisfies $\bigl|\PCN_q(n)-\theta(q^n-1)\,\delta_n\bigr|\le\theta\beta\,(L-\delta_n)$;
\item for every real $\delta_*$ with $0<\delta_*\le\delta_n$,
\[
\PCN_q(n)\ \ge\ \theta\bigl[(q^n-1)\,\delta_*-\beta\,(L-\delta_*)\bigr];
\]
in particular, if $(q^n-1)\,\delta_*\ge 2^\omega\,q^{n/2}\,L$, then $(q,n)$ has the Morgan--Mullen property.
\end{enumerate}
\end{theorem}
\begin{proof}
Write $1_{\CN}=\sum_bc(b)\psi_b$ (the expansion of the indicator function of the CN elements in additive
characters) and $\|c\|_1=\sum_b|c(b)|$. We first show $\|c\|_1\le L$. By Theorem~\ref{thm:essential},
$1_{\CN}=\prod_C(1-1_{V_C})$ over the essential conditions, so $c$ is the convolution of the coefficient functions of
the factors and $\|c\|_1$ is at most the product of their $\ell^1$-norms. For $V=V_C$ and $o=o(C)$, let
$V^\perp=\{a:\psi(aV)=1\}$. As $C$ is essential, $V$ has codimension $o$ (Theorem~\ref{thm:criterion}), so
$|V^\perp|=q^n/|V|=q^{o}$ and $1_V=q^{-o}\sum_{a\in V^\perp}\psi_a$. Hence
$1-1_V=(1-q^{-o})\psi_0-q^{-o}\sum_{a\in V^\perp\setminus\{0\}}\psi_a$, whose $\ell^1$-norm is
$(1-q^{-o})+(q^{o}-1)q^{-o}=2(1-q^{-o})$. Moreover $L\le2^M(1-\frac1q)^{M_1}$, since every factor of $L$ is at most
$2$ and the $M_1$ factors with $o(C)=1$ equal $2(1-\frac1q)$.

(a) We have $c(0)=q^{-n}\sum_w1_{\CN}(w)=\delta_n$ and $\hat P(0)=|P|=\theta(q^n-1)$, so
\[
\PCN_q(n)=\sum_bc(b)\hat P(b)=\theta(q^n-1)\,\delta_n+\sum_{b\ne0}c(b)\hat P(b).
\]
By Lemma~\ref{lem:Gamma}, $|\hat P(b)|\le\theta\beta$ for $b\ne0$, and $\sum_{b\ne0}|c(b)|=\|c\|_1-\delta_n\le
L-\delta_n$; so the last sum has absolute value at most $\theta\beta(L-\delta_n)$.

(b) By (a), $\PCN_q(n)\ge\theta\bigl[(q^n-1)\delta_n-\beta(L-\delta_n)\bigr]=\theta\bigl[(q^n-1+\beta)\delta_n-\beta L\bigr]$.
This increases with $\delta_n$, so replacing $\delta_n$ by $\delta_*\le\delta_n$ gives the asserted inequality.
If $(q^n-1)\delta_*\ge2^\omega q^{n/2}L$ then, as $\beta\le2^\omega q^{n/2}$ and $L>0$, the bracket
$(q^n-1)\delta_*-\beta L+\beta\delta_*$ is at least $\beta\delta_*>0$, so $\PCN_q(n)>0$.
\end{proof}

\section{Prime factors of \texorpdfstring{$q^n-1$}{q\^{}n-1} and essential types}\label{sec:bounds}
Besides the weight bound $W$, the criterion $(*)$ of Theorem~\ref{thm:introC} involves three quantities:
the number $\omega(q^n-1)$ of distinct prime factors of $q^n-1$, which enters through the character-sum estimate of Lemma~\ref{lem:Gamma}; the number $M$ of essential conditions,
which enters through $L\le2^M(1-\frac1q)^{M_1}$ (Theorem~\ref{thm:counting}); and the total weight $S$ of the essential
conditions of codimension~$\ge2$, which enters through the lower bound $\delta_*$ for $\delta_n$
(Theorem~\ref{thm:C}(a)). This section provides the tools for
evaluating these quantities. Lemma~\ref{lem:omega} bounds $\omega(q^n-1)$ without factoring $q^n-1$; it is used for the finitely many
pairs treated by computation in Section~\ref{sec:main} (for large $n$ a cruder bound suffices there).
Lemma~\ref{lem:weightcodim} gives $M$ and $S$ exactly as finite sums over the
types (level, order of the target) of the essential conditions and bounds the number of essential conditions of each
codimension (used for the analytic bound for $S$ in Section~\ref{sec:main}); it is also the starting point of
Section~\ref{sec:M}.

\begin{lemma}\label{lem:omega}
The number of distinct prime factors of $q^n-1$ satisfies
\[
\omega(q^n-1)\le\omega(q-1)+\sum_{k\mid n,\,k>1}t_k,
\]
where $t_k$ is the largest $t$ such that the
product of the $t$ smallest primes $\equiv1\pmod{\lcm(2,k)}$ is at most $\Phi_k(q)$. Also
$t_k\le\bigl\lfloor(\varphi(k)\log_2q+c_q)/\log_2(\lcm(2,k)+1)\bigr\rfloor$ with $c_q=-\sum_{j\ge1}\log_2(1-q^{-j})$.
\end{lemma}
\begin{proof}
Every prime $\ell\mid q^n-1$ has $k=\ord_\ell(q)\mid n$, and the primes with $k=1$ divide $q-1$. For $k\ge2$ (these
$\ell$ are the primitive prime divisors of $q^k-1$, cf.\ \cite{bang1886,zsigmondy1892}), $\ell\equiv1\pmod k$, $\ell$ is
odd ($q^n-1$ is odd for even $q$, and $\ord_2(q)=1$ for odd $q$), and $\ell\mid\Phi_k(q)$; the product of these distinct $\ell$ divides $\Phi_k(q)$. If
$t$ primes $\ell$ have $\ord_\ell(q)=k$, their product is therefore at least the product of the $t$ smallest primes
$\equiv1\pmod{\lcm(2,k)}$ and at most $\Phi_k(q)$, so $t\le t_k$. Finally
$\Phi_k(q)=q^{\varphi(k)}\prod_{j\mid k}(1-q^{-j})^{\mu(k/j)}\le q^{\varphi(k)}\prod_{j\ge1}(1-q^{-j})^{-1}$ (drop the factors
with $\mu(k/j)=1$), and $(\lcm(2,k)+1)^{t_k}\le\Phi_k(q)$ gives the bound for $t_k$, as $t_k$ is an integer.
\end{proof}

By Definition~\ref{def:cond} and the paragraph after it, the targets of the conditions at level $d$ are the elements of
$\mu_{n'/d'}$. The codimension $\lcm(d,\ord_e(q))$ of a condition (Theorem~\ref{thm:criterion}) and whether it is
essential depend only on its level $d$ and the order $e$ of its target; we call $(d,e)$ the \emph{type} of the
condition. Thus the types are the pairs $(d,e)$ with $d\mid n$ and $e\mid n'/d'$ (equivalently, $de\mid n$ and
$p\nmid e$), and a type is \emph{essential} if $d\mid\ord_e(q)$, i.e.\ if its conditions are essential.

\begin{lemma}\label{lem:weightcodim}
For every essential type $(d,e)$ there are exactly $\nu_{d,e}=d\varphi(e)/\ord_e(q)$ conditions of type $(d,e)$,
namely the orbits of $\xi\mapsto\xi^{q^d}$ on the $\varphi(e)$ elements of order $e$ of $\mu_{n'/d'}$ (a condition being
identified with the orbit of its target); each of them consists of $\ord_e(q^d)=\ord_e(q)/d$ targets and has
codimension $\ord_e(q)$. Hence, with both sums over the essential types, the number $M$ of essential
conditions and the total weight $S$ of the essential
conditions of codimension~$\ge2$ are
\[
M=\sum_{(d,e)}\nu_{d,e},\qquad S=\sum_{(d,e):\,\ord_e(q)\ge2}\nu_{d,e}\,q^{-\ord_e(q)}.
\]
The number of essential conditions of codimension $o$ is at most
\[
\frac{\sigma_1(o)}{o}\cdot\#\{\xi\in\mu_{n'}:\ \xi\text{ has degree }o\text{ over }\F_q\}
\ \le\ \frac{\sigma_1(o)}{o}\min(n',q^o-1).
\]
\end{lemma}
\begin{proof}
The targets of level-$d$ conditions are the elements of $\mu_{n'/d'}$, and conditions are the orbits of
$\xi\mapsto\xi^{q^d}$ on them; the $\varphi(e)$ targets of order $e$ fall into orbits of size $\ord_e(q^d)=\ord_e(q)/d$
(as $d\mid \ord_e(q)$), which gives $\nu_{d,e}$; by Theorem~\ref{thm:criterion} each of these conditions has codimension
$\lcm(d,\ord_e(q))=\ord_e(q)$. To prove the last assertion, fix $o$. The elements of $\mu_{n'}$ of degree $o$ over
$\F_q$, i.e.\ those whose order $e$ satisfies $\ord_e(q)=o$, lie in $\mu_{n'}\cap\F_{q^o}^\times=\mu_{\gcd(n',q^o-1)}$;
this gives the second inequality. For the first, recall from the first part that an essential condition has
codimension $o$ if and only if its target has an order $e$ with $\ord_e(q)=o$. For such an $e$, the levels of
essential conditions with a target of order $e$ divide $o$, and each level $d\mid o$ gives at most $d$ such
conditions per $q$-orbit of elements of order $e$ (such a $q$-orbit has $o$ elements and splits into $d$ orbits of
$\xi\mapsto\xi^{q^d}$; if $(d,e)$ is not a type, level $d$ gives none); there are
$\varphi(e)/o$ such $q$-orbits, so at most $\frac{\varphi(e)}o\sigma_1(o)$ conditions; sum over $e$.
\end{proof}

For the example $(q,n)=(2,6)$ of Section~\ref{sec:coord} ($n'=3$), the six types $(d,e)$ are $(1,1)$, $(1,3)$, $(2,1)$,
$(2,3)$, $(3,1)$, $(6,1)$; the essential ones ($\ord_1(2)=1$, $\ord_3(2)=2$) are $(1,1)$, $(1,3)$, $(2,3)$, with
$\nu_{d,e}=1,1,2$ (the conditions $(d,k)=(1,0)$; $(1,1)$; $(2,1)$, $(2,2)$ there). So $M=4$, $M_1=1$ and
$S=3\cdot2^{-2}=\frac34$; the three conditions of codimension $2$ form one component (Lemma~\ref{lem:qorbits}(a)),
of weight $\frac34=W(2)$ (Theorem~\ref{thm:introW}).

\section{The number of essential conditions}\label{sec:M}
In the criterion $(*)$ of Theorem~\ref{thm:introC}, the decisive term is the number $M$ of essential conditions
(Section~\ref{subsec:idea}). We write $M$ as a sum of local contributions $\rho(\eta)$, $\eta\in\mu_{n'}$
(Lemma~\ref{lem:Mrho}), and bound $\rho(\eta)/p^a$ in terms of $q$, $a=v_p(n)$ and the degree $\ord_{\ord\eta}(q)$ of
$\eta$ only (Proposition~\ref{prop:Mlocal}); this gives $M(q,n)\le n$ for all pairs (Corollary~\ref{cor:Mn}) and, for
$q\le4$, the sharper bounds of Proposition~\ref{prop:Msmall}.

For $\eta\in\mu_{n'}$ of order $k$ and a divisor $d=p^cd'$ of $n$ put $e_d(\eta)=\ord(\eta^{d'})=k/\gcd(k,d')$ and
$o_d(\eta)=\ord_{e_d(\eta)}(q)$; in particular $o(\eta):=o_1(\eta)=\ord_k(q)$ is the degree of $\eta$ over $\F_q$, that is,
the size of its $q$-orbit. As $e_d(\eta)\mid k$, we have $o_d(\eta)\mid o(\eta)$. At each level $d\mid n$ the element $\eta$ lies in the
fibre of exactly one condition $C$, namely the one with target $\eta^{d'}$ (up to $\xi\mapsto\xi^{q^d}$); this target
has order $e_d(\eta)$, so $o(C)=o_d(\eta)$. The level $d$ is \emph{valid at $\eta$} if $(d,e_d(\eta))$ is an essential
type, i.e.\ $d\mid o_d(\eta)$ (equivalently, if $C$ is essential; its codimension is then $o_d(\eta)$). We put
\[
\rho(\eta)=\sum_{d\ \text{valid at}\ \eta}\frac{p^{c}}{o_d(\eta)},\qquad
s(\eta)=\sum_{\substack{d\ \text{valid at}\ \eta\\ o_d(\eta)\ge2}}\frac{p^{c}\,q^{-o_d(\eta)}}{o_d(\eta)} .
\]

\begin{lemma}\label{lem:Mrho}
The number $M$ of essential conditions and the total weight $S$ of those of codimension~$\ge2$ are
$M=\sum_{\eta\in\mu_{n'}}\rho(\eta)$ and $S=\sum_{\eta\in\mu_{n'}}s(\eta)$.
\end{lemma}
\begin{proof}
There are $\nu_{d,e}=d\varphi(e)/\ord_e(q)$ essential conditions of type $(d,e)$ (Lemma~\ref{lem:weightcodim}), and exactly
$d'\varphi(e)$ elements $\eta\in\mu_{n'}$ with $\ord(\eta^{d'})=e$ (the $d'$-th power map $\mu_{n'}\to\mu_{n'/d'}$ is
$d'$-to-one and onto). Distributing each type evenly over these elements gives $p^c/\ord_e(q)$ per element. For fixed
$\eta$ and $d\mid n$, the element $\eta$ receives a share only from the type $(d,e_d(\eta))$, and only if this type is essential,
i.e.\ if $d$ is valid at $\eta$; summing over $d$ gives $\rho(\eta)$. For $S$, distribute in the same way the total
weight $\nu_{d,e}q^{-\ord_e(q)}$ of the conditions of each essential type $(d,e)$ with $\ord_e(q)\ge2$: this gives
$p^cq^{-\ord_e(q)}/\ord_e(q)$ per element, and $\eta$ receives such a share exactly from the levels $d$ valid at $\eta$
with $o_d(\eta)\ge2$; summing over $d$ gives $s(\eta)$.
\end{proof}

Since $|\mu_{n'}|=n/p^a$, a bound $\rho(\eta)\le\lambda p^a$ for all $\eta$ gives $M\le\lambda n$
(Lemma~\ref{lem:Mrho}), so we bound $\rho(\eta)/p^a$. Let $\eta$ have order $k$ and degree $o=\ord_k(q)$ over $\F_q$.
The levels valid at $\eta$ divide $o$, as $d\mid o_d(\eta)\mid o(\eta)=o$. The quantity $R_a(k)$ defined below is the sum defining
$\rho(\eta)/p^a$ with $d\mid n$ replaced by $d\mid o$, $v_p(d)\le a$. It equals $\rho(\eta)/p^a$ if $o'\mid n$, since
then every $d\mid o$ with $v_p(d)\le a$ divides $n$; so it is the worst case over all $n$ with $v_p(n)=a$. Replacing each
term by $p^{v_p(d)-a}\min(\frac1d,\frac do)$ and summing over all $d\mid o$ with $v_p(d)\le a$ gives the divisor sum
$f_a(o)$, which depends only on $p$, $a$ and $o$.
For integers $o\ge1$, $a\ge0$ define (here $a$ is a free parameter; in Proposition~\ref{prop:Mlocal} it will be
$a=v_p(n)$)
\[
f_a(o)=\sum_{\substack{d\mid o\\ v_p(d)\le a}}p^{\,v_p(d)-a}\min\Bigl(\frac1d,\frac d{o}\Bigr),\qquad
f(o)=\sum_{d\mid o}\min\Bigl(\frac1d,\frac d{o}\Bigr)
\]
(so $f_0$ only involves $p$-free $d$), and for $k\mid q^o-1$ with $\ord_k(q)=o$
\[
R_a(k)=\sum_{\substack{d\mid o,\ v_p(d)\le a\\ d\,\mid\, \ord_{k/\gcd(k,d')}(q)}}\frac{p^{\,v_p(d)-a}}{\ord_{k/\gcd(k,d')}(q)},\qquad
R^*_a(o)=\max\{R_a(k):\ k\mid q^o-1,\ \ord_k(q)=o\}.
\]

\begin{lemma}\label{lem:divsum}
For every $o\ge1$: $\sum_{d\mid o}\min(d,o/d)\le o$, i.e.\ $f(o)\le1$.
\end{lemma}
\begin{proof}
Let $s=\lfloor\sqrt o\rfloor$. The sum is $2\sum_{d\mid o,\,d<\sqrt o}d+[o=s^2]\,s$. If $o=s^2$, the divisors
below $\sqrt o$ are among $1,\dots,s-1$, and the sum is at most $s(s-1)+s=o$. If $o$ is not a square and
$s\mid o$, then $o=st$ with $t>s$ (as $o>s^2$), so $o\ge s(s+1)=2(1+\dots+s)$, which is at least the sum, as the divisors
below $\sqrt o$ are among $1,\dots,s$. If $s\nmid o$, the sum is at most
$2(1+\dots+(s-1))=s(s-1)<o$.
\end{proof}
The left-hand side of Lemma~\ref{lem:divsum} is the sequence A117004 of \cite{oeis}.

\begin{proposition}\label{prop:Mlocal}
Let $M$ be the number of essential conditions and let $a=v_p(n)$. Let $\eta\in\mu_{n'}$ have order $k$ and degree
$o=\ord_k(q)$ over $\F_q$. Then
\[
\rho(\eta)/p^a\ \le\ R_a(k)\ \le\ \min\bigl(R^*_a(o),\,f_a(o)\bigr),\qquad f_a(o)\le f(o)\le1 .
\]
Consequently, by Lemma~\ref{lem:Mrho},
\[
M\ \le\ p^a\sum_{\eta\in\mu_{n'}}\min\bigl(R^*_a(\ord_{\ord\eta}(q)),\,f_a(\ord_{\ord\eta}(q))\bigr)\ \le\ n.
\]
Moreover, let $T$ be a finite set of positive integers and let $\ell_T=\lcm(T)$ ($\ell_\emptyset=1$). If
$\vartheta\le\Theta$ are real numbers such that
$\min(R^*_a(o),f_a(o))\le\vartheta$ for $o\notin T$ and $\le\Theta$ for $o\in T$, then
\[
M\ \le\ \vartheta\,n+(\Theta-\vartheta)\,p^a\gcd(n',q^{\ell_T}-1).
\]
\end{proposition}
\begin{proof}
Write $e_d=e_d(\eta)$.
A valid level $d=p^cd'$ divides $n$ (so $c\le a$) and divides $\ord_{e_d}(q)$, which divides $\ord_k(q)$ because $e_d\mid k$.
Hence the valid levels are among the $d$ in the definition of $R_a(k)$, with the same terms $p^{c-a}/\ord_{e_d}(q)$
after division by $p^a$; so $\rho(\eta)/p^a\le R_a(k)\le R^*_a(o)$. For each term, $1/\ord_{e_d}(q)\le1/d$ (as
$d\mid \ord_{e_d}(q)$) and $1/\ord_{e_d}(q)\le\gcd(k,d')/\ord_k(q)\le d/o$: the residue of $q^{\ord_{e_d}(q)}$ modulo $k$
is $\equiv1\pmod{e_d}$, hence so are all its powers, and only $k/e_d=\gcd(k,d')$ residues modulo $k$ are
$\equiv1\pmod{e_d}$; so $q^{\ord_{e_d}(q)}$ has order at most $\gcd(k,d')$ modulo $k$, i.e.\
$\ord_k(q)\le\ord_{e_d}(q)\gcd(k,d')$. With $p^{c-a}\le1$
this gives $R_a(k)\le f_a(o)\le f(o)$, and $f(o)\le1$ is Lemma~\ref{lem:divsum}. Summing over $\eta$ and using Lemma~\ref{lem:Mrho} gives the bounds for
$M$; for the last one note that the $\eta$ with $\ord_{\ord\eta}(q)\in T$ satisfy $\eta^{q^{\ell_T}-1}=1$, so there are at most
$\gcd(n',q^{\ell_T}-1)$ of them, and that this count enters with the coefficient $p^a(\Theta-\vartheta)\ge0$.
\end{proof}

\begin{corollary}\label{cor:Mn}
The number $M(q,n)$ of essential conditions satisfies $M(q,n)\le n$ for every pair $(q,n)$
(Proposition~\ref{prop:Mlocal}).
\end{corollary}

For $q\le4$ the bound $M(q,n)\le n$ is not enough: in the criterion $(*)$ of Theorem~\ref{thm:introC} the term $M$ must be smaller
than about $\frac n2\log_2q\le n$. We therefore show $M(q,n)\le\lambda n$ with $\lambda=0.41$, $0.62$, $0.75$ for $q=2,3,4$
(outside explicit finite sets of $n$), by applying Proposition~\ref{prop:Mlocal} with bounds for $\min(R^*_a(o),f_a(o))$
in three ranges of $o$: for small $o$ and $q=2,3$, exact local maxima $R^*_a(o)$ (over all $k\mid q^o-1$ with
$\ord_k(q)=o$); exact values of $f_a(o)$ on a finite range; and the tail estimate below for large $o$. For $q=4$ and odd
$n$ the bound is proved by hand.

By Proposition~\ref{prop:Mlocal}, $R_a(k)\le f_a(o)$ for every $k$ with $\ord_k(q)=o$, so $R^*_a(o)\le f_a(o)$; we
keep the minimum because $R^*_a(o)$ is computed only for small $o$, where the divisor-sum bound is weak
(Lemma~\ref{lem:divsum} is sharp: $f(o)=1$ for $o=1,2,4,6,12$). For a few $o$, even $R^*_a(o)$ is large (for instance
$R^*_0(1)=1$ for every $q$); these $o$ form the set $T$ of Proposition~\ref{prop:Mlocal}. The corresponding $\eta$ lie
in $\mu_{q^{\ell_T}-1}$, so they add at most the term $(\Theta-\vartheta)\,p^a\gcd(n',q^{\ell_T}-1)$, which in
Proposition~\ref{prop:Msmall} is bounded independently of $n$, as there $T\ne\emptyset$ only for $a\le2$.

\paragraph{Tail estimate.} We use $f_a(o)\le f(o)\le\tau(o)/\sqrt o\le\kappa_s\,\allowbreak o^{1/s-1/2}$. The second inequality
holds as each of the $\tau(o)$ terms of $f(o)$ satisfies $\min(\frac1d,\frac do)\le(\frac1d\cdot\frac do)^{1/2}=o^{-1/2}$,
and the third as $\tau(m)\le\kappa_sm^{1/s}$ with $\kappa_3=(1536/35)^{1/3}\approx3.5273$ and
$\kappa_4=(127401984/25025)^{1/4}\approx8.4470$. The bound for $\tau$ follows from the classical argument, cf.\
\cite[Theorem~315]{hardy2008introduction}: for a prime $\ell\ge2^s$ and $\alpha=v_\ell(m)$, the factor $(\alpha+1)\ell^{-\alpha/s}$ of
$\tau(m)m^{-1/s}$ is at most $1$ (as $\alpha+1\le2^\alpha$), and the finitely many others are maximised directly, so
$\kappa_s^s=\prod_{\ell<2^s}\max_{\alpha\ge0}(\alpha+1)^s\ell^{-\alpha}$. For $s=3,4$ the maximum of $\tau(m)m^{-1/s}$
is attained at $m=2520=2^3\cdot3^2\cdot5\cdot7$ and $m=21621600=2^5\cdot3^3\cdot5^2\cdot7\cdot11\cdot13$, respectively,
both superior highly composite numbers in the sense of \cite[\S32]{ramanujan1915highly}. (Explicit bounds of a different shape,
$\tau(m)\le2^{c\ln m/\ln\ln m}$, are given in \cite{nicolas1983majorations}.)

\paragraph{Finitely many $a$.} For $a\ge v_p(o)$ the restriction $v_p(d)\le a$ in $f_a(o)$ and $R_a(k)$ is vacuous, so $f_a(o)=p^{v_p(o)-a}f_{v_p(o)}(o)$ and
$R_a(k)=p^{v_p(o)-a}R_{v_p(o)}(k)$ decrease in $a$; the maxima over $a\ge a_0$ are therefore attained at some
$a\le\max(a_0,v_p(o))$.

\begin{proposition}\label{prop:Msmall}
The number $M(q,n)$ of essential conditions satisfies the following bounds.
\begin{enumerate}[label=(\alph*),itemsep=1pt]
\item For $q=4$: $M(4,n)\le\frac34n$ if $n$ is even, and $M(4,n)\le\frac23n+1$ if $n$ is odd. Hence $M(4,n)\le\frac34n$ for
all $n\ge12$.
\item For $q=3$: $M(3,n)\le0.5983\,n$ if $3\mid n$; if $3\nmid n$ then $M(3,n)\le0.6\,n+0.4\gcd(n,3^8-1)$, hence
$M(3,n)\le0.62\,n$ for $n\ge131200$.
\item For $q=2$: $M(2,n)\le0.41\,n$ for every $n\ge3780$.
\end{enumerate}
\end{proposition}
\begin{proof}
In each case we bound $\min(R^*_a(o),f_a(o))$ by a constant $\vartheta$ for all $o$ outside a finite set $T$ (possibly
empty) and by a constant $\Theta$ for $o\in T$, using exact values for small $o$ and the tail estimate for large $o$
(for $q=4$ and odd $n$, an estimate by hand for all $o\ge2$), and then apply the last part of Proposition~\ref{prop:Mlocal}; we write $M$ for $M(q,n)$.
Throughout, Proposition~\ref{prop:Mlocal} is applied with $a=v_p(n)$. Since the statements concern all $n$ in the
given ranges, every $a\ge0$ occurs (every $a\ge1$ in the cases with $p\mid n$); by the paragraph \emph{Finitely many $a$}
above, only finitely many values of $a$ need to be computed. The local maxima $R^*_a(o)$ require the
divisors $k$ of $q^o-1$; the Python program factors $q^o-1$ with \texttt{sympy} \cite{sympy2017}.

(a) Here $p=2$. For $n$ even ($a\ge1$) exact evaluation gives $\max_{a\ge1}f_a(o)\le\frac34$ for $o\le12000$, and
$f_a(o)\le f(o)\le\kappa_3o^{-1/6}\le\frac34$ for $o\ge10822$; Proposition~\ref{prop:Mlocal} with $T=\emptyset$
($\vartheta=\Theta=\frac34$) gives $M\le\frac34n$.
For $n$ odd ($a=0$) only odd $d$ occur in $f_0(o)=\sum_{d\mid o,\ d\text{ odd}}\min(\frac1d,\frac do)$, and we show
by hand that $f_0(o)\le\frac23$ for all $o\ge2$. If $o$ is even, the map $d\mapsto o/d$ sends the odd divisors of $o$ injectively to
even divisors, with the same value of $\min(d,o/d)$, so $\sum_{d\mid o,\ d\text{ odd}}\min(d,o/d)\le\frac12\sum_{d\mid o}\min(d,o/d)
\le\frac o2$ by Lemma~\ref{lem:divsum}. If $o\ge3$ is odd, let $s=\lfloor\sqrt o\rfloor$; the divisors below $\sqrt o$ are
odd and at most $s$, and $\sum_{d\mid o}\min(d,o/d)=2\sum_{d\mid o,\,d<\sqrt o}d+[o=s^2]\,s$. For $o=s^2$ ($s\ge3$ odd) this is at
most $2(1+3+\dots+(s-2))+s=\frac{s^2+1}2\le\frac23o$. Otherwise it is at most $2(1+3+\dots+(2\lceil s/2\rceil-1))
=2\lceil s/2\rceil^2\le\frac{(s+1)^2}2$, which is $\le\frac23(s^2+1)\le\frac23o$ for $s\ge6$; for $o<36$ the odd non-squares are
the primes (value $2$) and $15,21,27,33,35$ (values $8,8,8,8,12$), all $\le\frac23o$. Since $f_0(1)=1$,
Proposition~\ref{prop:Mlocal} with $T=\{1\}$, $\ell_T=1$, $\vartheta=\frac23$, $\Theta=1$ gives
$M\le\frac23n+\frac13\gcd(n,3)\le\frac23n+1$. Finally $\frac23n+1\le\frac34n$ for $n\ge12$.

(b) Exact local maxima for $o\le36$ and $0\le a\le4$: $R^*_0(o)\le0.4375$ for $o\notin\{1,2,4,8\}$, and
$R^*_a(o)\le\frac49$ for all $a\ge1$. For $37\le o\le42000$: $f_0(o)\le0.6$ (with equality at $o=40$) and $\max_{a\ge1}f_a(o)\le0.4667$;
for $o>42000$: $f_a(o)\le\kappa_3o^{-1/6}<0.5983$ (this tail bound, not the computed values, determines
$\vartheta$ for $3\mid n$). If $3\mid n$ ($a\ge1$), Proposition~\ref{prop:Mlocal} with $T=\emptyset$
and $\vartheta=\Theta=0.5983$ gives $M\le0.5983\,n$. If $3\nmid n$ ($a=0$, $n'=n$), it applies with $T=\{1,2,4,8\}$,
$\ell_T=8$, $\vartheta=0.6$, $\Theta=1$ (as $f_0(o)\le f(o)\le1$). Finally, for $3\nmid n\ge131200$ we have $\gcd(n,6560)\le6560\le n/20$, so $M\le0.62\,n$.

(c) Exact local maxima for $o\le60$ and $0\le a\le6$ (and $(o,a)=(90,1)$): $R^*_0(o)\le0.4$ for
$o\notin\{1,2,3,6\}$; $R^*_1(o)\le0.375$ for $o\notin\{1,2,6\}$, $R^*_1(o)\le\frac34$ for $o\in\{1,2,6\}$; $R^*_2(o)\le0.375$ for $o\ne4$,
$R^*_2(4)\le\frac7{16}$; $R^*_a(o)\le\frac{15}{64}$ for $a\ge3$. For $61\le o\le340000$: $f_0(o)\le\frac{22}{63}$,
$f_a(o)\le\frac{12}{35}$ for all $a\ge1$ except $(o,a)=(90,1)$ (where $R^*_1(90)\le\frac2{15}$), $\max_{a\ge2}f_a(o)\le\frac{121}{360}$,
$\max_{a\ge3}f_a(o)\le\frac5{16}$; for $o>340000$, $f(o)\le\kappa_4o^{-1/4}<0.35$. The $\eta$ with $\ord_{\ord\eta}(2)\in T$
lie in $\mu_{2^{\ell_T}-1}$, i.e.\ in $\mu_{63}$ ($a\le1$, $\ell_T=6$) and in $\mu_{15}$ ($a=2$, $\ell_T=4$),
and for $n\ge3780=60\cdot63$ we have $p^a\gcd(n',63)\le63\le n/60$ if $a=0$,
$\le126\le n/30$ if $a=1$, and $p^a\gcd(n',15)\le60\le n/63$ if $a=2$. Proposition~\ref{prop:Mlocal} then gives the following bounds, where the last
column is $\vartheta+(\Theta-\vartheta)\,p^a\gcd(n',q^{\ell_T}-1)/n$ with the bound for $p^a\gcd(n',q^{\ell_T}-1)$ just stated:
\begin{center}\small
\begin{tabular}{llllllll}
\toprule
$a$ & $T$ & \begin{tabular}[b]{@{}l@{}}$o\le60$,\\ $o\notin T$\end{tabular} &
\begin{tabular}[b]{@{}l@{}}$61\le o$\\ ${}\le340000$\end{tabular} & $\vartheta$ & $\Theta$ &
\begin{tabular}[b]{@{}l@{}}$p^a\gcd(n',$\\ $q^{\ell_T}-1)$\end{tabular} & bound for $M/n$\\
\midrule
$0$ & $\{1,2,3,6\}$ & $0.4$ & $22/63$ & $0.4$ & $1$ & $\le n/60$ & $0.4+0.6/60=0.41$\\
$1$ & $\{1,2,6\}$ & $0.375$ & $12/35$ & $0.375$ & $3/4$ & $\le n/30$ & $0.375+0.375/30=0.3875$\\
$2$ & $\{4\}$ & $0.375$ & $121/360$ & $0.375$ & $7/16$ & $\le n/63$ & $0.375+0.0625/63<0.3760$\\
$\ge3$ & $\emptyset$ & $15/64$ & $5/16$ & $0.35$ & -- & -- & $0.35$\\
\bottomrule
\end{tabular}
\end{center}
Here $\vartheta$ bounds $\min(R^*_a(o),f_a(o))$ for $o\notin T$ (by the values listed for $o\le60$, the bounds for
$61\le o\le340000$, and $f(o)<0.35$ beyond), and $\Theta$ bounds it for $o\in T$ (by the values listed for $o\le60$; for
$a=0$, $\Theta=1$ because $f_0(o)\le f(o)\le1$). Hence $M\le0.41\,n$.
\end{proof}

\section{Structure of components}\label{sec:struct}
The criterion of Theorem~\ref{thm:introC} needs a bound $W<1$, uniform in $n$, for the weight of every component
without a condition of codimension~$1$; this section proves one (Theorem~\ref{thm:W}, stated in a slightly weaker form in the introduction as
Theorem~\ref{thm:introW}). The key fact is that the fibres of the conditions of a
component with minimal codimension~$o_0$ lie in $\mu_K$ for a $K$ depending only on $q$ and $o_0$
(Theorem~\ref{thm:struct}). This bounds the weight (Corollary~\ref{cor:weight}) and shows that the component already
occurs for a pair $(q,n_0)$ with $n_0$ dividing an explicit number (Proposition~\ref{prop:finite}); for the four pairs $(q,o_0)$ where
both bounds of Corollary~\ref{cor:weight} are too weak, the weights can therefore be bounded by a finite computation.

From now on we write a condition $(d,k)$ as $(d,\xi)$, where $\xi=\zeta^{d'k}$ is its target (defined up to
$\xi\mapsto\xi^{q^d}$; see Definition~\ref{def:cond} and the paragraph after it).
Throughout this section, $C_0=(d_0,\xi_0)$ is an essential condition of codimension $o_0$, and $e_0=\ord\xi_0$ (so
$e_0\mid q^{o_0}-1$; moreover $d_0\mid o_0$, hence $d_0'\mid o_0'$). Let $\cC$ be the component containing $C_0$ (if $o_0\ge2$,
then $\cC$ contains no condition of codimension $1$, by Lemma~\ref{lem:qorbits}(c)). Let
$Y\subseteq\mu_{n'}$ be the set of $\eta$ occurring in the fibres of the conditions of $\cC$ (a union of $q$-orbits, by
Lemma~\ref{lem:qorbits}(a)). Let $D_{\cC}$ be the least common multiple of the $p$-free parts $d'$ of the levels of the
conditions of $\cC$, and let $N_{\cC}(o)$ be the number of conditions of $\cC$ of codimension $o$.

The hypotheses of Theorem~\ref{thm:struct} below ensure that $\mu_K$ contains the fibre of $C_0$ (which lies in
$\mu_{d_0'e_0}$) and is stable under the two moves $\eta\mapsto\eta^q$ and $\eta\mapsto\eta\varepsilon$ by which a
component is explored (Section~\ref{subsec:idea}). Part~(b) of the theorem also uses the following elementary lemma; it
gives the divisibility $d'\ord\xi\mid K$ on which Proposition~\ref{prop:finite} rests.

\begin{lemma}\label{lem:coset}
In a finite cyclic group, a coset $\eta H$ of the subgroup $H$ of order $d$ contains an element of order
$\lcm(\ord\eta,d)$. In particular, the fibre $\{\eta:\eta^{d'}=\xi\}\subseteq\mu_{n'}$ of a condition contains an
element of order $d'\ord\xi$.
\end{lemma}
\begin{proof}
The subgroup $\langle\eta\rangle H$ is cyclic of order $m=\lcm(\ord\eta,d)$ and contains $\eta H$; identify it with
$\Z/m$ (orders of elements do not change), so that $\eta=x$ ($0\le x<m$) and $H=(m/d)\Z/m$; we need $h\in H$ such that $x+h$
is a unit. Let $t$ be the product of the primes $\ell\mid m$ with $\ell\nmid x$, and put $h=t\,m/d$. Let $\ell$ be a prime
dividing $m$. If $\ell\nmid x$, then $\ell\mid h$, so $\ell\nmid x+h$. If $\ell\mid x$, then $\ell\nmid t$, and
$\ell\nmid m/d$, since otherwise $v_\ell(\ord x)=v_\ell(m)$ and hence $\ell\nmid x$; so $\ell\nmid h$ and again
$\ell\nmid x+h$. Hence $x+h$ is a unit. For the fibre take $\eta$ with $\eta^{d'}=\xi$ and $H=\mu_{d'}$; then
$\lcm(\ord\eta,d')=d'\cdot\ord\eta/\gcd(\ord\eta,d')=d'\ord\xi$.
\end{proof}

\begin{theorem}[structure of components]\label{thm:struct}
Let $K$ be a $p$-free integer with
\[
\text{(i)}\ \ d_0'e_0\mid K\qquad\text{and}\qquad\text{(ii)}\ \ (\ord_K(q))_{p'}\mid K .
\]
Then:
\begin{enumerate}[label=(\alph*),itemsep=1pt]
\item every $\eta\in Y$ has order dividing $K$;
\item every condition $(d,\xi)$ of $\cC$ satisfies $\ord\xi\mid K$, $d\mid \ord_K(q)$ and $d'\ord\xi\mid K$; in particular
all codimensions occurring in $\cC$ divide $\ord_K(q)$, and $D_{\cC}\mid(\ord_K(q))_{p'}$;
\item for every $o$, with $D=D_{\cC}$ or any multiple of it,
\[
N_{\cC}(o)\ \le\ \sum_{\substack{d\mid o\\ d'\mid D}}d\cdot\gcd\bigl(q^o-1,\,D/d'\bigr)\ \le\ \sigma_1(o)\gcd(q^o-1,D).
\]
\end{enumerate}
\end{theorem}
\begin{proof}
We first prove (a) by following a chain of conditions of $\cC$; (b) then follows from (a), essentiality and
Lemma~\ref{lem:coset}, and (c) counts, for each codimension $o$, the possible targets by means of an invariant that is
preserved by the two moves that explore $\cC$.

(a) Order the conditions of $\cC$ so that the fibre of each one meets a $q$-orbit met by the fibre of an earlier one,
starting with $C_0$
(possible since $\cC$ is connected, Lemma~\ref{lem:qorbits}(a)), and let $Y_j$ be the union of the $q$-orbits of the
fibres of the first $j$ conditions. The fibre of $C_0$ consists of $\eta$ with $\eta^{d_0'}=\xi_0$, of order dividing
$d_0'e_0\mid K$, and $\eta\mapsto\eta^q$ preserves $\mu_K$. Suppose $Y_j\subseteq\mu_K$ and let $(d,\xi)$ be the next
condition. Its fibre contains an element $\eta$ of a $q$-orbit contained in $Y_j$, so $\eta\in Y_j$ and the fibre is
$\eta\mu_{d'}$ with $e:=\ord\xi=\ord(\eta^{d'})\mid\ord\eta\mid K$. Essentiality gives $d'\mid d\mid \ord_e(q)\mid \ord_K(q)$, hence
$d'\mid(\ord_K(q))_{p'}\mid K$ by (ii), and $\eta\mu_{d'}\subseteq\mu_K$.

(b) The target is $\xi=\eta^{d'}$ with $\eta\in Y$, so $e=\ord\xi\mid K$ and $d\mid \ord_e(q)\mid \ord_K(q)$. By
Lemma~\ref{lem:coset} the fibre contains an element of order $d'e$, which lies in $Y$, so $d'e\mid K$.

(c) Fix $\eta_0$ in the fibre of $C_0$. Every element of $Y$ is reached from $\eta_0$ by the moves $\eta\mapsto\eta^q$
and $\eta\mapsto\eta\varepsilon$ with $\varepsilon\in\mu_{d'}$, $d'$ the $p$-free part of a level of $\cC$. Indeed, list the
conditions of $\cC$ as in (a), so that each one has a fibre meeting the $q$-orbit of a fibre element of an earlier one
(Lemma~\ref{lem:qorbits}(a)). Inside a fibre $\eta\mu_{d'}$ any two elements differ by some
$\varepsilon\in\mu_{d'}$ (Lemma~\ref{lem:qorbits}(b)); from a fibre element of an earlier condition, $q$-th powers lead to an
element of the next fibre, and $\varepsilon$-moves then lead to all of it; finally $q$-th powers reach the whole $q$-orbits,
whose union is $Y$. Both moves
preserve the $q$-orbit of $\eta^{D}$, since $\varepsilon^{D}=1$ (as $d'\mid D_{\cC}\mid D$). Put $y=\eta_0^{D}$ and
$[y]=\{y^{q^i}\}$; then $\eta^{D}\in[y]$ for all $\eta\in Y$.

Now fix $o$ and a level $d=p^cd'$ with $d\mid o$ and $d'\mid D$ (by essentiality and the definition of $D_{\cC}$, every level of a
condition of $\cC$ of codimension $o$ is of this form). A condition of $\cC$ of codimension $o$ at level $d$ has a target
$x=\eta^{d'}$ with $\eta\in Y$, hence
\[
x\in G_o:=\mu_{n'}\cap\mu_{q^o-1}\quad(\text{as }\ord_{\ord x}(q)=o)\qquad\text{and}\qquad x^{D/d'}=\eta^{D}\in[y].
\]
The kernel of $x\mapsto x^{D/d'}$ on the cyclic group $G_o$ has order $\gcd(|G_o|,D/d')\le\gcd(q^o-1,D/d')$. If there
is at least one such target $x$, then $y':=x^{D/d'}\in[y]$ satisfies $\ord y'\mid\ord x$, so $|[y]|=\ord_{\ord y'}(q)$ divides
$\ord_{\ord x}(q)=o$. Therefore the number of targets of conditions of $\cC$ of codimension $o$ at level $d$ is at
most
\begin{align*}
\#\{x\in G_o:\ x^{D/d'}\in[y]\}\ &\le\ \sum_{y'\in[y]}\#\{x\in G_o:\ x^{D/d'}=y'\}\\
&\le\ |[y]|\,\gcd(q^o-1,D/d')\ \le\ o\,\gcd(q^o-1,D/d').
\end{align*}
By Lemma~\ref{lem:weightcodim} each condition at level $d$ with codimension $o$ is an orbit of exactly $\ord_{\ord x}(q^d)=o/d$
targets, so there are at most $\frac do\cdot o\gcd(q^o-1,D/d')=d\gcd(q^o-1,D/d')$ such conditions. Summing over the
levels gives the first inequality; the second follows from $\gcd(q^o-1,D/d')\le\gcd(q^o-1,D)$ and
$\sum_{d\mid o}d=\sigma_1(o)$.
\end{proof}

\begin{remark}
Theorem~\ref{thm:struct} does not use that $o_0$ is the \emph{minimal} codimension in $\cC$: the conclusions hold for the
component of any essential condition, with $K$ built from that condition. Minimality is used only
in Corollary~\ref{cor:weight} (to restrict the sum to $o\ge o_0$). Part~(c) does not involve $K$; in Corollary~\ref{cor:weight},
$K$ enters only through~(b), which bounds the codimensions and $D_{\cC}$. The hypothesis that $K$ is $p$-free only ensures
that $\ord_K(q)$ is defined. Section~\ref{sec:4599} works out the components with minimal
codimension $9$ for $(q,n)=(2,4599)$; for them both bounds in (c), with $D=D_{\cC}=9$, are attained for $o=9$.
\end{remark}

\paragraph{Admissible $K$.} Put $K_0=o_0'(q^{o_0}-1)$, a $p$-free multiple of $d_0'e_0$ that depends only on $q$ and $o_0$.
We call a $p$-free multiple $K$ of $K_0$ \emph{admissible} if it satisfies condition~(ii) of Theorem~\ref{thm:struct}; then
Theorem~\ref{thm:struct} applies with this $K$ to every essential condition of codimension $o_0$, for every $n$. Two
admissible choices are useful: (A) is explicit and is used in Theorem~\ref{thm:W} for $q\ge4$ and for $q=3$ with $o_0\ge375$;
(B) is the least admissible $K$ (see below) and is used for $q=2$ (where the bounds obtained from (A) are too large, e.g.\
larger than $1$ for $o_0=4$), for $q=3$ with $o_0\le374$, and in Proposition~\ref{prop:finite}.
\begin{enumerate}[label=(\Alph*),itemsep=1pt]
\item \emph{The crude choice.} Let $L_0=\lcm(1,\dots,o_0)_{p'}$. Then $K=\lcm(K_0,L_0)$ satisfies (i), and by
Lemma~\ref{lem:crude} below it satisfies (ii) and $\ord_K(q)\mid o_0'\lcm(1,\dots,o_0)$.
\item \emph{The closure.} Let
\[
O_0=\lcm\bigl(o_0,\ \ord_{\ell^{v_\ell(K_0)}}(q):\ \ell\mid o_0'\bigr),\qquad\text{where }v_\ell(K_0)=v_\ell(o_0')+v_\ell(q^{o_0}-1),
\]
so that $\ord_{K_0}(q)\mid O_0$. Put $D_0=(O_0)_{p'}$ and $D_{i+1}=\lcm(D_i,(\ord_{D_i}(q))_{p'})$. This sequence
stabilises (see the paragraph after Lemma~\ref{lem:closuresize} below); let $D_\infty$ be its eventual value. Then $K=\lcm(K_0,D_\infty)$ satisfies (i) and (ii), with
$(\ord_K(q))_{p'}\mid D_\infty$ and $\ord_K(q)\mid O_\infty:=\lcm(O_0,\ord_{D_\infty}(q))$. Computing $K$ needs no
factorisation of $q^{o_0}-1$, only $v_\ell(q^{o_0}-1)$ and $\ord_{\ell^{v_\ell(K_0)}}(q)$ for $\ell\mid o_0'$, and orders of
$q$ modulo the numbers $D_i$, for which $\log_2D_i=O((\log o_0)^2)$ by Lemma~\ref{lem:closuresize}.
\end{enumerate}
\begin{lemma}\label{lem:crude}
For $K=\lcm(K_0,L_0)$ as in (A): $(\ord_K(q))_{p'}\mid K$ and $\ord_K(q)\mid o_0'\lcm(1,\dots,o_0)$.
\end{lemma}
\begin{proof}
Recall that for a prime $r\nmid q$ and $\alpha\ge1$ the order $\ord_{r^\alpha}(q)$ divides $|(\Z/r^\alpha)^\times|$ and
$\ord_{r^\alpha}(q)/\ord_r(q)$ is a power of $r$; recall also the lifting lemma: if $q^m\equiv1\pmod{r^v}$ with $v\ge1$ then
$q^{mr}\equiv1\pmod{r^{v+1}}$ (expand $(1+r^vt)^r$ by the binomial theorem). Every prime $r\mid K$ divides $q^{o_0}-1$ (then $\ord_r(q)\mid o_0$) or satisfies $r\le o_0$
(then $\ord_r(q)\le r-1<o_0$); in both cases every prime power dividing $\ord_r(q)$ is at most $o_0$.

\emph{First claim.} Let $\ell\ne p$ and $r^\alpha\,\|\,K$. If $r\ne\ell$ then
\[
v_\ell(\ord_{r^\alpha}(q))=v_\ell(\ord_r(q))\le v_\ell(\lcm(1,\dots,o_0))=v_\ell(L_0)\le v_\ell(K);
\]
if $r=\ell$ then
$v_\ell(\ord_{\ell^\alpha}(q))\le\alpha-1<v_\ell(K)$. So $v_\ell(\ord_K(q))\le v_\ell(K)$.

\emph{Second claim.} Let $\ell^v\,\|\,K$, so $v=\max(v_\ell(o_0')+v_\ell(q^{o_0}-1),\allowbreak \,\allowbreak v_\ell(L_0))$. If
$v_\ell(q^{o_0}-1)\ge1$ and $v=v_\ell(o_0')+v_\ell(q^{o_0}-1)$, the lifting lemma gives
$\ord_{\ell^v}(q)\mid\ord_{\ell^{v_\ell(q^{o_0}-1)}}(q)\,\allowbreak \ell^{v_\ell(o_0')}\mid o_0\,\allowbreak \ell^{v_\ell(o_0')}$, which divides
$o_0'\lcm(1,\dots,o_0)$. Otherwise $v=v_\ell(L_0)$ (note $v_\ell(o_0')\le v_\ell(L_0)$), so $\ord_{\ell^v}(q)<\ell^v\le o_0$, and hence
$\ord_{\ell^v}(q)$ divides $\lcm(1,\dots,o_0)$. Since $\ord_K(q)$ is the least common multiple of these orders,
$\ord_K(q)\mid o_0'\lcm(1,\dots,o_0)$.
\end{proof}

The following lemma shows that the closure (B) terminates; its bounds are also used in step~(ii) of the proof of
Theorem~\ref{thm:W} ($q=2$, $o_0\ge4096$).

\begin{lemma}\label{lem:closuresize}
Let $P$ be the smallest set of primes that contains the prime divisors of $D_0$ and, with every $r\in P$, the prime
divisors of $r-1$. Then $D_0\le o_0^3$, $2^{|P|}\le D_0^2$ and $D_i\le D_0^{|P|+1}$ for every $i$; in particular
$\log_2D_i\le\log_2D_0+2(\log_2D_0)^2$ for every $i$. All these bounds hold for $D_\infty$.
\end{lemma}
\begin{proof}
We first bound $O_0$ and hence $D_0$, and then the primes and the exponents that can occur in the numbers $D_i$.
For $\ell\mid o_0'$: if $\ord_\ell(q)\mid o_0$ then $\ord_{\ell^{v_\ell(K_0)}}(q)\mid o_0\ell^{v_\ell(o_0')}$ by the lifting lemma
(recalled in the proof of Lemma~\ref{lem:crude}), otherwise $v_\ell(q^{o_0}-1)=0$ and
$\ord_{\ell^{v_\ell(K_0)}}(q)\mid\varphi(\ell^{v_\ell(o_0')})$. Hence
$O_0$ divides $o_0\,o_0'\,\varphi(o_0')$, so $O_0\le o_0^3$. Every prime of $D_{i+1}$ divides $D_i$ or $\ell-1$ for a prime $\ell\mid D_i$
(since $\ord_{D_i}(q)$ divides $\lcm_{\ell^\alpha\|D_i}\ell^{\alpha-1}(\ell-1)$), so the primes of every $D_i$ lie in $P$.
For a single prime $r$ the corresponding set (the primes of the Pratt tree of $r$) has at
most $2\log_2r$ elements; this follows from the bound $2\log_2r-1$ for the number of nodes of the Pratt tree
\cite[(1.5)]{ford2010pratt}, and also by induction on $r$: true for $r=2,3$, and for odd $r$ it is at most
$1+1+\sum_{\ell\mid r-1,\allowbreak \ \ell\ \rm odd}2\log_2\ell\le2+2\log_2\frac{r-1}2<2\log_2r$. Hence $|P|\le2\log_2D_0$,
i.e.\ $2^{|P|}\le D_0^2$, and all elements of $P$ are at most $D_0$. For the exponents: $v_\ell(\ord_{D_i}(q))\le\max(v_\ell(D_i)-1,\ \max_{r\mid D_i}v_\ell(r-1))$,
so by induction on $i$ every prime of $D_i$ lies in $P$ and $v_\ell(D_i)\le\max(v_\ell(D_0),\max_{r\in P}v_\ell(r-1))$;
hence $D_i$ divides $D_0\prod_{r\in P}(r-1)$ (bound the maximum by the sum), so $D_i\le D_0\prod_{r\in P}r\le D_0^{|P|+1}$
for every $i$, and $|P|\le2\log_2D_0$ gives the bound for $\log_2D_i$.
\end{proof}

\emph{Proof of the claims in (B).} First, $\ord_K(q)=\lcm(\ord_{K_0}(q),\ord_{D_\infty}(q))\mid O_\infty$ and $(\ord_K(q))_{p'}\mid\lcm(D_0,\allowbreak (\ord_{D_\infty}(q))_{p'})\mid D_\infty\mid K$
(as $D_0\mid D_\infty$, and $(\ord_{D_\infty}(q))_{p'}\mid D_\infty$ since the sequence is stable at $D_\infty$).
Here $\ord_{K_0}(q)\mid O_0$ because $\ord_{K_0}(q)$ is the least common multiple of the orders of $q$ modulo the prime powers
$\ell^{v_\ell(K_0)}\,\|\,K_0$: for $\ell\mid o_0'$ this order is one of the
terms of $O_0$; for $\ell\nmid o_0'$ the prime power divides $q^{o_0}-1$, so the order divides $o_0$. The sequence $D_i$ is
increasing with respect to divisibility and bounded (Lemma~\ref{lem:closuresize}), so it stabilises and $D_\infty$ is well defined.
More precisely, each strict step at least doubles $D_i$, so by Lemma~\ref{lem:closuresize} the sequence is constant
for $i\ge2(\lfloor\log_2D_0\rfloor+1)^2+1$.

In fact $O_0=\ord_{K_0}(q)$, $O_\infty=\ord_K(q)$ and $D_\infty=(\ord_K(q))_{p'}$, and $K$ is the least admissible $K$ (it
divides every admissible $K$); this is not needed for the proof, but it identifies $K$ with the
number described in Section~\ref{subsec:idea}. Indeed, $o_0=\ord_{q^{o_0}-1}(q)$ and the other terms of $O_0$ are orders of
$q$ modulo divisors of $K_0$, so $O_0\mid\ord_{K_0}(q)$; hence $O_0=\ord_{K_0}(q)$ and
$\ord_K(q)=\lcm(\ord_{K_0}(q),\ord_{D_\infty}(q))=O_\infty$. As $K_0$ and every $D_i$ divide $K$, induction on $i$ gives $D_i\mid(\ord_K(q))_{p'}$, so
$D_\infty=(\ord_K(q))_{p'}$. If $K'$ is admissible, then $D_0=(\ord_{K_0}(q))_{p'}$ divides $(\ord_{K'}(q))_{p'}\mid K'$,
and $D_i\mid K'$ implies $(\ord_{D_i}(q))_{p'}\mid(\ord_{K'}(q))_{p'}\mid K'$; so every $D_i$ divides $K'$, and $K\mid K'$.
Similarly, the iteration $J\mapsto\lcm(J,(\ord_J(q))_{p'})$ started at $J_0=K_0$ satisfies $J_{i+1}=\lcm(K_0,D_i)$ (by
induction, as $\ord_{\lcm(K_0,D_i)}(q)=\lcm(O_0,\ord_{D_i}(q))$ and $D_0\mid D_i$), so it becomes stationary at $K$.

\begin{corollary}\label{cor:weight}
Let $\cC$ be a component without codimension-one conditions and let $o_0$ be its \emph{minimal} codimension. With
$D_\infty$ and $O_\infty$ from (B),
\[
w_{\cC}=\sum_{C'\in\cC}q^{-o(C')}\ \le\ \Wcl(q,o_0):=\sum_{\substack{o\mid O_\infty\\ o\ge o_0}}q^{-o}\!\!\sum_{\substack{d\mid o\\ d'\mid D_\infty}}d\gcd(q^o-1,D_\infty/d').
\]
Using (A) instead of (B) one obtains $w_{\cC}\le\Wcr(q,o_0)$, where $\Wcr$ is defined by the same expression with
$O_\infty$ and $D_\infty$ replaced by $\Lambda=o_0'\lcm(1,\dots,o_0)$ and $\Lambda_{p'}$ (the subscripts stand for
\emph{closure} and \emph{crude}).
\end{corollary}
\begin{proof}
Apply Theorem~\ref{thm:struct} with $C_0$ a condition of $\cC$ of codimension $o_0$: every codimension $o$ in $\cC$ divides
$\ord_K(q)\mid O_\infty$ (resp.\ $\Lambda$) and is $\ge o_0$, and $D_{\cC}\mid(\ord_K(q))_{p'}\mid D_\infty$ (resp.\ $\Lambda_{p'}$); use (c).
\end{proof}

For four pairs $(q,o_0)$ both bounds of Corollary~\ref{cor:weight} exceed the value $W(q)$ of Theorem~\ref{thm:W} (for
instance $\Wcl(3,2)=\Wcr(3,2)=\frac{28}{27}$), and the weights are computed instead (step~(i) of the proof of
Theorem~\ref{thm:W}). This is a finite computation because of the following reduction, which rests on
Theorem~\ref{thm:struct}(b) for the closure (B).

\begin{proposition}\label{prop:finite}
Let $o_0\ge2$, and put $K=\lcm(K_0,D_\infty)$ and $N^{**}=\lcm(K,O_\infty)$. For every $n$, the components with minimal
codimension $o_0$ for $(q,n)$ coincide, as sets of conditions $(d,\xi)$ with their codimensions, with those for
$(q,\gcd(n,N^{**}))$. Hence every component for $(q,n)$ with minimal codimension $o_0$ has the same conditions and
weight as a component for $(q,n_0)$ with $n_0\mid N^{**}$.
\end{proposition}
\begin{proof}
The idea is that, by Theorem~\ref{thm:struct}(b), every condition of a component with minimal codimension $o_0$ has a
level $d$ and a target $\xi$ with $d\mid N^{**}$ and $d'\ord\xi\mid N^{**}$, so that it is already a condition for
$\gcd(n,N^{**})$.
For all exponents we use the roots of unity in the fixed algebraic closure $\overline{\F}_q$, so that
$\mu_{n_0'}\subseteq\mu_{n'}$ for $n_0\mid n$, and we describe a condition by its level and its target, $(d,\xi)$ with
$\xi\in\mu_{n'/d'}$, up to $\xi\mapsto\xi^{q^d}$; by Theorem~\ref{thm:criterion}, Definition~\ref{def:cond} and the
paragraph after it, this
description, the fibre $\{\eta:\eta^{d'}=\xi\}$ (up to $\eta\mapsto\eta^{q^d}$) and the codimension $\lcm(d,\ord_{\ord\xi}(q))$
do not depend on the exponent. Write $n_0=\gcd(n,N^{**})$. A pair $(d,\xi)$ is a condition for $n$ iff $d\mid n$ and $d'\ord\xi\mid n'$ (then
$\xi\in\mu_{n'/d'}$), and whether it is essential does not depend on $n$; its fibre lies in $\mu_{d'\ord\xi}$, so the fibre is
the same subset of $\overline{\F}_q$ for $n$ and $n_0$ whenever the pair is a condition for both. Every condition for $n_0$ is one
for $n$. Let $C_0$ be an essential condition for $n_0$ of codimension $o_0$ and let $\cC$ and $\cC^{(0)}$ be its components for $n$ and for $n_0$, respectively.
By Theorem~\ref{thm:struct}(b), applied to $C_0$ for the exponent $n$, every $(d,\xi)\in\cC$ satisfies $d\mid \ord_K(q)\mid N^{**}$ and
$d'\ord\xi\mid K\mid N^{**}$, hence $d\mid n_0$ and $d'\ord\xi\mid n_0'$, so $\cC$ consists of essential conditions for $n_0$.
Since the component relation is read off from the fibres and their $q$-orbits in the same way for $n$ and for $n_0$
(Lemma~\ref{lem:qorbits}(a), applied to both exponents; a chain for $n$ between members of $\cC$ stays in $\cC$), the set $\cC$ is
connected for $n_0$, so $\cC\subseteq\cC^{(0)}$; and $\cC^{(0)}\subseteq\cC$ because every essential condition for $n_0$ is
an essential condition for $n$ with the same fibre. So $\cC=\cC^{(0)}$.
A component for $n$ with minimal codimension $o_0$ contains a condition $C_0$ of codimension $o_0$, which by the same
argument (Theorem~\ref{thm:struct}(b) for the pair $(d_0,\xi_0)$) is a condition for $n_0$; conversely a component for $n_0$
with minimal codimension $o_0$ contains such a $C_0$. In both cases the two components of $C_0$ coincide.
\end{proof}

\begin{theorem}[weights]\label{thm:W}
Every component without a condition of codimension~$1$ has weight at most $W(q)$, where
\begin{gather*}
W(2)=\tfrac34,\qquad W(3)=\tfrac49,\qquad W(4)=0.2946,\qquad W(5)=0.3392,\qquad W(7)=0.1683,\\
W(q)\le0.1006\ \ (8\le q\le31),\qquad W(q)=0.0118\ \ (q\ge32),
\end{gather*}
and, individually for $8\le q\le31$,
\begin{alignat*}{4}
W(8)&=0.0469,&\qquad W(9)&=0.1006,&\qquad W(11)&=0.0670,&\qquad W(13)&=0.0478,\\
W(16)&=0.0118,&\qquad W(17)&=0.0279,&\qquad W(19)&=0.0223,&\qquad W(23)&=0.0152,\\
W(25)&=0.0129,&\qquad W(27)&=0.0110,&\qquad W(29)&=0.0096,&\qquad W(31)&=0.0084.
\end{alignat*}
(As $W(q)\le0.3392$ for $q\ge4$, this contains Theorem~\ref{thm:introW}.)
\end{theorem}
\begin{proof}
Let $\cC$ be a component without a condition of codimension~$1$, and let $o_0\ge2$ be its minimal codimension. By
Corollary~\ref{cor:weight}, $w_{\cC}\le\Wcl(q,o_0)$ and $w_{\cC}\le\Wcr(q,o_0)$, bounds that depend only on $q$ and $o_0$.
So it suffices to show that one of them is at most $W(q)$, for every $q$ and every $o_0\ge2$, except for the four pairs
$(q,o_0)$ of step~(i) below, where both are larger and the weights are computed using Proposition~\ref{prop:finite}.
Steps (ii), (iii), (iv) and (vi) treat the remaining $o_0$ for $q=2$, $q=3$, $4\le q\le31$ and $q\ge32$, respectively, and
step~(v) treats large $o_0$ for all $q\ge3$. For $4\le q\le31$ and for $q\ge32$, the value $W(q)$ is the largest bound of step~(iv) and of step~(vi), respectively,
rounded up to four decimal places.

\emph{Evaluation of $\Wcl$ and $\Wcr$.} Both have the form
$\Sigma(O,D)=\sum_{o\mid O,\,o\ge o_0}q^{-o}I_D(o)$ with $I_D(o)=\sum_{d\mid o,\,d'\mid D}d\gcd(q^o-1,D/d')$, where
$(O,D)=(O_\infty,D_\infty)$ resp.\ $(\Lambda,\Lambda_{p'})$. Since $I_D(o)$ increases when $D$ is replaced by a multiple,
$\Sigma(O,D)$ increases when $O$ and $D$ are replaced by multiples. Moreover $I_D(o)\le\sigma_1(o)D\le o^2D$, and for
$m\ge7$ the ratio of consecutive terms of $\sum_{o\ge m}o^2q^{-o}$ is at most $(8/7)^2/2<\frac23$, so
$\sum_{o\ge m}o^2q^{-o}\le3m^2q^{-m}$. With the cut-off $m=o_0+41$ we therefore use
\[
\Sigma(O,D)\ \le\ \sum_{\substack{o\mid O\\ o_0\le o\le o_0+40}}q^{-o}I_D(o)\ +\ 3m^2D\,q^{-m}
\]
(the last term may be dropped if $O<m$), evaluated in exact rational arithmetic.

(i) \emph{The pairs $(q,o_0)\in\{(2,3),(2,6),(3,2),(3,4)\}$.} Here $\Wcl(q,o_0)=25/32$,
$227921/2^{18}\approx0.8694$, $28/27$, $25929152/3^{16}\approx0.6023$ are larger than $W(q)$ (these values of $\Wcl$, like
those quoted in steps (ii) and (iii), are given for information; they are not needed for the proof), and we compute exactly. One finds $N^{**}=42$, $378$, $16$, $320$ (for the argument it suffices
that $N^{**}$ divides these numbers). By
Proposition~\ref{prop:finite} it suffices to consider
$n_0\mid N^{**}$; for each such $n_0$ the essential conditions are listed by Lemma~\ref{lem:weightcodim}. Call a map on
$\mu_{n_0'}$ a \emph{labelling} if it is constant on $q$-orbits and on the fibre of every essential condition. By
Lemma~\ref{lem:qorbits}(a), the fibres of all conditions of a component then carry the same label, so the weight of a
component without a condition of codimension $1$ is at most the total weight of the essential conditions of
codimension~$\ge2$ whose fibres carry its label. The components themselves give a labelling (by
Lemma~\ref{lem:qorbits}(a)), for which these totals are the weights; the Python program uses it. For every
$n_0\mid378$ ($q=2$; this covers $N^{**}=42$ and $378$, as $42\mid378$), resp.\ every $n_0\mid320$ ($q=3$; this covers
$N^{**}=16$ and $320$), a labelling is found for which all these totals are at most $\frac34$, resp.\ $\frac49$; hence every
component with minimal codimension $o_0$, for the four pairs $(q,o_0)$ of this step, has weight at most $W(q)$.
(For information, the exact maxima for the four minimal codimensions are $35/64$,
$9/32$, $4/9$ and $4/27$, attained at $n_0=42,126,8,20$, and the same computation, with $n_0$ running through the
divisors of $N^{**}=6,60,620,5334,2040$, resp.\ $78,2420,4368$, gives the
maxima $3/4$, $7/16$, $196615/2^{20}\approx0.18751$, $140668567555/2^{41}\approx0.06397$,
$15/256$ for $q=2$, $o_0=2,4,5,7,8$, and $4/27$, $57513730/3^{19}\approx0.04948$, $16/729$ for
$q=3$, $o_0=3,5,6$; in particular, $W(2)=3/4$ and $W(3)=4/9$ are attained, for $(q,n)=(2,6)$ and $(3,8)$. None
of these exact values is needed for the proof.)

(ii) \emph{The case $q=2$.} For $2\le o_0\le4095$, $o_0\ne3,6$: $\Wcl(2,o_0)\le\frac34$, and $\Wcl(2,o_0)\le0.032$ for $o_0\ge9$,
by the evaluation above; for instance $\Wcl(2,2)=\frac34$ and $\Wcl(2,4)=\frac7{16}$. For
$o_0\ge4096$ put $s=\lfloor\log_2o_0\rfloor\ge12$. By Lemma~\ref{lem:closuresize}, $D_0\le o_0^3<2^{3(s+1)}$ and
$2^{|P|}\le D_0^2<2^{6(s+1)}$, so $|P|+1\le6(s+1)$ and $D_\infty\le D_0^{|P|+1}\le2^{18(s+1)^2}$. The evaluation with
$m=o_0$ (no exact terms) gives $\Wcl(2,o_0)\le3o_0^2D_\infty2^{-o_0}$, and
$128\,o_0^2D_\infty\le2^{18(s+1)^2+2(s+1)+7}\le2^{2^s}\le2^{o_0}$, because $18(s+1)^2+2(s+1)+7\le2^s$ for $s\ge12$ (by
induction on $s$); hence $\Wcl(2,o_0)\le\frac3{128}<0.032$.

(iii) \emph{The case $q=3$.} For $3\le o_0\le374$, $o_0\ne4$: $\Wcl(3,o_0)\le\frac49$, and $\Wcl(3,o_0)\le0.051$ for $o_0\ge7$, by the
evaluation above (for instance $\Wcl(3,3)=\frac4{27}$). For
$o_0\ge375$: $\Wcr(3,o_0)\le0.051$ by (v).

(iv) \emph{The case $4\le q\le31$.} For $2\le o_0<o_1(q)$: $\Wcr(q,o_0)\le W(q)$ by the evaluation above, where $o_1(q)=39$, $20$,
$12$, $11$, $9$, $8$ for $q=4$, $5$, $7$, $8$, $9$, $11$ and $o_1(q)=7$ for $q\ge13$. For $o_0\ge o_1(q)$: (v).

(v) \emph{Large $o_0$.} By the bound $\log\lcm(1,\dots,\lfloor x\rfloor)<1.03883\,x$ ($x>0$) of Rosser and Schoenfeld
\cite[Theorem~12]{rosser1962approximate} (the left-hand side is Chebyshev's function) and $\exp(1.03883)<2.83$, we have
$\Lambda_{p'}\le\Lambda\le o_0\lcm(1,\dots,o_0)\le o_0\,2.83^{o_0}$,
and the evaluation with $m=o_0\ge7$ gives
\[
\Wcr(q,o_0)\ \le\ 3o_0^2\Lambda_{p'}q^{-o_0}\ \le\ T_q(o_0):=3o_0^3\,(2.83/q)^{o_0}.
\]
Since $T_q(o+1)/T_q(o)=(1+\frac1o)^3\,2.83/q$ decreases in $o$, we have $T_q(o_0)\le W$ for all $o_0\ge o_1$ as soon as
$T_q(o_1)\le W$ and $(1+\frac1{o_1})^3\,2.83\le q$. This is checked for $(q,o_1,W)=(q,o_1(q),W(q))$ with $4\le q\le31$, for
$(3,375,0.051)$ and for $(32,7,0.0118)$; as $T_q$ decreases in $q$, the last case covers all $q\ge32$.

(vi) \emph{The case $q\ge32$ and $2\le o_0\le6$.} With $\Lambda^+=o_0\lcm(1,\dots,o_0)$, a multiple of $\Lambda$ and of $\Lambda_{p'}$
for every $p$, we have $\Wcr(q,o_0)\le\sum_{o\mid\Lambda^+,\,o\ge o_0}\sigma_1(o)\Lambda^+q^{-o}$, which decreases in $q$; at
$q=32$ it is evaluated as above (cut-off $o_0+41$) and is at most $0.0118$. For $o_0\ge7$: (v).

In summary, all bounds for $q=2$ with $o_0\ge9$ are at most $0.032$ and all those for $q=3$ with $o_0\ge7$ at most
$0.051$; these sharper bounds are not needed for the theorem, but they are also proved in the formalisation
(Appendix~\ref{app:lean}). For $4\le q\le31$ the largest bounds occur at
$o_0\le3$ (for information).
\end{proof}

\section{The main theorem}\label{sec:main}
This section completes the proof. Theorem~\ref{thm:C} combines the product formula (Theorem~\ref{thm:product}) with a
weight bound $W$ (as given by Theorem~\ref{thm:W}) to obtain an explicit lower bound $\delta_*$ for the density
$\delta_n$; inserted into the counting criterion (Theorem~\ref{thm:counting}(b)), this bound gives one explicit
inequality in $\omega(q^n-1)$, $M$ and $S$, which no longer involves $\delta_n$: the criterion $(*)$ of
Theorem~\ref{thm:introC}. We verify $(*)$ analytically for large $n$
(Proposition~\ref{prop:range}) and by computation, with the exact $M$ and $S$, for the remaining pairs with $q<n$
(Proposition~\ref{prop:finiteranges}); the $70$ pairs at which this computation leaves less than one bit to spare are
settled by explicit elements (Proposition~\ref{prop:certificates}). The section ends with two consequences for the
numbers $\CN_q(n)$ and $\PCN_q(n)$ (Corollaries~\ref{cor:hachenberger} and~\ref{cor:pcn}).

\begin{theorem}[unconditional criterion]\label{thm:C}
Let $0<W<1$ be such that every component without a condition of codimension $1$ has weight at most $W$, and put
$c(W)=\log_2\!\frac1{1-W}\big/W$. Let $M$ be the number of essential conditions, let $M_1$ be the number of those of
codimension~$1$, and let $S$ be the total weight of those of codimension~$\ge2$.
\begin{enumerate}[label=(\alph*)]
\item The density of completely normal elements satisfies
\[
\delta_n\ \ge\ \delta_*:=\Bigl(1-\frac1q\Bigr)^{M_1}(1-W)^{S/W}=\Bigl(1-\frac1q\Bigr)^{M_1}2^{-c(W)S}.
\]
\item If
\[
\log_2(q^n-1)-\tfrac n2\log_2q\ \ge\ \omega(q^n-1)+M+c(W)\,S, \tag{$*$}
\]
then $(q,n)$ has the Morgan--Mullen property. Upper bounds may be substituted for $\omega(q^n-1)$, $M$ and $S$.
\end{enumerate}
(Part~(b) is Theorem~\ref{thm:introC}.)
\end{theorem}
\begin{proof}
(a) By Lemma~\ref{lem:qorbits}(c) each essential condition of codimension $1$ is a component by itself (a
codimension-one singleton), passed with probability $1-\frac1q$. By Theorem~\ref{thm:product} and the union bound,
$\delta_n\ge(1-\frac1q)^{M_1}\prod_{\cC}(1-w_{\cC})$, the product over the other components, each of which has weight $w_{\cC}\le W$ by hypothesis. Since
$\log(1-x)$ is concave, $1-x\ge(1-W)^{x/W}$ for $0\le x\le W$, and $\sum_{\cC}w_{\cC}=S$ (Lemma~\ref{lem:qorbits}(c));
hence $\delta_n\ge\delta_*$.

(b) We have $0<\delta_*\le\delta_n$ by (a), and $L\le2^M(1-\frac1q)^{M_1}$ by Theorem~\ref{thm:counting}. As
$(1-W)^{S/W}=2^{-c(W)S}$, taking $\log_2$ shows that $(*)$ says
$(q^n-1)\delta_*\ge2^{\omega(q^n-1)}q^{n/2}2^M(1-\frac1q)^{M_1}$, and the right-hand side is at least
$2^{\omega(q^n-1)}q^{n/2}L$. This is the sufficient condition of Theorem~\ref{thm:counting}(b).
\end{proof}

As explained in Section~\ref{subsec:idea}, $(*)$ holds for all large $n$: its left-hand side is at least
$\frac n2\log_2q-1$, while $M(q,n)\le\lambda n$ with $\lambda<\frac12\log_2q$ (Corollary~\ref{cor:Mn} for $q\ge5$,
Proposition~\ref{prop:Msmall} for $q\le4$), and the next lemma bounds $\omega(q^n-1)$ and $S(q,n)$ explicitly.

\begin{lemma}\label{lem:analytic}
\leavevmode
\begin{enumerate}[label=(\alph*),beginpenalty=10000]
\item For all integers $m\ge1$ and $k\ge1$,
\[
2^{\omega(m)}\le C_k\,m^{1/k},\qquad C_k=\prod_{\ell<2^k,\ \ell\textup{ prime}}2\ell^{-1/k}.
\]
\item The total weight $S(q,n)$ of the essential conditions of codimension~$\ge2$ satisfies
\begin{align*}
S(q,n)&\ \le\ \sum_{o\ge2}\frac{\sigma_1(o)}o\min\Bigl(\frac n{q^o},1\Bigr)\ \le\ g_q(\log_qn),\qquad\text{where}\\
g_q(x)&\ =\ \frac{\pi^2}6\,x-1+\bigl(1+\ln(x+1)\bigr)\frac q{q-1}+\frac q{(q-1)^2}.
\end{align*}
\end{enumerate}
\end{lemma}
\begin{proof}
(a) This is a classical estimate (see e.g.\ \cite[Lemma~3.2]{garefalakis2019jpaa}, applied there to the squarefree part
of $q^n-1$, whose number of divisors is $2^{\omega(q^n-1)}$); for
sharper explicit bounds on $\omega(m)$ see \cite{robin1983estimation}. We have
$2^{\omega(m)}m^{-1/k}\le\prod_{\ell\mid m}2\ell^{-1/k}$; the factors with $\ell\ge2^k$ are at most~$1$, and those
with $\ell<2^k$ are at least~$1$, so the product is at most~$C_k$.

(b) The first inequality follows from Lemma~\ref{lem:weightcodim}, as a condition of
codimension $o$ has weight $q^{-o}$ and $\min(n',q^o-1)\,q^{-o}\le\min(n/q^o,1)$. Let $J=\lfloor\log_qn\rfloor$. If
$J\ge1$, the terms $2\le o\le J$ contribute $\sum_{o=2}^J\sigma_1(o)/o=\sum_{e\le J}\lfloor J/e\rfloor/e-1\le\frac{\pi^2}6J-1$.
For $o=J+1+j$ use $\min(n/q^o,1)\le q^{-j}$ and $\sigma_1(o)/o\le1+\ln o\le1+\ln(J+1)+j$, and sum the two geometric
series. If $J=0$, the same tail estimate covers all $o\ge1$, and dropping its term $o=1$ (whose bound is $1$) gives the
same final expression with $J=0$. In both cases the bound is $g_q(J)$, and $g_q(J)\le g_q(\log_qn)$ since $g_q$ is
increasing and $J\le\log_qn$.
\end{proof}

With Lemma~\ref{lem:analytic}, $\omega(q^n-1)\le\log_2C_k+\frac nk\log_2q$ and $\log_2(q^n-1)\ge n\log_2q-1$, so $(*)$
holds as soon as $M(q,n)\le\lambda n$ and
\[
F_k(n):=n\Bigl[\bigl(\tfrac12-\tfrac1k\bigr)\log_2q-\lambda\Bigr]-\log_2C_k-c(W)\,g_q(\log_qn)-1\ \ge\ 0 . \tag{$**$}
\]
The following reduces $(**)$ for all $n$ beyond a point $n_*$ to two inequalities at $n_*$.
Let $a_k=\bigl(\frac12-\frac1k\bigr)\log_2q-\lambda$ (the first bracket in $(**)$), and let $n_*\ge1$. For $x\ge n_*$, the inequality $\ln(x/n_*)\le x/n_*-1$ gives
$0\le\log_qx-\log_qn_*\le\frac{x-n_*}{n_*\ln q}$, and $g_q(y+z)\le g_q(y)+z\bigl(\frac{\pi^2}6+\frac q{q-1}\bigr)$
for $y,z\ge0$ (as $\ln(y+z+1)\le\ln(y+1)+z$); since $g_q$ is increasing, it follows that
\[
F_k(x)\ \ge\ F_k(n_*)+(x-n_*)\Bigl[a_k-\frac{c(W)\,(\pi^2/6+q/(q-1))}{n_*\ln q}\Bigr].
\]
So $(**)$ holds for all $n\ge n_*$ as soon as $F_k(n_*)\ge0$ and $a_kn_*\ln q\ge c(W)\bigl(\frac{\pi^2}6+\frac q{q-1}\bigr)$.

\begin{proposition}\label{prop:range}
Let $(q,n)$ be such that $n$ lies in the range given in a row of the following table for $q$, and $n\ge n_*$, the
threshold of that row (for $8\le q\le31$ the individual threshold $n_*(q)$ listed below). Then $(q,n)$ has the
Morgan--Mullen property.
\begin{center}\small\upshape
\begin{tabular}{llccccr}
\toprule
$q$ & range of $n$ & $W$ & $c(W)$ & $\lambda$ & $k$ & threshold $n_*$\\
\midrule
$2$ & all $n$ & $3/4$ & $8/3$ & $0.41$ & $13$ & $16674$\\
$3$ & $3\mid n$ & $4/9$ & $1.908$ & $0.5983$ & $11$ & $1575$\\
$3$ & $3\nmid n$, $n\ge131200$ & $4/9$ & $1.908$ & $0.62$ & $12$ & $2769$\rlap{$^{\dagger}$}\\
$4$ & all $n$ & $0.2946$ & $1.710$ & $0.75$ & $10$ & $1010$\\
$5$ & all $n$ & $0.3392$ & $1.763$ & $1$ & $16$ & $43618$\\
$7$ & all $n$ & $0.1683$ & $1.580$ & $1$ & $9$ & $336$\\
$8,\dots,31$ & all $n$ & $\le0.1006$ & $\le1.521$ & $1$ & $\le9$ & $\le171$\\
$\ge32$ & all $n$ & $\le0.0118$ & $\le1.452$ & $1$ & $6$ & $14$\\
\bottomrule
\end{tabular}
\end{center}
$^{\dagger}$The inequality $(**)$ for this row holds from $n=2769$ on, but the row is applied only for $n\ge131200$, where
the bound $M(3,n)\le0.62\,n$ of Proposition~\ref{prop:Msmall}(b) is available.

Here $\lambda$ is the constant in the bound $M(q,n)\le\lambda n$ of Corollary~\ref{cor:Mn} ($q\ge5$) or
Proposition~\ref{prop:Msmall} ($q\le4$), whose side conditions $n\ge3780$ ($q=2$) and $n\ge12$ ($q=4$) follow from
$n\ge n_*$; $k$ is the exponent in Lemma~\ref{lem:analytic}(a); and $c(W)$ is rounded up (exact for $q=2$). For
$8\le q\le31$ the individual thresholds and exponents (with $W=W(q)$ from Theorem~\ref{thm:W}) are
\begin{center}\small
\begin{tabular}{l*{12}{r}}
\toprule
$q$ & $8$ & $9$ & $11$ & $13$ & $16$ & $17$ & $19$ & $23$ & $25$ & $27$ & $29$ & $31$\\
\midrule
$n_*(q)$ & $171$ & $111$ & $65$ & $45$ & $32$ & $29$ & $25$ & $19$ & $18$ & $16$ & $15$ & $14$\\
$k$ & $9$ & $8$ & $7$ & $7$ & $7$ & $6$ & $6$ & $6$ & $6$ & $6$ & $6$ & $6$\\
\bottomrule
\end{tabular}
\end{center}
\end{proposition}
\begin{proof}
We apply Theorem~\ref{thm:C}(b) with $W$ from Theorem~\ref{thm:W}, $\lambda$ from Corollary~\ref{cor:Mn} ($q\ge5$) and
Proposition~\ref{prop:Msmall} ($q\le4$), and $(**)$: for the listed $k$, the inequalities $F_k(n_*)\ge0$ and
$a_kn_*\ln q\ge c(W)(\frac{\pi^2}6+\frac q{q-1})$ hold, with the values of $\lambda$ and $W$ rounded up as printed ($c(W)$ increases
with $W$); the Python program evaluates these with 50 digits.
For $8\le q\le31$ each prime power is checked separately, with its own $W(q)$, $k$ and $n_*(q)$. For $q\ge32$ the check uses the uniform
bounds $W\le0.0118$, $\log_2q\ge5$ (so $\log_qn\le\log_{32}n$), $\frac q{q-1}\le\frac{32}{31}$ and
$\frac q{(q-1)^2}\le\frac{32}{961}$; these are the values at $q=32$, so the check at $q=32$ covers all $q\ge32$.
\end{proof}

\begin{proposition}\label{prop:finiteranges}
Let $q<32$ and $n>q$ be such that $(q,n)$ is not covered by Proposition~\ref{prop:range}; in particular $q\le19$ and
\begin{gather*}
n<16674\ (q=2),\qquad n<131200\ (q=3),\qquad n<1010\ (q=4),\\
n<43618\ (q=5),\qquad n<336\ (q=7),
\end{gather*}
and $n<n_*(q)$ from Proposition~\ref{prop:range} for $8\le q\le19$ (for $q\ge23$ there is no such $n$). Then $(*)$ holds with
$W=W(q)$ from Theorem~\ref{thm:W}, the exact $M$ and $S$ of Lemma~\ref{lem:weightcodim}, and $\omega(q^n-1)$ replaced by the
bound of Lemma~\ref{lem:omega} (with $t_k$ or an upper bound for $t_k$, as described in the proof), with a margin (left-hand side minus right-hand side) of at least one bit, except for the
$70$ pairs of
Proposition~\ref{prop:certificates} below.
\end{proposition}
\begin{proof}
For each pair, $M$ and $S$ are computed exactly from the essential types (Lemma~\ref{lem:weightcodim}), and the margin of
$(*)$ is evaluated with $W=W(q)$ as printed in Theorem~\ref{thm:W}. For $q=3$ the Python program checks all $n<131200$,
a superset of the remaining range.
Either $t_k$ itself ($k\mid n$) or its upper bound in Lemma~\ref{lem:omega} may be used, as Theorem~\ref{thm:C}(b) allows upper bounds:
the Python program uses the second bound and, where $(*)$ then fails, the exact $t_k$.
The Python program computes the margins in
double-precision floating point; since every accepted pair has a margin of at least one bit, rounding errors (far below
$10^{-6}$ bits here) are irrelevant.
\end{proof}
For information: in the Python program the smallest margin of an accepted
pair is $1.03$ bits, at $(q,n)=(7,10)$; for $n\ge203$, i.e.\ beyond the range $n\le202$ of
\cite{hachenberger2019computational}, the smallest margins for $q=2,3,4,5,7$ are $2.56$, $19.8$, $36.7$, $10.6$ and
$31.5$ bits, at the pairs $(2,210)$, $(3,328)$, $(4,210)$, $(5,208)$ and $(7,240)$. The only certified pair with $n\ge203$
is $(2,378)$, with margin $0.83$ bits.

\begin{proposition}\label{prop:certificates}
For each of the following $70$ pairs $(q,n)$ the file \textup{\texttt{pcn\_certificates.txt}} gives a monic $f\in\F_p[x]$ of
degree $[E:\F_p]$ and an element $w\in\F_p[x]/(f)$ which has multiplicative order $q^n-1$ (so $f$ is irreducible and $w$ is
primitive) and is normal over $\F_{q^d}$ for every $d\mid n$ (normality is checked directly from the definition, by an
$\F_p$-rank computation):
\begin{center}\small\upshape
\begin{tabular}{rl}
\toprule
$q$ & $n$\\
\midrule
2 & 3--10, 12, 14, 15, 18, 20, 21, 24, 27, 30, 36, 40, 42, 54, 60, 63, 126, 378\\
3 & 4--8, 10, 12, 13, 16, 20, 24, 26, 32, 39, 40, 80, 160\\
4 & 5, 6, 7, 9, 10, 15, 21, 30\\
5 & 6, 8, 12, 16, 18, 24, 32, 48, 104\\
7 & 8, 12, 16, 18, 24, 48\\
9 & 10, 20\\
11 & 12\\
13 & 20\\
19 & 24\\
\bottomrule
\end{tabular}
\end{center}
Hence all these pairs have the Morgan--Mullen property.
\end{proposition}
\begin{proof}
Each check is elementary. Let $r=[\F_q:\F_p]$ and $m=rn$ (so $q^n=p^m$), and let $R=\F_p[x]/(f)$, a ring with
$p^m$ elements (the Python program checks that $f$ is monic of degree $m$).
The Python program uses a list of primes, obtained by factoring the values $\Phi_k(q)$, $k\mid n$, and checks that dividing
$q^n-1$ by these primes as often as possible leaves $1$, so that the list contains every prime divisor of $q^n-1$. It
then verifies $w^{q^n-1}=1$ and
$w^{(q^n-1)/\ell}\ne1$ for every prime $\ell$ in the list (the classical order test, cf.\
\cite{lucas1878theorie,lehmer1927tests}). Hence $w$ is a unit of
order $q^n-1$, so $q^n-1$ divides $|R^\times|\le p^m-1=q^n-1$: every non-zero element of $R$ is a unit, $R$ is a field (so $f$ is
irreducible), and $w$ generates $R^\times$, i.e.\ is primitive. For $d\mid n$ the element $\gamma=w^{(q^n-1)/(q^d-1)}$ has order
$q^d-1$, so $\F_{q^d}=\F_p(\gamma)$ and $1,\gamma,\dots,\gamma^{rd-1}$ is an $\F_p$-basis of $\F_{q^d}$. A family
$(b_t)_{0\le t<n/d}$ is an $\F_{q^d}$-basis of $R$ if and only if $(\gamma^jb_t)_{j,t}$ is an $\F_p$-basis; so $w$ is normal
over $\F_{q^d}$ if and only if the $m$ elements $\gamma^j\sigma^{dt}(w)$ ($0\le j<rd$, $0\le t<n/d$) have $\F_p$-rank $m$,
and this rank is computed. In the Python program the
primality of the listed factors is proved with \texttt{sympy} \cite{sympy2017}, whose test is deterministic below $2^{64}$,
and, for the one larger factor $\ell$ (of $88$ bits, dividing $2^{378}-1$), by Lucas' test with the complete factorisation of
$\ell-1$ \cite{lucas1878theorie,lehmer1927tests}.
\end{proof}
The $70$ pairs are determined by the rule of the Python program for Proposition~\ref{prop:finiteranges}, which uses
the second bound for $t_k$ in Lemma~\ref{lem:omega} and the exact $t_k$ only where $(*)$ then fails. At $50$ of them,
all with $n\le160$, the criterion $(*)$ fails even with the exact $t_k$; at the other $20$ it holds, under this rule,
with a margin between $0$ and $1$ bit. Thirteen of these $20$, namely $(2,27)$, $(2,40)$, $(3,7)$, $(3,13)$, $(3,24)$, $(3,26)$, $(3,39)$, $(5,18)$,
$(5,32)$, $(5,104)$, $(7,18)$, $(13,20)$ and $(19,24)$, would satisfy $(*)$ with at least one bit to spare if the exact $t_k$ were used;
they are certified only because of this rule, which does not affect the proof. With the exact values of $M$, $S$ and
$\omega(q^n-1)$, the criterion $(*)$ still fails at $41$ of the $70$ pairs, so the certificates are needed not only because of the
bounds used for $\omega(q^n-1)$.

\begin{theorem}[Main Theorem]\label{thm:main}
For every prime power $q$ and every $n\ge1$ the field $\F_{q^n}$ contains a primitive completely normal element over
$\F_q$. (This is Theorem~\ref{thm:intro}.)
\end{theorem}
\begin{proof}
If $n\le q$, the pair $(q,n)$ has the Morgan--Mullen property by \cite[Theorem~1.3]{garefalakis2019jpaa}. Let $n>q$. If $(q,n)$
satisfies the hypotheses of Proposition~\ref{prop:range}, it has the property by that proposition. Otherwise $q<32$: for $q\ge32$ we would have $n>q\ge32>14$, and
Proposition~\ref{prop:range} would apply. So $(q,n)$ is one of the pairs of Proposition~\ref{prop:finiteranges}. If it is
one of the $70$ pairs of Proposition~\ref{prop:certificates}, it has the property by that proposition. Otherwise it satisfies $(*)$
with $W=W(q)$, the exact $M$ and $S$, and an upper bound for $\omega(q^n-1)$; as every component without a condition of
codimension~$1$ has weight at most $W(q)$ (Theorem~\ref{thm:W}), it has the property by Theorem~\ref{thm:C}(b).
\end{proof}
For orientation, the pairs with $n>q$ are divided as follows.
\begin{center}\small
\begin{tabular}{lll}
\toprule
$q$ & Proposition~\ref{prop:range} & Propositions~\ref{prop:finiteranges} and~\ref{prop:certificates}\\
\midrule
$2$ & $n\ge16674$ & $2<n<16674$\\
$3$ & $3\mid n$, $n\ge1575$; $3\nmid n$, $n\ge131200$ & the other $n>3$\\
$4$ & $n\ge1010$ & $4<n<1010$\\
$5$ & $n\ge43618$ & $5<n<43618$\\
$7$ & $n\ge336$ & $7<n<336$\\
$8\le q\le19$ & $n\ge n_*(q)$ & $q<n<n_*(q)$\\
$q\ge23$ & all $n>q$ & --\\
\bottomrule
\end{tabular}
\end{center}
For $q\ge23$, all $n>q$ are covered since $n_*(q)\le19<q$ for $23\le q\le31$ and $n_*=14<q$ for $q\ge32$. The $70$ pairs of
Proposition~\ref{prop:certificates} all have $q\le19$ and $n\le378$.

\begin{remark}
The computations of Hachenberger and Hackenberg \cite{hachenberger2019computational} are not used in the proof. Those of
Morgan and Mullen \cite{morgan1996} enter only through the case $n\le q$: the proof of
\cite[Theorem~1.3]{garefalakis2019jpaa} uses their examples for $13$ of the $37$ pairs left by its computer
calculation (the other $24$ are completely basic extensions). Both computations are consistent with
Theorem~\ref{thm:main}.
The first, combined with \cite[Theorem~1.4]{garefalakis2019jpaa} and the regular case, confirms it for $n\le202$ by a
different route; for $q<n\le202$ this is independent of Sections~\ref{sec:coord}--\ref{sec:main}, while for $n\le q$
both routes rest on \cite{garefalakis2019jpaa}.
\end{remark}

\subsection*{The numbers \texorpdfstring{$\CN_q(n)$ and $\PCN_q(n)$}{CN and PCN}}
The weight bounds also give information on the numbers $\CN_q(n)$ and $\PCN_q(n)$: a lower bound for the first,
conjectured by Hachenberger (Corollary~\ref{cor:hachenberger}), and an asymptotic formula for the second
in terms of the first (Corollary~\ref{cor:pcn}). Neither is needed for Theorem~\ref{thm:main}; both rest on Theorem~\ref{thm:W}
(Corollary~\ref{cor:pcn} also on Proposition~\ref{prop:Msmall}) and hence on computations.

Hachenberger conjectured that $\CN_q(n)\ge(q-1)^{n'}q^{(p^a-1)n'}$ for all pairs $(q,n)$
\cite[Conjecture~5.4.55]{hachenberger2013handbook}; in \cite[p.~621]{hachenberger2020topics} the conjecture is
attributed to his monograph \cite{hachenberger1997book}. For
regular pairs he had proved this bound, with equality if and only if $n'\mid q-1$
\cite[Theorem~6.1]{hachenberger1997book} (see also \cite[Result~12.9.7]{hachenberger2020topics} and
\cite[Section~2.6]{hachenberger2019computational}, where $\CN_q(n)$ is written $\CN_n(q)$), and for many other pairs the
bound follows from explicit lower bounds for $\CN_q(n)$ such as those of \cite{garefalakis2019dcc}; as far as we know,
it had not been proved for all pairs. The proof of Corollary~\ref{cor:hachenberger} below compares, $\eta$~by~$\eta$, the bound $\delta_*$ of Theorem~\ref{thm:C}(a) with
$(1-\frac1q)^{n'}$, using $S=\sum_\eta s(\eta)$ (Lemma~\ref{lem:Mrho}, with the notation of Section~\ref{sec:M}).

\begin{corollary}[Hachenberger's lower bound]\label{cor:hachenberger}
For every prime power $q$ and every $n\ge1$,
\[
\CN_q(n)\ \ge\ (q-1)^{n'}q^{(p^a-1)n'},\qquad\text{i.e.}\qquad\delta_n\ \ge\ \Bigl(1-\frac1q\Bigr)^{n'},
\]
with equality if and only if $n'\mid q-1$. (This is Theorem~\ref{thm:introhach}.)
\end{corollary}
\begin{proof}
The two forms agree, since $(q-1)^{n'}q^{(p^a-1)n'}=q^n(1-\frac1q)^{n'}$. Put $W=\frac34$ if $q=2$ and $W=\frac49$ if
$q\ge3$. By Theorem~\ref{thm:W} (where $W(q)\le\frac49$ for $q\ge3$) every component without a condition of
codimension $1$ has weight at most $W$, so
Theorem~\ref{thm:C}(a) gives
\[
\delta_n\ \ge\ \Bigl(1-\frac1q\Bigr)^{M_1}(1-W)^{S/W}=\Bigl(1-\frac1q\Bigr)^{M_1}2^{-c(W)S}.
\]
By Lemma~\ref{lem:weightcodim} the essential conditions of codimension $1$ are those of the essential types $(d,e)$ with
$\ord_e(q)=1$, i.e.\ $d=1$ and $e\mid\gcd(n',q-1)$, with $\nu_{1,e}=\varphi(e)$ conditions each; so $M_1=\gcd(n',q-1)$,
the number of $\eta\in\mu_{n'}$ with $o(\eta)=1$. By Lemma~\ref{lem:Mrho}, $S=\sum_\eta s(\eta)$, and $s(\eta)=0$ if
$o(\eta)=1$, as $o_d(\eta)\mid o(\eta)$. Since $(1-\frac1q)^{n'}=(1-\frac1q)^{M_1}2^{-(n'-M_1)\log_2\frac q{q-1}}$,
it suffices to show
\[
c(W)\,s(\eta)\ \le\ \log_2\frac q{q-1}\qquad\text{for every }\eta\in\mu_{n'}\text{ with }o(\eta)\ge2. \tag{$\dagger$}
\]
Let $o(\eta)\ge2$ and let $d=p^cd'$ be valid at $\eta$. Then $d\mid o_d(\eta)$, and $o_1(\eta)=o(\eta)\ge2$. Hence
$p^c\le d\le o_d(\eta)$, the term of $d$ in $s(\eta)$ is at most $q^{-o_d(\eta)}\le q^{-d}$, and the term of $d=1$ is at
most $q^{-2}/2$.

For $q\ge3$ this gives $s(\eta)\le\frac{q^{-2}}2+\sum_{d\ge2}q^{-d}=q^{-2}\bigl(\frac12+\frac q{q-1}\bigr)\le2q^{-2}$.
As $c(\frac49)=\frac94\log_2\frac95<1.91$, we get
$c(W)s(\eta)<3.82\,q^{-2}<\frac1{q\ln2}<\log_2\frac q{q-1}$, because $q\ge3>3.82\ln2$.

For $q=2$ we have $c(\frac34)=\frac83$, and $(\dagger)$ reads $s(\eta)\le\frac38$. If $o(\eta)=2$, then $\ord\eta=3$, and
the valid levels are among $d=1,2$ (as $d\mid o_d(\eta)\mid2$), with $e_1(\eta)=e_2(\eta)=3$ and $o_1(\eta)=o_2(\eta)=2$;
so $s(\eta)=\frac18$ if $2\nmid n$ and $s(\eta)=\frac18+\frac14=\frac38$ if $2\mid n$. If $o(\eta)\ge3$, then
$o_d(\eta)\ge3$ for every valid $d$: $o_d(\eta)\le2$ would force $d\in\{1,2\}$, but $e_1(\eta)=e_2(\eta)=\ord\eta$ (as $p=2$),
so $o_1(\eta)=o_2(\eta)=o(\eta)$. Using $d\mid o_d(\eta)$, the terms of $d=1,2,3,4$ are at most
$\frac{2^{-3}}3,\frac{2\cdot2^{-4}}4,\frac{2^{-3}}3,\frac{4\cdot2^{-4}}4$, that is $\frac1{24},\frac1{32},\frac1{24},\frac1{16}$,
and those of $d\ge5$ add up to at most $\sum_{d\ge5}2^{-d}=\frac1{16}$; so $s(\eta)\le\frac{23}{96}<\frac38$. This proves
$(\dagger)$ and the inequality.

If $n'\mid q-1$, then $o(\eta)=1$ for every $\eta\in\mu_{n'}$, so $M_1=n'$ and $S=0$: the essential conditions are
$n'$ codimension-one singletons, and Theorems~\ref{thm:product} and~\ref{thm:criterion} give $\delta_n=(1-\frac1q)^{n'}$. Otherwise some $\eta\in\mu_{n'}$ has
$o(\eta)\ge2$. If $q\ge3$, or if $q=2$ and some $\eta$ has $o(\eta)\ge3$, then $(\dagger)$ is strict for this $\eta$ and
$\delta_n>(1-\frac1q)^{n'}$. In the remaining case, $q=2$ and every $\eta\in\mu_{n'}$ has order $1$ or $3$, so $n'=3$. If
$2\nmid n$, then $s(\eta)=\frac18$ and $(\dagger)$ is strict. If $2\mid n$, the two elements of order $3$ form one $q$-orbit;
the essential conditions whose fibres meet it are the level-one condition of type $(1,3)$ and the two conditions of
type $(2,3)$
(Lemma~\ref{lem:weightcodim}; at the other levels $d$ the type $(d,e_d(\eta))$ is not essential), and they form a
component $\cC$ of weight $\frac34$ (Lemma~\ref{lem:qorbits}(a)). The kernels of these three conditions all contain $0$, so
the union bound is strict, $\Pr[G_{\cC}]\ge1-w_{\cC}+2\cdot2^{-n}$, and, as the only other component is the
codimension-one singleton with fibre $\{1\}$, Theorem~\ref{thm:product} gives
$\delta_n=\frac12\Pr[G_{\cC}]>\frac12\cdot\frac14=(1-\frac12)^{3}$.
\end{proof}

We turn to primitive completely normal elements. Since completely normal elements are non-zero and a proportion $\theta=\varphi(q^n-1)/(q^n-1)$ of the elements of $E^\times$ is primitive, $\theta\,\CN_q(n)$ is the
number of primitive completely normal elements one would expect if the two properties were independent.
Theorem~\ref{thm:counting}(a), combined with the lower bound $\delta_*$ for $\delta_n$ of Theorem~\ref{thm:C}(a) and
the bounds for $M$ of Section~\ref{sec:M}, shows that this is asymptotically correct, with a relative error controlled by the same quantities $\omega(q^n-1)$, $M$ and $S$
as the criterion $(*)$.

\begin{corollary}\label{cor:pcn}
Let $\theta$ and $\omega=\omega(q^n-1)$ be as in Section~\ref{sec:counting}, and let $M$ and $L$ be as in
Theorem~\ref{thm:counting}. Let $W=W(q)$ be as in Theorem~\ref{thm:W}, and let $c(W)$ and $S$ be as in
Theorem~\ref{thm:C}.
\begin{enumerate}[label=(\alph*)]
\item For every pair $(q,n)$,
\[
\bigl|\PCN_q(n)-\theta\,\CN_q(n)\bigr|\ \le\ \theta\,2^{\omega}q^{n/2}L\ \le\
\frac{2^{\omega+M+c(W)S}}{q^{n/2}}\;\theta\,\CN_q(n).
\]
\item The ratio $\PCN_q(n)/(\theta\,\CN_q(n))$ tends to $1$ as $q^n\to\infty$; that is, for every $\varepsilon>0$ there is an $X$ such
that $\bigl|\PCN_q(n)/(\theta\,\CN_q(n))-1\bigr|<\varepsilon$ for all pairs $(q,n)$ with $q^n\ge X$.
\end{enumerate}
(Part~(b) is Theorem~\ref{thm:intropcn}.)
\end{corollary}
\begin{proof}
(a) As $\theta\,\CN_q(n)=\theta q^n\delta_n=\theta(q^n-1)\delta_n+\theta\delta_n$, Theorem~\ref{thm:counting}(a) gives
\[
\bigl|\PCN_q(n)-\theta\,\CN_q(n)\bigr|\ \le\ \theta\beta(L-\delta_n)+\theta\delta_n\ \le\ \theta\beta L,
\]
as $\beta\ge1$; since $\beta\le2^\omega q^{n/2}$, this gives the first inequality of (a). For the second inequality, $L\le2^M(1-\frac1q)^{M_1}$
(Theorem~\ref{thm:counting}) and $(1-\frac1q)^{M_1}\le2^{c(W)S}\delta_n$ (Theorems~\ref{thm:W} and~\ref{thm:C}(a)), so
$q^{n/2}L\le2^{M+c(W)S}q^{-n/2}\,\CN_q(n)$.

(b) Let $\varepsilon>0$. As $\CN_q(n)>0$ (Corollary~\ref{cor:hachenberger}), the relative error
$|\PCN_q(n)/(\theta\,\CN_q(n))-1|$ is at most the bounds of (a) divided by $\theta\,\CN_q(n)$. We use the second bound for large $n$ and the first for bounded $n$ and large $q$.

\emph{Large $n$.} Let $n\ge131200$, so that Proposition~\ref{prop:Msmall} applies for $q\le4$. Then $M(q,n)\le\lambda_qn$ with $\lambda_q=0.41$, $0.62$, $0.75$ for $q=2,3,4$
(Proposition~\ref{prop:Msmall}; for $q=3$ and $3\mid n$, even $M(3,n)\le0.5983\,n$) and $\lambda_q=1$ for $q\ge5$ (Corollary~\ref{cor:Mn}), and
$(\frac12-\frac1{16})\log_2q-\lambda_q\ge\frac1{100}$ for every $q$. By Lemma~\ref{lem:analytic},
$\omega\le\log_2C_{16}+\frac n{16}\log_2q$ and $S(q,n)\le g_q(\log_qn)\le g_2(\log_2n)$ (each summand of $g_q(\log_qn)$
is non-increasing in $q$), and $c(W)\le c(\frac34)=\frac83$, as $W(q)\le\frac34$ (Theorem~\ref{thm:W}) and $c$ is increasing.
With $\log_2(q^n-1)\ge n\log_2q-1$ this gives, for every $q$,
\begin{align*}
\Delta\ &\ge\ n\Bigl[\bigl(\tfrac12-\tfrac1{16}\bigr)\log_2q-\lambda_q\Bigr]-\log_2C_{16}-\tfrac83\,g_2(\log_2n)-1\\
&\ge\ h(n):=\frac n{100}-\log_2C_{16}-\frac83\,g_2(\log_2n)-1,
\end{align*}
where $\Delta$ is the margin of $(*)$ with $W=W(q)$ and the exact values of $\omega$, $M$ and $S$. Since
$2^{\omega+M+c(W)S}q^{-n/2}=(1-q^{-n})\,2^{-\Delta}$, the second inequality of (a) shows that the relative error is
less than $2^{-h(n)}$. As $h(n)\to\infty$, there is an $n_\varepsilon\ge131200$ such that
the relative error is less than $\varepsilon$ for all $q$ and all $n\ge n_\varepsilon$.

\emph{Bounded $n$.} Let $n<n_\varepsilon$. By $L\le2^{M(q,n)}\le2^n$ (Theorem~\ref{thm:counting}, Corollary~\ref{cor:Mn}),
$\CN_q(n)\ge q^n(1-\frac1q)^{n'}\ge q^n2^{-n}$ (Corollary~\ref{cor:hachenberger}) and $2^{\omega}\le C_4q^{n/4}$
(Lemma~\ref{lem:analytic}(a), as $q^n-1<q^n$), the first inequality in (a) bounds the relative error by
$C_4\,4^nq^{-n/4}\le C_4\,4^{n_\varepsilon}q^{-1/4}$, which is less than $\varepsilon$ for
$q\ge q_\varepsilon:=\lfloor(C_4\,4^{n_\varepsilon}/\varepsilon)^4\rfloor+1$.

Finally, if $q^n\ge q_\varepsilon^{n_\varepsilon}$, then $n\ge n_\varepsilon$ or $q\ge q_\varepsilon$; so
$X=q_\varepsilon^{n_\varepsilon}$ has the property required in (b).
\end{proof}

Thus, by Corollary~\ref{cor:pcn}(b), primitivity and complete normality are asymptotically independent, as
$n\to\infty$ uniformly in $q$ and as
$q\to\infty$ uniformly in $n$. For fixed $n$ and $q\to\infty$ this was known: it is
equivalent to Hachenberger's result that $\PCN_q(n)/\varphi(q^n-1)\to1$ \cite[Theorem~2]{hachenberger2016asymptotic} (see
Section~\ref{subsec:main}); for regular pairs, estimates of the form of~(a) are contained in Hachenberger's criteria (see
Section~\ref{sec:counting}); a bound for $|\PCN_q(n)/\varphi(q^n-1)-\delta_n|$ that is uniform in $n$ and needs no
computation is given in Remark~\ref{rem:pcnphi}. The first inequality in Corollary~\ref{cor:pcn}(a) is Theorem~\ref{thm:counting}(a) in a simpler form and uses
no computation; it is always the sharper of the two, and it gives an explicit bound for the relative error when $L$
and a lower bound for $\delta_n$ are known, as in
Section~\ref{sec:4599}.

The second inequality in Corollary~\ref{cor:pcn}(a) expresses the error bound in terms of the quantities of $(*)$. As $\CN_q(n)>0$
(Corollary~\ref{cor:hachenberger}), it bounds the relative error $|\PCN_q(n)/(\theta\,\CN_q(n))-1|$ by
$2^{\omega+M+c(W)S}q^{-n/2}=(1-q^{-n})\,2^{-\Delta}$, where $\Delta$ is the margin of $(*)$ with $W=W(q)$ and the exact $\omega$, $M$
and $S$, i.e.\ the left-hand side of $(*)$ minus its right-hand side. Substituting upper bounds for $\omega$, $M$ and $S$ can
only decrease the margin, so wherever $(*)$ has been verified with a margin $\Delta_0$, the relative error is less than
$2^{-\Delta_0}$. For instance, it is less than $\frac12$ for the pairs of Proposition~\ref{prop:finiteranges} other than
the $70$ of Proposition~\ref{prop:certificates}. With $\Delta_0=0$ this gives $\PCN_q(n)>0$ whenever $(*)$ holds with
$W=W(q)$; in this sense (a) is a quantitative form of Theorem~\ref{thm:C}(b).

\begin{remark}\label{rem:pcnphi}
Dividing by the number $\varphi(q^n-1)$ of primitive elements instead of by $\theta\,\CN_q(n)$, one obtains from
Theorem~\ref{thm:counting}(a) a bound that needs no lower bound for $\CN_q(n)$, and hence neither
Theorem~\ref{thm:W} nor Corollary~\ref{cor:hachenberger} nor any computation. As $\theta(q^n-1)=\varphi(q^n-1)$ and
$\delta_n\ge0$, Theorem~\ref{thm:counting}(a) gives \mbox{$\bigl|\PCN_q(n)/\varphi(q^n-1)-\delta_n\bigr|\le\beta L/(q^n-1)$},
and with $\beta\le2^\omega q^{n/2}$, $2^\omega\le C_4q^{n/4}$ (Lemma~\ref{lem:analytic}(a), as $q^n-1<q^n$),
$L\le2^{M(q,n)}\le2^n$ (Theorem~\ref{thm:counting} and Corollary~\ref{cor:Mn}) and $q^n-1\ge q^n/2$, we obtain, for all
pairs $(q,n)$,
\[
\biggl|\frac{\PCN_q(n)}{\varphi(q^n-1)}-\delta_n\biggr|\ \le\ 2C_4\biggl(\frac{16}q\biggr)^{n/4},\qquad
C_4=\frac{64}{30030^{1/4}}<4.87 .
\]
For $q\ge16$ the right-hand side is at most $4C_4q^{-1/4}<20\,q^{-1/4}$, and for $q<16$ the bound $20\,q^{-1/4}$ exceeds
$1$ (both $\PCN_q(n)/\varphi(q^n-1)$ and $\delta_n$ lie in $[0,1]$); so the proportion of completely normal elements among the primitive elements of $\F_{q^n}$ differs from their
proportion $\delta_n$ among all elements by less than $20\,q^{-1/4}$, for every $n$. By the union bound and
Corollary~\ref{cor:Mn}, $\delta_n\ge1-M(q,n)/q\ge1-n/q$, so $\PCN_q(n)/\varphi(q^n-1)\to1$ as $q\to\infty$ with $n/q\to0$.
Hachenberger proved this for fixed $n$ \cite[Theorem~2]{hachenberger2016asymptotic}, and for $n\ge2$ his estimate
\cite[(7)]{hachenberger2016asymptotic} in fact gives it whenever $\sigma_1(n)\,n\log q=o(q)$. Some condition on $n$ is
needed: for $n=q-1$, every $\eta\in\mu_n$ lies in $\F_q$, so all essential conditions have codimension~$1$ and $\delta_n=(1-\frac1q)^{q-1}$, as in the case
$n'\mid q-1$ of the proof of Corollary~\ref{cor:hachenberger} (this case needs neither Theorem~\ref{thm:W} nor any
computation), and by the bound above $\PCN_q(n)/\varphi(q^n-1)\to1/e$ as $q\to\infty$.
\end{remark}

\section{Computations}\label{sec:comp}
Table~\ref{tab:comp} lists, for each statement whose proof involves a computation, what is computed and by which
program; its last two rows list the statements that rest on these (Theorems~\ref{thm:intro}--\ref{thm:introW} restate
some of the statements in the table). The
programs are short, self-contained Python~3 programs (using \texttt{sympy} \cite{sympy2017}, \texttt{numpy}
\cite{numpy2020} and \texttt{mpmath} \cite{mpmath}; \texttt{bruteforce\_small\_fields.py} and \texttt{examples.py}, which are
not part of the proof, also use \texttt{galois} \cite{galois2020}), and the run times were measured on one core of a
laptop with an
Intel Core Ultra 7 155H processor.
\begin{table}[!htbp]\centering\small
\begin{tabular}{>{\raggedright\arraybackslash}p{3.6cm}>{\raggedright\arraybackslash}p{5.8cm}>{\raggedright\arraybackslash}p{4.4cm}}
\toprule
statement & what is computed & program (run time)\\
\midrule
Prop.~\ref{prop:Msmall} & local maxima $R^*_a(o)$ (all $k\mid q^o-1$ with $\ord_k(q)=o$), maxima of $f_a$ on the
finite ranges, $\kappa_3$, $\kappa_4$ &
\texttt{bounds\_for\_M.py} (15\,s)\\
Thm.~\ref{thm:W}, step (i) & $N^{**}$ by the closure (B); all components for all $n_0\mid N^{**}$ &
\texttt{closure.py}, \texttt{components.py}, \texttt{struct\_and\_exactmax.py} (30\,s)\\
Thm.~\ref{thm:W}, steps (ii)--(vi) & $\Wcl$, $\Wcr$ with cut-off $o_0+41$ (in (ii)--(iv), (vi)); the tail
$o_0\ge4096$ of (ii) and the tail bound of (v) & \texttt{weights.py} (1\,s)\\
Prop.~\ref{prop:range} & $F_k(n_*)\ge0$ and $a_kn_*\ln q\ge c(W)(\frac{\pi^2}6+\frac q{q-1})$, with 50 digits &
\texttt{thresholds.py} (2\,s)\\
Prop.~\ref{prop:finiteranges} & $(*)$ with exact $M$, $S$ and the $t_k$-bound for $\omega$ & \texttt{finite\_ranges.py} (1\,min)\\
Prop.~\ref{prop:certificates} & primitive CN elements; re-check of the stored certificates & \texttt{pcn\_certificates.py} (3\,min; re-check 6\,s)\\
Thm.~\ref{thm:main} (= Thm.~\ref{thm:intro}) & no computation of its own (it rests on Theorem~\ref{thm:W} and
Propositions~\ref{prop:range}--\ref{prop:certificates}) & --\\
Cor.~\ref{cor:hachenberger}, \ref{cor:pcn} & no computation of their own (they rest on Theorem~\ref{thm:W},
Corollary~\ref{cor:pcn} also on Proposition~\ref{prop:Msmall}); tested on $30$ small fields, see (C1) &
\texttt{bruteforce\_small\_}\allowbreak\texttt{fields.py} (2\,min)\\
\bottomrule
\end{tabular}
\caption{The computations and the programs that carry them out.}\label{tab:comp}
\end{table}
The program \texttt{struct\_and\_exactmax.py} also computes the exact maxima quoted for information in the proof of
Theorem~\ref{thm:W} and carries out the check~(C2) of Section~\ref{sec:checks}; the run times in Table~\ref{tab:comp} are those
of the whole programs. The program \texttt{examples.py} reproduces the numbers of Section~\ref{sec:4599} and of the examples $(2,21)$ and
$(3,10)$ of Section~\ref{sec:remarks}; the README of the programs describes the helper modules, and the README of the
formalisation lists the Lean files that redo each computation.
The exact values of $M$ and $S$ only need the list of essential types $(d,e)$ (Lemma~\ref{lem:weightcodim}); no component
computation is needed outside Theorem~\ref{thm:W}. Combining explicit
analytic estimates with large but finite verified computations is common in analytic number theory; see, for example,
\cite{helfgott2013numerical,helfgott2015ternary,platt2021rh}.

\subsection*{Formal verification}
Appendix~\ref{app:lean} describes the formalisation in Lean~4 with Mathlib of our proof of the case $n>q$ and of
Corollaries~\ref{cor:hachenberger} and~\ref{cor:pcn}, including every computer-assisted step; it gives the formal statements,
the trusted base and the differences from the proofs in the paper.

\subsection*{Code and data availability}
The Lean formalisation (built with Lean's build tool Lake \cite{leanref_lake}, with continuous integration via
\texttt{lean-action} \cite{leanaction}) is available at \url{https://github.com/florian-hellwig/morgan-mullen}; the
Python programs (those of Table~\ref{tab:comp} and \texttt{examples.py}), with the versions of the libraries used, and
the file of the $70$ certificates are in the folder \texttt{verification\_code} of this repository. The repository is archived at Zenodo,
\url{https://doi.org/10.5281/zenodo.23661993} (release \texttt{v1.0}, commit \texttt{09291f9dbd175ffdd8c22cddcbaff7152c516272}).

\section{An example: the pair \texorpdfstring{$(2,4599)$}{(2,4599)}}\label{sec:4599}

This section illustrates Sections~\ref{sec:product}--\ref{sec:main} with one pair; it is not part of the proof (which
settles this pair in Proposition~\ref{prop:finiteranges}, as $4599<16674$). For $(q,n)=(2,4599)$ the union bound gives
no information: the $M=1145$ essential conditions have total weight $\frac12+S\approx3.41$, where $\frac12$ is the weight of
the only condition of codimension $1$ (cf.\ Section~\ref{subsec:idea}). But they fall into $63$ components, none of
weight more than $\frac{33}{64}$. We list the types and the components, check the last two rows of the table by hand
and compare them with Section~\ref{sec:struct}, and evaluate the criteria of Theorems~\ref{thm:counting}(b)
and~\ref{thm:C}(b) and the error bound of Corollary~\ref{cor:pcn}(a).

Here $q=2$ and $n=4599=3^2\cdot7\cdot73$, so $a=0$, $n'=n$ and the coordinates are the resolvents
$F_\eta(w)=\sum_{t\in\Z/n}\eta^{-t}w^{2^t}$, $\eta\in\mu_{4599}$. The divisors $e$ of $n$ and their orders
$\ord_e(2)$ are
\[
\begin{array}{c|cccccccccccc}
e&1&3&7&9&21&63&73&219&511&657&1533&4599\\\hline
\ord_e(2)&1&2&3&6&6&6&9&18&9&18&18&18
\end{array}
\]
and the essential types $(d,e)$ of Lemma~\ref{lem:weightcodim} (here $d'=d$ and $n'=n$, so these are the pairs with
$de\mid n$ and $d\mid\ord_e(2)$), with their numbers $\nu_{d,\allowbreak e}=d\varphi(e)/\ord_e(2)$ as exponents, are
\[
\begin{array}{l}
(1,1)^1,\ (1,3)^1,\ (1,9)^1,\ (1,7)^2,\ (3,7)^6,\ (1,21)^2,\ (3,21)^6,\ (1,63)^6,\\
(1,73)^8,\ (3,73)^{24},\ (9,73)^{72},\ (1,219)^8,\ (3,219)^{24},\ (1,657)^{24},\\
(1,511)^{48},\ (3,511)^{144},\ (9,511)^{432},\ (1,1533)^{48},\ (3,1533)^{144},\ (1,4599)^{144},
\end{array}
\]
in total $M=1145$ essential conditions; $\mu_{4599}$ splits into $293$ $q$-orbits.
The conditions form $63$ components (Lemma~\ref{lem:qorbits}(a)). In the table, ``number'' is the number of
components of the shape given by the row, ``conditions'' the number of conditions of one of them, ``dim'' the total
size of the $q$-orbits it reads, i.e.\ the dimension of its coordinate space (Lemma~\ref{lem:coord}(c)), and
$(d,e)^m$ means that one component contains $m$ conditions of type $(d,e)$ (no exponent: one). So ``number'' times
$m$ is $\nu_{d,e}$, and the dimensions of all $63$ components add up to $n=4599$:
\begin{center}\small
\begin{tabular}{rrrll}
\toprule
number & conditions & dim & weight $w_{\cC}$ & types in one component\\
\midrule
1 & 1 & 1 & $1/2$ & $(1,1)$\\
1 & 1 & 2 & $1/4$ & $(1,3)$\\
1 & 1 & 6 & $1/64$ & $(1,9)$\\
2 & 5 & 9 & $33/64$ & $(1,7),(1,21),(3,7)^3$\\
2 & 6 & 18 & $3/32$ & $(1,63)^3,(3,21)^3$\\
8 & 20 & 81 & $6663/2^{18}$ & $(1,\allowbreak 73),\allowbreak (3,\allowbreak 73)^3,\allowbreak (9,\allowbreak 73)^9,\allowbreak (1,\allowbreak 219),\allowbreak (3,\allowbreak 219)^3,\allowbreak (1,\allowbreak 657)^3$\\
48 & 20 & 81 & $6663/2^{18}$ & $(1,\allowbreak 511),\allowbreak (3,\allowbreak 511)^3,\allowbreak (9,\allowbreak 511)^9,\allowbreak (1,\allowbreak 1533),\allowbreak (3,\allowbreak 1533)^3,\allowbreak (1,\allowbreak 4599)^3$\\
\bottomrule
\end{tabular}
\end{center}
The table is computed by the program \texttt{examples.py} (available with the programs of Section~\ref{sec:comp}),
which also gives $M$, $S$, the bounds for $\delta_{4599}$, $L$ and $\omega$ and the margins below.

The shape of the last row can be checked by hand. Let $\eta_0$ have order $511=2^9-1$; its $q$-orbit $O$
has $9$ elements. Since $n$ is odd and the level of an essential condition divides some $\ord_e(2)\in\{1,2,3,6,9,18\}$, all
essential conditions have level $1$, $3$ or $9$, so the fibre at level $d$ through $x$ is $x\mu_d\subseteq x\mu_9$ (Lemma~\ref{lem:qorbits}(b)). Hence
$Y=O\mu_9=\{\eta\varepsilon:\eta\in O,\ \varepsilon\in\mu_9\}$, a union of $q$-orbits, contains every fibre that meets
it; and $Y$ is the union of the fibres $\eta\mu_9$ ($\eta\in O$) of the level-$9$ conditions with targets $\eta^9$,
which meet $O$. By Lemma~\ref{lem:qorbits}(a), $Y$ is therefore the set of fibre elements of the component of the
level-one condition with target $\eta_0$. As $\xi\mapsto\xi^9$ is a bijection of $\mu_{511}$, the $9$ cosets
$\eta\mu_9$ are disjoint, and $Y$ consists of the $9$ elements of $O$, $18$ elements of order $1533$ (one $q$-orbit) and
$54$ of order $4599$ (three $q$-orbits). As $Y\cap\mu_{511}=O$, the $432/9=48$ $q$-orbits of elements of order $511$ give $48$
different components. The conditions of the component are: at level $1$ the five $q$-orbits of $Y$
($1+1+3$ conditions, codimensions $9,18,18$); at level $3$ the classes under $\xi\mapsto\xi^8$ of the targets $x^3$,
$x\in Y$, namely of the $9$ targets $\eta^3$, $\eta\in O$, of order $511$ ($3$ conditions, codimension $9$) and of
$18$ targets of order $1533$ ($3$ conditions, codimension $18$); at level $9$ the
$9$ targets $\eta^9$, $\eta\in O$, of order $511$ ($9$ conditions, codimension $9$). The weight is
$(1+3+9)2^{-9}+(1+3+3)2^{-18}=6663/2^{18}\approx0.0254$. The same argument with $73$, $219$, $657$ in place of $511$,
$1533$, $4599$ gives the sixth row ($72/9=8$ components).

These components illustrate Section~\ref{sec:struct}. Their minimal codimension is $o_0=9$, and
$K=K_0=9\cdot511=4599$ (also the choice~(B) of $K$) satisfies (i) and (ii) of Theorem~\ref{thm:struct} for every
condition $C_0$ of codimension $9$, as $d_0\mid9$, $e_0\mid511$ and $\ord_{4599}(2)=18$. So for \emph{every} $n$ the
fibres of the component of such a condition lie in $\mu_{4599}$, and its levels and codimensions divide $18$ (for
$n=4599$ itself, part~(a) says nothing, as $K=n$).
With $D=D_{\cC}=9$ both bounds of Theorem~\ref{thm:struct}(c) are attained for $o=9$ ($1+3+9=13$ conditions); for
$o=18$ the first is $81$ (the second $351$), against the actual $7$, and $54$ of the $81$ come from the levels $2$, $6$, $18$, which need $2\mid n$.
Accordingly $w_{\cC}=6663/2^{18}$ lies just below $\Wcl(2,9)=6737/2^{18}$ (Corollary~\ref{cor:weight}).

By Theorem~\ref{thm:product},
$\delta_{4599}\ge\tfrac12\cdot\tfrac34\cdot\tfrac{63}{64}\cdot\bigl(\tfrac{31}{64}\bigr)^2\bigl(\tfrac{29}{32}\bigr)^2
\bigl(1-\tfrac{6663}{2^{18}}\bigr)^{56}\approx2^{-5.893}$, whereas the union bound over all $1145$ conditions gives only
$\delta_{4599}\ge1-3.41<0$. Further, $L=\prod_C2(1-2^{-o(C)})\approx2^{1139.65}$ (the product over all $1145$ essential conditions);
Lemma~\ref{lem:omega} with the exact $t_k$ gives $\omega(2^{4599}-1)\le253$ (its second bound gives $390$); and the
sufficient condition of Theorem~\ref{thm:counting}(b), with the lower bound for $\delta_{4599}$ just given, holds with
a margin of about $901$ bits.
The cruder criterion $(*)$ of Theorem~\ref{thm:C}(b), by which Proposition~\ref{prop:finiteranges} settles this pair,
holds for $W=W(2)=\frac34$ (Theorem~\ref{thm:W}), $M=1145$ and $S=95281/2^{15}\approx2.908$ with a margin of about
$894$ bits ($757$ with $\omega\le390$, the bound used for this pair by the program for Proposition~\ref{prop:finiteranges}). In both forms the lower bound for
the density costs only a few bits (about $5.9$ in Theorem~\ref{thm:counting}(b), and $c(W)S\approx7.75$ in $(*)$);
the decisive comparison is $\omega(2^{4599}-1)+M\le1398$ against $\frac n2\log_2q=2299.5$. By the first inequality of
Corollary~\ref{cor:pcn}(a), with the bounds for $\delta_{4599}$, $L$ and $\omega$ above, $\PCN_2(4599)$ equals
$\theta\,\CN_2(4599)$ up to a relative error of less than $2^{-900}$ (the exponent is the margin found above for
Theorem~\ref{thm:counting}(b)).

\section{Further checks}\label{sec:checks}
The following checks are not used in the proof. A failure would reveal a misstated definition or hypothesis, a
wrong printed value or a programming error. The checks (C1), (C2) and (C4) are tests on examples.
\begin{enumerate}[label=(C\arabic*),leftmargin=3em]
\item \emph{Sections~\ref{sec:coord}--\ref{sec:product} against the definition} (\texttt{bruteforce\_small\_fields.py},
which also uses the Python library \texttt{galois} \cite{galois2020}).
The program uses $30$ selected pairs (listed in it) with $q^n\le65536$ and $\mu_{n'}\subseteq E$, so that $\Omega=E$
in Section~\ref{sec:coord}; they include pairs with $p\mid n$ and pairs that are not completely basic, such as
$(2,12)$, $(3,8)$ and $(5,6)$, the first and the last of which are not even regular. For each of these pairs and every
$w\in E$, complete normality is decided by the definition, tested as in the proof of
Proposition~\ref{prop:certificates}: $w$ is normal over $\F_{q^d}$ if and only if the $rn$ elements
$\gamma_d^{\,j}\sigma^{dt}(w)$ ($0\le j<rd$, $0\le t<n/d$) have $\F_p$-rank $rn$, where $r=[\F_q:\F_p]$ and $\gamma_d$
is a fixed generator of $\F_{q^d}^\times$ (in that proof $w$ is primitive and $\gamma_d=w^{(q^n-1)/(q^d-1)}$). The result
agrees with the test whether all essential fibre sums $C_{d,k}(w)$ are non-zero (Theorem~\ref{thm:essential}) and with
the test whether all fibre sums are non-zero (Theorem~\ref{thm:criterion}). Moreover, the program checks that
every kernel $V_{d,k}$ has exactly $q^{n-\lcm(d,o(d,k))}$ elements (Theorem~\ref{thm:criterion}); every essential
condition of codimension~$1$ forms a component by itself (Lemma~\ref{lem:qorbits}(c));
the identity $\CN_q(n)\cdot|E|^{s-1}=\prod_{\cC}\#\{w\in E:\ C(w)\ne0\text{ for all }C\in\cC\}$ holds, with the product
over the $s$ components, formed by
shared pair orbits as in Section~\ref{sec:product} (this is Theorem~\ref{thm:product}, multiplied by $|E|^s$); and
the bound $\CN_q(n)\ge(q-1)^{n'}q^{(p^a-1)n'}$ holds, with equality if and only if $n'\mid q-1$ (Corollary~\ref{cor:hachenberger}).
Finally, the primitive completely normal elements are counted (primitivity being decided by discrete logarithms), and
the two-sided estimate of Theorem~\ref{thm:counting}(a) and both inequalities of Corollary~\ref{cor:pcn}(a) are
checked.
\item \emph{Structure theorem} (\texttt{struct\_and\_exactmax.py}, part~(2)). For $q=2$, $3$, $4$, $5$, $7$, $8$, $9$,
$11$, $13$, $16$, $25$ and $27$, all $n\ge2$ up to a bound depending on $q$ (from $3000$ for $q=2$ to $300$ for $q\ge16$), and $40$ random $n$ between
this bound and $20$ times it, all components are computed, from the $q$-orbits as in Lemma~\ref{lem:qorbits}(a).
A component containing a condition of codimension~$1$ is checked to consist of this condition alone. For each of the
other components (about $1.1\cdot10^5$ in total), with $o_0$ its minimal codimension and $K$, $O_\infty$, $D_\infty$ as
in~(B), the program checks (a) and (b) of Theorem~\ref{thm:struct} and the first inequality of~(c) with $D=D_{\cC}$.
In~(b) it tests $d\mid\ord_K(q)$ and $D_{\cC}\mid(\ord_K(q))_{p'}$ as $d\mid O_\infty$ and $D_{\cC}\mid D_\infty$, which
is the same, since $O_\infty=\ord_K(q)$ and $D_\infty=(\ord_K(q))_{p'}$ (Section~\ref{sec:struct}, after~(B)). No
violation occurs.
\item \emph{Minimality of the thresholds} (\texttt{thresholds.py}). For $q\le7$ each threshold $n_*$ of
Proposition~\ref{prop:range} is the least $n$ for which $F_k(n)\ge0$ and
$a_kn\ln q\ge c(W)(\frac{\pi^2}6+\frac q{q-1})$ hold with the printed $k$, $\lambda$ and $W$; for $q\ge8$ (with
$\lambda=1$ and $W=W(q)$, and for $q\ge32$ with the values at $q=32$) each threshold is the least $n$ for which both hold for some $k$ with
$4\le k\le11$, and the printed $k$ attains it.
\item \emph{Certificates} (\texttt{pcn\_certificates.py -{}-selftest}). For $(q,n)=(2,6)$, $(3,4)$ and $(2,12)$ the order
test of the proof of Proposition~\ref{prop:certificates} is applied to every element of $E$, and the rank test to the
elements it accepts. The order test accepts
exactly $\varphi(q^n-1)$ elements (the program prints this number next to the count), and the rank test accepts $6$,
$16$ and $360$ of them; these are the values of $\PCN_2(6)$, $\PCN_3(4)$ and $\PCN_2(12)$ found in (C1) by a
different program.
\end{enumerate}

\section{Remarks}\label{sec:remarks}
\paragraph{What the proof uses.} Besides elementary algebra and classical facts about finite fields (the normal basis
theorem and Dedekind's lemma in Lemma~\ref{lem:classical}; Vinogradov's formula and the evaluation
$|G(\chi,\psi)|=q^{n/2}$ of Gauss sums \cite[Theorem~5.11]{lidl1997finite}, \cite{ireland1990classical} in
Lemma~\ref{lem:Gamma}), the proof uses two external results, namely the case $n\le q$ of the conjecture
\cite[Theorem~1.3]{garefalakis2019jpaa} and the bound $\log\lcm(1,\dots,\lfloor x\rfloor)<1.03883\,x$ ($x>0$) of Rosser and
Schoenfeld \cite[Theorem~12]{rosser1962approximate} (the left-hand side is Chebyshev's function). In addition, the proof
uses the computations of Section~\ref{sec:comp}. The computations are carried out in exact integer or
rational arithmetic, except that the Python programs use floating point (with $50$ digits in
Proposition~\ref{prop:range} and in double precision elsewhere) in three places: for the real-valued quantities in $(*)$
and in Proposition~\ref{prop:range} (logarithms, $\log_2C_k$, $c(W)$, $g_q$, $\pi^2/6$); for two checks in the program
for Theorem~\ref{thm:W} ($e^{1.03883}<2.83$ and the signs of two derivatives by which it verifies, for real
$y=\log_2o_0\ge12$, the variant $18y^2+5y+7<2^y$ of the inequality of step~(ii)); and for the second bound for $t_k$ in Lemma~\ref{lem:omega}. In the first two places the margins are many
orders of magnitude larger than the rounding errors; in the third, the bound is computed with a safety margin of
$10^{-9}$, so that the computed value is never smaller than the exact bound.

\paragraph{Sharpness.} The weight bounds are attained for $q=2,3$: for $q=2$ and $6\mid n$, the three essential conditions
with a target of order~$3$ (one of type $(1,3)$ and two of type $(2,3)$ in the notation of
Lemma~\ref{lem:weightcodim}) form a component of weight $3\cdot2^{-2}=\frac34$, and for $q=3$ and $8\mid n$ the two
conditions of type $(1,8)$ and the two of type $(2,4)$ form a component of weight $4\cdot3^{-2}=\frac49$.
The uniformity concerns the weights of the
components, not the density $\delta_n$, which for fixed $q$ is not bounded away from $0$: since completely normal
elements are normal, $\delta_n$ is at most the density of normal elements, and the latter is smaller than
$0.63/\sqrt{\log_qn}$ for infinitely many $n$ \cite{frandsen2000density}. In the other direction, $\delta_n$ is bounded below by a negative power of $n$: for fixed $q$,
Theorem~\ref{thm:C}(a) with $W=W(q)$, $M_1=\gcd(n',q-1)\le q-1$ (see the proof of Corollary~\ref{cor:hachenberger})
and Lemma~\ref{lem:analytic}(b) give
\[
\delta_n\ \ge\ \Bigl(1-\frac1q\Bigr)^{q-1}2^{-c(W(q))\,g_q(\log_qn)}\ =\ n^{-\pi^2c(W(q))/(6\log_2q)+o(1)}
\qquad(n\to\infty)
\]
(the exponent is about $-4.39$ for $q=2$; we have not tried to optimise it); for $p\nmid n$ and $n$ large, this is much
stronger than Corollary~\ref{cor:hachenberger}. The bound $M(q,n)\le n$ is attained, e.g.\ for $n\mid q-1$ (every
$\eta\in\mu_n$ is then fixed by $\sigma$, so only level $1$ is valid at $\eta$ and $\rho(\eta)=1$, whence $M(q,n)=n$ by
Lemma~\ref{lem:Mrho}; all $n$ conditions are codimension-one singletons) and, with $n>q$, for $n=q+1$ with
$q\equiv1\pmod4$.

\paragraph{Relation to earlier work.} Lemma~\ref{lem:classical} is classical: it is equivalent to the usual
criterion that $w$ is normal over $\F_Q$ if and only if $\gcd\bigl(x^m-1,\sum_{i<m}\sigma_Q^i(w)x^{m-1-i}\bigr)=1$
\cite[Theorem~2.39]{lidl1997finite}, rewritten in terms of the resolvents. For $m$ a power of $p$, where $\mu_{m'}=\{1\}$ and
$R_1(w)$ is the trace of $w$, it reduces to Perlis' criterion that $w$ is normal if and only if its trace is non-zero
\cite{perlis1942normal}. Resolvent criteria for normality are also used in characteristic~$0$ (see e.g.\
\cite[Proposition~2.3]{jung2012normal}), and for complete normality in \cite{koo2013,das2026pv}.

Hachenberger's theory of completely free elements
\cite{hachenberger1994,hachenberger1996survey,hachenberger1997book,hachenberger2001decomposition,hachenberger2013handbook,hachenberger2020topics}
decomposes $E$ as the direct sum of the cyclotomic modules $\ker\Phi_k(\sigma)^{p^a}$, $k\mid n'$; in the coordinates of
Section~\ref{sec:coord} the module for $k$ consists of the $w$ with $F_\eta(w)(u)=0$ for all $u$ whenever $\ord\eta\ne k$.
An element of such a module is a \emph{complete generator} if it generates the module simultaneously with respect to
all its module structures over intermediate fields \cite[Section~1]{hachenberger2001decomposition},
\cite[Section~5.1]{hachenberger2019computational}. The projections of a completely normal element to these modules are
complete generators of the modules \cite[Proposition~18.7 and the remark after Definition~18.4]{hachenberger1997book}
(see also \cite[Section~2]{hachenberger2019exceptional}, \cite[Section~5.2]{hachenberger2019computational}). The converse (every sum of
complete generators, one for each module, is completely normal) holds if and only if $\gcd(n',\ord_{\rad(n')}(q))=1$
\cite[Theorem~19.11]{hachenberger1997book}, \cite[Theorem~5.4.45]{hachenberger2013handbook} (see also
\cite[Theorem~4.2]{hachenberger2001regular}, \cite[Section~2]{hachenberger2019exceptional}), in particular for regular
pairs, but also, e.g., for the non-regular pair $(2,6)$. For regular pairs, $\CN_q(n)$ is therefore the product of
the numbers of complete generators of the modules. Moreover, complete generation of a non-exceptional module
$\ker\Phi_k(\sigma)^{p^a}$ reduces to generation over the single intermediate field $\F_{q^{\tau_k}}$, and that of an
\emph{exceptional} module \cite[Definition~20.2]{hachenberger1997book} to simultaneous generation over
$\F_{q^{\tau_k}}$ and $\F_{q^{2\tau_k}}$ \cite[Theorem~20.3]{hachenberger1997book},
\cite[Proposition~4.2]{hachenberger2019exceptional}; here $\tau_k$ is the largest integer whose square divides
$\ord_k(q)/\ord_{\rad(k)}(q)$ (the number $\tau(q,k)$ of \cite[(20.2)]{hachenberger1997book}, called the
\emph{central index} in \cite[Definition~4.1]{hachenberger2019exceptional}). The numbers of complete
generators are given explicitly for the non-exceptional modules \cite[Proposition~21.1]{hachenberger1997book}; for the
exceptional ones his monograph gives a lower bound \cite[Proposition~21.2]{hachenberger1997book}, shows in its proof that
this bound is exact if $p\nmid n$ (the case $\pi=1$ in his notation), and conjectures that it is always exact \cite[p.~125]{hachenberger1997book} (not to be confused with his conjectured lower bound,
Corollary~\ref{cor:hachenberger}). Hachenberger later proved this
conjecture and thereby obtained an explicit formula for $\CN_q(n)$ for every regular pair
\cite[Remark~6.1, Section~13 and Theorem~13.1]{hachenberger2019exceptional}. If
$\gcd(n',\ord_{\rad(n')}(q))=1$ (in particular for regular pairs), our components are compatible with this decomposition.
Indeed, if $(d,\xi)$ is essential, every prime $r$ dividing $d'$ divides $\ord\xi$: $r$ divides $n'$ and $\ord_{\ord\xi}(q)$,
every prime of $\ord_{\ord\xi}(q)$ divides $\ord\xi$ or $\ord_{\rad(n')}(q)$, and $r\nmid\ord_{\rad(n')}(q)$ by the
hypothesis. So all
elements of the fibre have order $d'\ord\xi$, and every component lies in a single cyclotomic module. Otherwise
the fibre of a single condition can contain elements of different orders; for $(q,n)=(3,10)$, each of the two essential
conditions at level~$2$ with a target of order~$5$ has a fibre consisting of one element of order~$5$ and one of
order~$10$, so it reads coordinates of both $\ker\Phi_5(\sigma)$ and
$\ker\Phi_{10}(\sigma)$. 

For arbitrary pairs, Hachenberger's Complete Decomposition Theorem
\cite[Sections~2 and~5]{hachenberger2001decomposition} (its sufficiency part is the Decomposition Theorem of
\cite[Section~19]{hachenberger1997book}) decides, for a
generalised cyclotomic module $\ker\Phi_k(\sigma^t)$ (a cyclotomic module in the terminology of
\cite{hachenberger2001decomposition}) and a prime $r\ne p$ dividing $t$ but not $k$, when its splitting into
$\ker\Phi_k(\sigma^{t/r})$ and $\ker\Phi_{kr}(\sigma^{t/r})$ is \emph{agreeable} \cite[Definition~19.1]{hachenberger1997book},
\cite[Definition~2.2]{hachenberger2001decomposition}, i.e.\ every sum of complete generators of the two parts is a
complete generator of the module. Splitting repeatedly in this way, starting from $E=\ker(\sigma^n-1)$ and as long as
some splitting is agreeable, decomposes $E$ into generalised cyclotomic modules such that every sum of complete
generators, one for each summand, is completely normal. By his Uniqueness Theorem the result does not depend on the
order of the steps; following him, we call it the finest agreeable decomposition
\cite[Section~6]{hachenberger2001decomposition}. The number $\CN_q(n)$ is the product of the numbers of complete
generators of its summands \cite[Section~2]{hachenberger2001decomposition},
\cite[Section~5.3, (5.2)]{hachenberger2019computational}. For summands that are not regular in the sense of
\cite[Definition~20.1]{hachenberger1997book}, such as $\ker\Phi_7(\sigma^3)$ for $(2,21)$ below, these numbers can be
computed by the enumeration of \cite[Section~5.3]{hachenberger2019computational}, but we are not aware of a closed
formula for them. Hachenberger and
Hackenberg also showed
that complete normality needs to be tested only over the fields $\F_{q^d}$ with $d$ in their $(q,n)$-essential set
$\mathcal D^*$ \cite[Definitions~3.2.1 and~3.2.3, Proposition~3.2.4]{hachenberger2019computational}, whose construction rests on
\cite[Lemma~15.3 and Theorem~15.5]{hachenberger1997book};
Theorem~\ref{thm:essential} refines this reduction to individual targets (see the paragraph after its proof).

Theorems~\ref{thm:criterion}
and~\ref{thm:product} are closely related to this decomposition theory, but are stated in terms of fibre sums of
partial Lagrange resolvents. The two factorisations of $\CN_q(n)$ need not coincide. In the following example ours is finer, and
Hachenberger's own computation \cite[Example~14.6]{hachenberger1997book} uses the same finer splitting of
$\ker\Phi_7(\sigma^3)$. The pair
$(2,21)$ is not regular
($\ord_{21}(2)=6$), and the summand $\ker\Phi_7(\sigma^3)=\ker(\Phi_7\Phi_{21})(\sigma)$ of its finest agreeable
decomposition, of dimension $18$, splits into two components of dimension $9$ (the total size of the Frobenius orbits
they read, as in Section~\ref{sec:4599}). Each of them reads one orbit of elements of order~$7$ and one of order~$21$, so
the elements of $E$ whose coordinates outside the component vanish form $\ker g(\sigma)h(\sigma)$ for an irreducible
cubic factor $g$ of $\Phi_7$ and an irreducible sextic factor $h$ of $\Phi_{21}$ over $\F_2$; this kernel is not a generalised
cyclotomic module. For each of the two components, exactly $294$ of its $2^9$ possible coordinate vectors satisfy all its
conditions, and Theorem~\ref{thm:product} gives $\CN_2(21)=1\cdot3\cdot294^2=259308$, where $1$ and $3$ are the corresponding numbers for the two remaining components (of dimensions $1$ and $2$, with one condition each). In
\cite[Example~14.6]{hachenberger1997book} the two
$9$-dimensional pieces appear as the $\F_8$-modules $U_{K,f}$, $f\in\{x^3+x+1,x^3+x^2+1\}$, treated by simultaneous
generation over $\F_2$ and $\F_8$ \cite[Theorem~12.2]{hachenberger1997book}. (These numbers, and those for $(3,10)$
above, are reproduced by the program \texttt{examples.py} of Section~\ref{sec:comp}.)

Conversely, for regular pairs his factorisation can be finer than ours. For a non-exceptional module
$\ker\Phi_k(\sigma)^{p^a}$, the number of complete generators given by \cite[Proposition~21.1]{hachenberger1997book} and
\cite[Theorem~13.1]{hachenberger2019exceptional} is the product, over the irreducible factors $f$ of $\Phi_{k/\tau_k}$
over $\F_{q^{\tau_k}}$, of the number of generators over $\F_{q^{\tau_k}}$ of the subspace $\ker f(\sigma^{\tau_k})^{p^a}$;
each such subspace lies in the space $\ker g(\sigma)^{p^a}$ of the elements whose coordinates outside a single
component vanish, where $g$ is the irreducible factor of $\Phi_k$ over $\F_q$ divisible by $f(x^{\tau_k})$, and this
space is the direct sum of $\tau_k$ such subspaces. For $(2,27)$, for instance, where $\tau_{27}=3$, the elements of
order~$27$ form one $q$-orbit and hence one component, of dimension $18$, whereas his count of the complete generators
of $\ker\Phi_{27}(\sigma)$ is $63^3=250047$, with one factor for each of the three quadratic factors of $\Phi_9$ over
$\F_8$; altogether $\CN_2(27)=3\cdot63^4=47258883$. We have not compared the two decompositions
systematically.

As far as we know, what is new here is the description of complete normality over finite fields and of its
independent pieces by fibre sums of partial Lagrange resolvents, and above all the structure theorem (Theorem~\ref{thm:struct}) with the resulting
bound $W(q)$ for the weight of a component, uniform in $n$ (Theorem~\ref{thm:W}). This bound lets the character-sum
argument work uniformly in $n$, for regular and non-regular pairs alike; it also yields
Corollaries~\ref{cor:hachenberger} and~\ref{cor:pcn}. Earlier
approaches to
the conjecture exploit the module structure of $E$ for regular pairs
\cite{hachenberger2001regular,hachenberger2010twopower,hachenberger2019exceptional}, bound the number of completely
normal elements from below by means of a sum over all intermediate fields
\cite{hachenberger2016asymptotic,garefalakis2019jpaa,garefalakis2019dcc} or over an essential set of them
\cite[Proposition~3.2.5]{hachenberger2019computational}, or use translates, by elements of $\F_q$, of
products of completely normal elements of subextensions of coprime degrees \cite{garefalakis2026cmb} (such products are
completely normal by the Reduction Theorem \cite[Section~4]{hachenberger1997book}, which goes back to Blessenohl and
Johnsen \cite{blessenohl1986}; an elementary proof is given in \cite{hachenberger1994}).

\paragraph{Further questions.} \emph{Sieving.} Below the thresholds of Proposition~\ref{prop:range}, the criterion
$(*)$ is checked pair by pair, and the $70$ pairs at which this check leaves less than one bit to spare are settled by
explicit elements (Propositions~\ref{prop:finiteranges} and~\ref{prop:certificates}). In $(*)$, every prime divisor of
$q^n-1$ and every essential condition enters the error term at full cost, through the factors $2^{\omega(q^n-1)}$ and
$L\le2^M$. For the primitive normal basis theorem \cite{lenstra1987primitive}, Cohen and Huczynska gave a proof without
a computer by means of a sieve on the additive side \cite{cohen2003pnbt}. In that proof, every prime divisor of
$q^n-1$ still enters at full cost, but of the irreducible factors of $x^n-1$ only those in a chosen core do; the others
enter only additively, as long as the sum of $q^{-\deg g}$ over these sieved factors $g$ is less than~$1$ (for a sieve
on both sides, see \cite[Section~4]{cohen2010strong}). For complete normality, Garefalakis and Kapetanakis remark that
in their approach such a sieve could be applied only on the multiplicative side \cite[Section~3]{garefalakis2019jpaa},
and they use it on that side in \cite[Section~5.1]{garefalakis2019dcc}. Since distinct components are independent
(Theorem~\ref{thm:product}), whole components of small weight might take the place of the irreducible factors of
$x^n-1$ on the additive side. How far can such a sieve lower the thresholds? Can it settle every pair with $n>q$ by the
criterion alone, without checking pairs individually, and also the pairs with $n\le q$ that Garefalakis and
Kapetanakis treat by a computer calculation or by the examples of Morgan and Mullen \cite{garefalakis2019jpaa}?

\emph{Closed formulae.} Theorem~\ref{thm:product} writes $\CN_q(n)$ as $q^n$ times the product of the
probabilities $\Pr[G_{\cC}]$ that a random element passes the individual components, and Corollary~\ref{cor:hachenberger}
uses only lower bounds for these factors. For a given pair, each factor can be computed by enumerating the coordinate
vectors of its component, as for $(2,21)$ above; what we ask for are closed formulae that hold for all pairs. They
would extend Hachenberger's formula for regular pairs \cite[Theorem~13.1]{hachenberger2019exceptional} (for the
non-exceptional modules already in \cite[Proposition~21.1]{hachenberger1997book}, and for $p\nmid n$ obtainable from
\cite[Theorem~19.10 and Section~21]{hachenberger1997book}), and by Corollary~\ref{cor:pcn} they would make the main term $\theta\,\CN_q(n)$ of the
asymptotic formula for $\PCN_q(n)$ explicit.

\emph{Small fields.} For small fields the error bounds of Corollary~\ref{cor:pcn}(a) can be trivial (for $(2,8)$
both equal $\theta\,\CN_2(8)$), and the deviation is not negligible: $\PCN_2(8)=56$ (counted in (C1) of
Section~\ref{sec:checks}), whereas $\CN_2(8)=128$ by the equality case of Corollary~\ref{cor:hachenberger}, so that
$\theta\,\CN_2(8)=128^2/255\approx64.25$. In particular, the analogue $\PCN_q(n)\ge\varphi(q^n-1)(1-\frac1q)^{n'}$ of
Corollary~\ref{cor:hachenberger} fails for $(2,8)$, as $\varphi(255)/2=64>56$. A natural question is how large the
relative deviation $|\PCN_q(n)/(\theta\,\CN_q(n))-1|$ can be for small $q^n$, and for which pairs this analogue
holds. Exact values of $\PCN_q(n)$ for small pairs are given in \cite{morgan1996} and, together with those of
$\CN_q(n)$, in \cite[Section~1, Computational Result~3]{hachenberger2019computational}; for many extensions of
prime-power degree, lower bounds for $\PCN_q(n)$ are given in \cite{hachenberger2010twopower}.

\emph{Refinements.} In analogy with the prescribed-trace versions of the primitive normal basis theorem
\cite{cohen1999prescribed,mazumder2026traces}, one may ask whether the component structure can be combined with trace
conditions, so as to obtain primitive completely normal elements with suitably prescribed traces; for many extensions
of prime-power degree, Hachenberger \cite[Definition~12, Theorem~13 and Sections~3--4]{hachenberger2010twopower} shows
that every completely normal element of a suitable intermediate field is the trace of a primitive completely normal
element. Likewise, the strong
primitive normal basis theorem \cite{cohen2010strong} suggests asking for which pairs $(q,n)$ there is a primitive
element $w$ such that both $w$ and $w^{-1}$ are completely normal.

\section*{Acknowledgements}
This work was carried out during an exchange semester at Nanyang Technological University (NTU), Singapore; NTU
kindly provided hospitality and access to its library. I am grateful to Anna-Lena Horlemann (University of St.\ Gallen)
for her encouraging words and for her helpful comments on the presentation of an earlier draft. I thank
Gary Greaves (NTU) for many conversations about the use of AI in mathematical proofs, for valuable comments on the
manuscript and for his support. My thanks also go to Silvio Takafumi Zimmermann (University of Zurich) for reading a draft
and for his critical questions about it, and to Joachim Rosenthal (University of Zurich) for his steady support throughout my
studies. The formalisation builds on Mathlib, and I thank its community. The exchange was supported by a
Swiss-European Mobility Programme (SEMP) scholarship awarded by the University of Zurich and funded by the Swiss
Confederation through Movetia. Any remaining errors are my own.

\section*{Use of AI tools}
The research questions, the overall direction and all final decisions of this work are the author's; the arguments
were generated by AI language models (Claude Fable 5.1 and Claude Opus 5.5, Anthropic) under the author's direction.
In particular, the author posed the questions that led to extending the known cases of the conjecture to further pairs
$(q,n)$, and directed the work on the pair $(2,4599)$ (Section~\ref{sec:4599}), which the first approaches could not
settle, and on generalising the argument found for it. The text of this paper was drafted by the same AI models under
the author's direction and heavily revised by the author. The author has checked every step of the final proof by hand
and has verified the main citations against the original sources. The author takes full responsibility for the
correctness of the results.

The Python programs (Section~\ref{sec:comp}) were written by AI under the author's direction; the author has run them
and confirmed their output, but has not reviewed their code line by line. In addition, at the author's request, AI models reviewed the mathematics and reproduced all
computations with separately written code, in many separate sessions. The Lean~4 formalisation was
generated entirely by AI\@. The author has built it on his own computer and run its audit files, confirming that it
relies only on the axioms and the trusted base set out in Appendix~\ref{app:lean}. He has also verified that the formal
statements of the main theorem and of Corollaries~\ref{cor:hachenberger} and~\ref{cor:pcn}, together with the
definitions they use, faithfully capture these results and that the two cited results used as axioms are stated
correctly; he has not reviewed the rest of the formalisation.
\appendix

\section{The Lean formalisation}\label{app:lean}
This appendix describes the formalisation mentioned in Sections~\ref{sec:intro} and~\ref{sec:comp}. It is not needed
for reading the rest of the paper.

\subsection*{What is formalised}
Our proof of the case $n>q$ and of Corollaries~\ref{cor:hachenberger} and~\ref{cor:pcn}, including every
computer-assisted step, has been formalised in Lean~4 \cite{demoura2021lean4} (v4.34.1) with Mathlib
\cite{mathlib2020,mathlib4repo}; the documentation of the formalisation uses the statement numbers of this paper (a few
file, namespace, section and command names keep those of an earlier draft; the README of the
repository lists them). The main theorem is
\begin{center}\small
\begin{tabular}{@{}l@{}}
\texttt{theorem morgan\_mullen (F E : Type*) [Field F] [Field E] [Finite E] [Algebra F E] :}\\
\texttt{\ \ \ \ HasMMProperty F E}
\end{tabular}
\end{center}
where \texttt{HasMMProperty F E} says that $E$ contains a generator $w$ of $E^\times$ that is normal over every intermediate
field $K$ of $E/F$, normality over $K$ meaning that the conjugates $gw$, $g\in\mathrm{Gal}(E/K)$, form a $K$-basis of $E$.
Among the results taken from Mathlib are the normal basis theorem (\texttt{Mathlib/FieldTheory/Galois/NormalBasis.lean},
by Crim, Liu, Springer and Xu), Dedekind's lemma and the identity $G(\chi,\psi_b)\,G(\chi^{-1},\psi_{-b})=q^n$
($\chi\ne1$, $b\ne0$) for the Gauss sums of Lemma~\ref{lem:Gamma}. Corollary~\ref{cor:hachenberger} is formalised as
the theorems \texttt{hachenberger\_lower\_bound} and \texttt{hachenberger\_eq\_iff} (file \texttt{Sec8Corollary.lean}), in
which $\CN_q(n)$ is the number of elements of $E$ that are normal over every intermediate field.
Corollary~\ref{cor:pcn} is formalised in the file \texttt{Sec8PCN.lean}, with $\PCN_q(n)$ the number of these elements
that generate $E^\times$. Part~(a) is \texttt{pcn\_theta\_cn}, whose two inequalities are the theorems
\texttt{pcn\_\allowbreak sub\_\allowbreak theta\_\allowbreak cn\_\allowbreak abs\_\allowbreak le} and
\texttt{pcn\_\allowbreak error\_\allowbreak le\_\allowbreak star}; part~(b), in the form given there with
$|E|=q^n\ge X$, is \texttt{pcn\_\allowbreak ratio\_\allowbreak tendsto\_\allowbreak one}. In the formal statement of part~(a),
$M$, $S$ and $L$ are defined through a combinatorial model of the essential conditions (file \texttt{Model.lean});
checking this statement therefore also involves these definitions.

\subsection*{Axioms and trusted base}
The formalisation contains no \texttt{sorry} (Lean's placeholder for an omitted proof). The command
\texttt{\#print axioms} (reference output in the file \texttt{AXIOMS\_AUDIT.txt} of the repository) shows that, besides
the three axioms on which all of Mathlib rests (\texttt{propext}, \texttt{Classical.choice}, \texttt{Quot.sound}), the
main theorem depends on only four further axioms. Two of them are the cited results: \texttt{garefalakis\_kapetanakis}
(the case $n\le q$ \cite[Theorem~1.3]{garefalakis2019jpaa}) and \texttt{rosser\_schoenfeld\_psi} (the bound
$\log\lcm(1,\dots,\lfloor x\rfloor)<1.03883\,x$ for $x>0$ of Rosser and Schoenfeld \cite[Theorem~12]{rosser1962approximate}, whose
left-hand side is Chebyshev's function). The other two, \texttt{Lean.ofReduceBool} and \texttt{Lean.trustCompiler}
\cite{leancore_reducebool,leanref_axioms}, enter because the large finite computations are carried out by $25$ Boolean
checkers that are evaluated outside Lean's kernel (its proof checker). These checkers are translated by Lean's compiler
and run by Lean's interpreter, which calls natively compiled library code such as arbitrary-precision arithmetic. The
theorems stating that they are sound (a run returning \texttt{true} implies the mathematical statement), by contrast,
are checked by the kernel. Smaller computations, for instance for $\kappa_3$ and
$\kappa_4$ in Proposition~\ref{prop:Msmall}, are checked by the kernel directly. This is the familiar method of proof by
reflection \cite{boutin1997reflection,gregoire2005ring}, with
evaluation by compiled code as in \cite{gregoire2002compiled,boespflug2011full}; like Lean's \texttt{native\_decide},
which is not used here, it trusts Lean's compiler and interpreter \cite{ullrich2019beans} together with the native
implementations of the library functions used. (In Lean~v4.34.1, \texttt{Lean.ofReduceBool} and \texttt{Lean.trustCompiler} are
marked as deprecated in favour of the mechanism that \texttt{native\_decide} now uses. That mechanism records each
native evaluation of a Boolean term $c$ as an auxiliary axiom $c=\texttt{true}$ and relies on the same compiler and
interpreter. The toolchain
is pinned, so the deprecation does not affect the formalisation.) The audit file \texttt{AuditDeep.lean}, run by the
continuous integration on every push, lists the $25$ checkers and verifies that no constant of this project used by them
carries \texttt{implemented\_by} or \texttt{extern} (attributes that let the executed code differ from the definition to
which the soundness proofs refer). The theorems for Corollary~\ref{cor:hachenberger}, for Corollary~\ref{cor:pcn}(a) and its
second inequality, and for Corollary~\ref{cor:pcn}(b) depend on the same axioms as the main theorem except
\texttt{garefalakis\_kapetanakis}; the first inequality of Corollary~\ref{cor:pcn}(a) depends only on the three axioms
of Mathlib. The proofs of Corollary~\ref{cor:hachenberger} and of
the second inequality of Corollary~\ref{cor:pcn}(a) evaluate only the $6$ checkers for Theorem~\ref{thm:W}, and that of
Corollary~\ref{cor:pcn}(b) in addition the $5$ checkers for Proposition~\ref{prop:Msmall}.

Data produced
outside Lean (factorisations; the numbers $\hat O$, $\hat D$, $\hat M$ certifying multiples of $O_\infty$ and $D_\infty$
for Theorem~\ref{thm:W}, described below; the $70$ certificates of Proposition~\ref{prop:certificates}) and the output
of unverified Lean routines (the labellings for step~(i) of the proof of Theorem~\ref{thm:W}, the inverse matrices in
Proposition~\ref{prop:certificates}) are only checked, never trusted. In particular, every prime factor used in
Proposition~\ref{prop:certificates} is certified by a Pratt certificate \cite{pratt1975every}, whose check rests on
Mathlib's formalisation of Lucas' test (for the closely related Pocklington certificates in proof assistants, see
\cite{caprotti2001primality,gregoire2006pocklington}), and the prime factors
of $q^o-1$ used in Proposition~\ref{prop:Msmall} (at most $13$ digits) are checked by trial division. In sum, one has to
trust Lean's kernel, the axioms named above (in particular the two cited results and their formal statements), Lean's
compiler and interpreter for the $25$ checkers (with the native library code they call), and the formal statement of the main theorem, which uses four short
definitions (file \texttt{Basic.lean}). Neither the Python programs nor any data are trusted, and the correctness of
the formal proof does not depend on how its Lean code was written (see the statement on the use of AI tools before this appendix).
External type checkers such as Lean4Lean \cite{carneiro2024lean4lean} do not implement \texttt{Lean.reduceBool}; they
could re-check the rest of the formal proof only if the $25$ evaluations were recorded as auxiliary axioms instead (the
mechanism mentioned above).

\subsection*{Differences from the proofs in the paper}
The formalisation follows the proofs in the paper, with the same constants, tables, ranges and thresholds, except
where noted below. One argument
takes a different route. Lemma~\ref{lem:Gamma} is proved without the explicit coefficients of Vinogradov's formula: after
identifying $E^\times$ with $\Z/(q^n-1)$ by a discrete logarithm, the indicator function of $P$ becomes
$\prod_{\ell\mid q^n-1}(1-1_{\ell\Z/(q^n-1)})$, and the $\ell^1$-argument of the proof of Theorem~\ref{thm:counting} is
carried out on $\Z/(q^n-1)$, whose additive characters correspond to the multiplicative characters of $E$. The coefficients of this character expansion
then have $\ell^1$-norm at most $\prod_{\ell\mid q^n-1}2(1-\frac1\ell)=2^\omega\theta$, the trivial character has coefficient
$\theta$, and the other characters give Gauss sums of absolute value $q^{n/2}$. Some elementary steps in
Lemma~\ref{lem:coord}(a), Theorem~\ref{thm:criterion}, the paragraph after Definition~\ref{def:cond},
Theorem~\ref{thm:product} and Corollary~\ref{cor:pcn}(b) are arranged differently; see the README of the
formalisation. For instance, the identity
$R_{\hat\xi}=\frac1{d'}C_{d,k}$ is obtained by summing a geometric series over the fibre, and in Corollary~\ref{cor:pcn}(b)
the exponent $\omega+M+c(W)S-\frac n2\log_2q$ is bounded directly, without the margin $\Delta$. The proof of Corollary~\ref{cor:hachenberger} differs in three elementary steps: it uses only
$M_1\le\#\{\eta\in\mu_{n'}:o(\eta)=1\}$, which suffices; it proves $(\dagger)$ for $q\ge3$ in the form
$(1-W)^{s(\eta)/W}\ge1-\frac1q$ by Bernoulli's inequality; and in the case of equality with $q=2$ it obtains
strictness from the strict union bound in a component that contains two distinct conditions, without determining $n'$.
Finally,
Proposition~\ref{prop:Mlocal} is formalised for rational $\vartheta$, $\Theta$, which suffices for its
applications.

\subsection*{Computations and definitions handled differently}
Where the Python programs and the formalisation carry out a computation differently, or the formalisation uses a
different definition, the differences are as follows.
\begin{itemize}[itemsep=1pt]
\item Proposition~\ref{prop:Msmall}: the formalisation checks factorisations of $q^o-1$ computed outside Lean (the
product is $q^o-1$, and each factor is prime by trial division) instead of factoring with \texttt{sympy}.
\item Theorem~\ref{thm:struct}: the formalisation sets $\ord_K(q)=0$ if $\gcd(q,K)>1$ and does not assume that $K$ is
$p$-free, since with this convention condition~(ii) forces $p\nmid K$ unless $K=0$.
\item Closure~(B) in Section~\ref{sec:struct}: the formalisation defines $D_\infty$ as
$D_i$ with \mbox{$i=2(\lfloor\log_2D_0\rfloor+1)^2+1$}; the sequence is constant from this index on.
\item Theorem~\ref{thm:W}: in the evaluation of $\Wcl$, the formalisation replaces $O_\infty$ and $D_\infty$ by multiples.
This is allowed because $\Sigma(O,D)$ increases when $O$ and $D$ are replaced by multiples (paragraph \emph{Evaluation
of $\Wcl$ and $\Wcr$} in the proof). It is done because multiples can be certified by congruences alone, without
computing any order (the Python
program computes $O_\infty$ and $D_\infty$ directly): if $o_0\mid\hat O$, $q^{\hat O}\equiv1\pmod{K_0}$,
$\hat O_{p'}\mid\hat D$,
$q^{\hat M}\equiv1\pmod{\hat D}$, $\hat M_{p'}\mid\hat D$ and $\hat M\mid\hat O$, then $D_\infty\mid\hat D$,
$O_\infty\mid\hat O$ and $N^{**}\mid\lcm(K_0,\hat D,\hat O)$. Indeed $\ord_{K_0}(q)\mid\hat O$ gives $O_0\mid\hat O$ (every term of
$O_0$ other than $o_0$ divides $\ord_{K_0}(q)$), hence $D_0\mid\hat D$; if $D_i\mid\hat D$ then $\ord_{D_i}(q)\mid \ord_{\hat D}(q)\mid\hat M$, so
$D_{i+1}\mid\hat D$; and $O_\infty=\lcm(O_0,\ord_{D_\infty}(q))\mid\lcm(\hat O,\hat M)=\hat O$.
In step~(i) this also certifies that $N^{**}$ divides $42$, $378$, $16$ and $320$ for the four pairs $(q,o_0)$ of that
step. For the labellings of
step~(i), the formalisation verifies the two defining properties for a labelling computed by an unverified routine,
while the Python program uses the components themselves. The sharper bounds $0.032$ ($q=2$, $o_0\ge9$) and $0.051$
($q=3$, $o_0\ge7$) noted at the end of the proof are the bounds proved in the formalisation.
\item Proposition~\ref{prop:range}: the formalisation evaluates the inequalities with proved rational bounds (the Python
program with $50$ digits); for $q=17$ it uses $k=7$, which gives the same threshold $29$.
\item Proposition~\ref{prop:finiteranges}: for $q=3$ the formalisation checks exactly the pairs of the proposition (the
Python program all $n<131200$). It computes $t_k$, when $\varphi(k)\lfloor\log_2q\rfloor\le3000$, by a sieve (proving
only that the computed value is at least $t_k$), and otherwise the second bound of
Lemma~\ref{lem:omega} with $\log_2q$ and $c_q$ replaced by proved rational upper bounds and $\log_2(\lcm(2,k)+1)$ by a
proved rational lower bound. It verifies $(*)$ with the bound of Lemma~\ref{lem:omega} (with these values for $t_k$)
plus $1$ in place of $\omega(q^n-1)$, i.e.\ with the margin of one bit required in
Proposition~\ref{prop:finiteranges}, in exact rational arithmetic: each $q^{-o}$
in $S$ is replaced by an upper bound that is a multiple of $2^{-64}$, $\log_2q$ and $c(W)$ are replaced by proved
rational bounds, and $\log_2(q^n-1)\ge n\log_2q-\frac1{1000}$ is used, which holds since $q^n\ge2048$ for every accepted
pair.
\item Proposition~\ref{prop:certificates}: in the formalisation $f$ is monic of degree $m$ by construction, and the list
of primes is computed outside Lean and only checked. Rank $m$ is certified by checking an inverse matrix computed by an
unverified routine; in this check,
$\gamma$ is computed as $\prod_{0\le t<n/d}\sigma^{dt}(w)$, which is the same element, and only $\gamma^{q^d}=\gamma$ is
checked, which suffices for the implication from rank $m$ to normality.
\item The formalisation uses no floating point (Section~\ref{sec:remarks}, \emph{What the proof uses}): all such
steps use proved rational bounds instead.
\end{itemize}

\subsection*{Statements that are not formalised}
Statements given only for illustration or information are not formalised. These include the values quoted for information in the
proof of Theorem~\ref{thm:W} (among them the exact values of $N^{**}$ and $\Wcl$ and the exact maxima), the
identification of the closure~(B) with the least admissible $K$ in Section~\ref{sec:struct}, the margins quoted after
the proof of Proposition~\ref{prop:finiteranges}, the note after the proof of Proposition~\ref{prop:certificates} on the choice of the
$70$ pairs, the remarks following the proof of Corollary~\ref{cor:pcn}, Remark~\ref{rem:pcnphi}, the paragraph \emph{Sharpness} in Section~\ref{sec:remarks} (attainment of $W(2)$, $W(3)$ and
of $M(q,n)\le n$, and the lower bound for $\delta_n$), the comparisons with earlier work (after the proof of
Theorem~\ref{thm:essential}, in Section~\ref{sec:counting}, in the remark after Theorem~\ref{thm:main} and in
Section~\ref{sec:remarks}), the example $(2,8)$ in Section~\ref{sec:remarks}, the worked examples (in Sections~\ref{subsec:idea}, \ref{sec:coord}, \ref{sec:product}, \ref{sec:bounds},
\ref{sec:4599} and~\ref{sec:remarks}) and the checks (C1)--(C4) of Section~\ref{sec:checks}.

\bibliographystyle{amsplain}
{\raggedright
\bibliography{references}}

\end{document}